\documentclass[9pt]{amsart}

\usepackage{amsthm, amssymb, amsmath}
\usepackage{mathtools}
\usepackage{enumitem}
\usepackage{dsfont, mathrsfs}
\usepackage{tikz-cd}
\usepackage{hyperref}

\newtheorem{thm}{Theorem}[section]
\newtheorem{prop}[thm]{Proposition}
\newtheorem{lem}[thm]{Lemma}
\newtheorem{cor}[thm]{Corollary}

\theoremstyle{definition}

\newtheorem{defn}[thm]{Definition}
\newtheorem{example}[thm]{Example}
\newtheorem{remark}[thm]{Remark}
\newtheorem{constr}[thm]{Construction}
\newtheorem{notation}[thm]{Notation}
\newtheorem{question}[thm]{Question}

\newtheorem{mainthm}{Theorem}

\def\Q{\mathbb{Q}}
\def\Z{\mathbb{Z}}
\def\CC{\mathscr{C}}
\def\DD{\mathscr{D}}
\def\AA{\mathscr{A}}
\def\E{\mathbb{E}}
\def\F{\mathbb{F}}

\def\S{\mathbb{S}}

\def\LL{\mathscr{L}}
\def\PP{\mathscr{P}}
\def\P{\mathbb{P}}
\def\OO{\mathscr{O}}
\def\e{\varepsilon}
\def\one{\mathds 1}

\def\En{\mathbb{E}_n}

\def\CAlg{\operatorname{CAlg}}
\def\Map{\operatorname{Map}}
\def\Hom{\operatorname{Hom}}
\def\hom{\operatorname{hom}}
\def\Fun{\operatorname{Fun}}
\def\End{\operatorname{End}}
\def\Mod{\operatorname{Mod}}

\def\id{\operatorname{id}}
\def\Sp{\operatorname{Sp}}
\def\fib{\operatorname{fib}}
\def\cof{\operatorname{cof}}
\def\colim{\operatorname{colim}}

\def\Supp{\operatorname{Supp}}
\def\into{\hookrightarrow}
\def\Idem{\operatorname{Idem}}
\def\Alg{\operatorname{Alg}}
\def\Free{\operatorname{Free}}

\def\oblv{\operatorname{Oblv}}
\def\Oblv{\operatorname{Oblv}}
\def\op{\operatorname{op}}
\def\Loc{\operatorname{Loc}}
\def\Spc{\operatorname{Spc}}
\def\ev{\operatorname{ev}}
\def\coev{\operatorname{coev}}

\def\Spec{\operatorname{Spec}}
\def\Perf{\operatorname{Perf}}

\makeatletter
\renewcommand{\eqref}[1]{\textup{\hyperref[#1]{\tagform@{\ref*{#1}}}}}
\makeatother

\pdfinterwordspaceon

\usepackage[backend=biber,backref=true,maxbibnames=4,isbn=false]{biblatex}
\DefineBibliographyStrings{english}{
  backrefpage={},
  backrefpages={}
}

\usepackage{xpatch}
\DeclareFieldFormat{backrefparens}{
  \addperiod
  \delimcontext{backref}
  \let\backreflink\hyperlink
  \renewcommand*{\hyperlink}[2]{\raisebox{4pt}{\scriptsize\backreflink{##1}{##2}}}
  
  #1}
\DeclareDelimFormat[backref]{multilistdelim}{\raisebox{4pt}{\scriptsize,}{\scriptsize\space}}
\renewbibmacro*{pageref}{
  \iflistundef{pageref}
    {}
    {\printtext[backrefparens]{
       \addspace
       \printlist[pageref][-\value{listtotal}]{pageref}}}}

\title[Strong finiteness for localisations]{A strong finiteness condition for smashing localisations}
\author{Isabel Longbottom}
\thanks{\copyright{} 2026. This manuscript version is made available under the \href{https://creativecommons.org/licenses/by-nc-nd/4.0/}{CC-BY-NC-ND 4.0} license. Sections 1--3 correspond to the article published in Topology and its Applications 395 (2027) 110014, \url{https://doi.org/10.1016/j.topol.2026.110014}, with results numbered as there. Section 4 does not appear in the published article.}

\address{Department of Mathematics \\ Harvard University \\ Cambridge MA 02138}
\email{ilongbottom@math.harvard.edu}

\begin{document}

\begin{abstract}
We define a class of smashing localisations which we call \emph{compactly central} and classify compactly central localisations of the category of spectra and of the category of $p$-local spectra. In both cases the compactly central localisations are those finite localisations which correspond to closed subsets of the Balmer spectrum with quasicompact complement. In particular, the finite chromatic localisation $L_n^f$ is compactly central. 

We generalise our results to the category $D(X)$ for $X$ a quasicompact quasiseparated scheme with an ample family of line bundles. Analogously to spectra, we show that compactly central localisations of $D(X)$ correspond to closed subsets of $X$ with quasicompact complement. This includes any scheme which is quasiprojective over an affine base.
\end{abstract}

\maketitle

\tableofcontents

\section{Introduction}

Throughout, we let $\CC$ denote a presentably symmetric monoidal $\infty$-category. That is, $\CC$ is a presentable $\infty$-category equipped with a symmetric monoidal structure whose tensor product preserves small colimits separately in each variable. By the adjoint functor theorem (\cite{htt} Corollary 5.5.2.9), $\CC$ then has an internal hom $\hom(-,-): \CC^{\op} \times \CC \to \CC$, characterised by $\Map_\CC(A, \hom(B, C)) \simeq \Map_\CC(A \otimes B, C)$. We will often require $\CC$ to be stable and rigid. Our convention is that a rigid category is compactly generated, and its compact and dualisable objects agree. Let $\one$ denote the monoidal unit in $\CC$. Recall that $\Sp$ denotes the stable $\infty$-category of spectra, and $\Sp_{(p)}$ its localisation with respect to some prime $p \in \Z$. The homotopy groups of $X \in \Sp_{(p)}$ are $\Z_{(p)}$-modules. We provide background on localisations and smashing localisations in Subsection \ref{subsec: intro background}, and introduce the novel results of this paper in Subsection \ref{subsec: intro real content}.

\subsection{Background on localisations}\label{subsec: intro background}
A \emph{localisation} of $\CC$ is a functor $L: \CC \to \DD$ which admits a fully faithful right adjoint $\iota$. If $\CC$ is stable, we impose the additional conditions that $\DD$ is stable and $\iota$ is exact. Such a functor simplifies $\CC$ by inverting some of its morphisms. Objects of $\CC$ lying in the essential image of $\iota$ are called \emph{local}, and morphisms in $\CC$ whose image under $L$ is invertible are called \emph{$L$-local equivalences}. If $\CC$ is stable, then objects of $\CC$ whose image under $L$ is the zero object of $\DD$ are called \emph{acyclic}. The full subcategory of $\CC$ on the acyclic objects is the \emph{kernel} of the localisation $L$, denoted $\ker L$. The unit of the adjunction provides a natural transformation $\eta: \id_\CC \implies \iota \circ L$ mapping any object of $\CC$ to its localisation.

The composite $\iota \circ L$ is an endofunctor of $\CC$ which is idempotent in the sense that $(\iota \circ L) \circ (\iota \circ L)$ is naturally equivalent to $\iota \circ L$. Indeed, a localisation may equivalently be defined as an endofunctor $L$ of $\CC$ equipped with a natural transformation $\eta: \id_\CC \implies L$, subject to the condition that $\eta L$ and $L \eta$ are both equivalences $L \implies L\circ L$. When $\CC$ is stable, we must also assume $L$ is exact. From this perspective, $\DD$ can be recovered as the full subcategory of $\CC$ on objects $X$ for which $\eta_X: X \to LX$ is an equivalence. We will implicitly use both of these perspectives on localisation, and identify $\DD$ with its essential image in $\CC$. We use $L\CC \simeq \DD$ to denote the full subcategory of local objects in $\CC$.

A localisation $L$ of a stable category $\CC$ is determined by its kernel, and the local objects can be recovered as the right orthogonal to the kernel. That is, $X \in \CC$ is local if and only if $\Map_\CC(A, X) \simeq *$ for every acyclic $A$. Our additional assumption on localisations of stable categories guarantees that $L$ is exact, $\ker L$ is closed under desuspension, and $\iota$ identifies $\DD$ with a stable subcategory of $\CC$. Recall that a full subcategory of a stable $\infty$-category $\CC$ is \emph{thick} if it is a stable subcategory closed under retracts, and \emph{localising} if it is moreover closed under small coproducts in $\CC$. For a set of objects $N$ of $\CC$ we write $\langle N\rangle$ for the smallest localising subcategory of $\CC$ containing $N$. The kernel of any localisation $L$ of $\CC$ is localising and in particular thick.

One is often interested in understanding which localisations of $\CC$ are compatible with its symmetric monoidal structure. We have the following key definition.

\begin{defn}[\cite{ravenel84} Definition 1.28]
    A localisation $L$ of $\CC$ is \emph{smashing} if there is a natural equivalence $L(-) \simeq A \otimes -$ for some object $A$ of $\CC$. If such an $A$ exists then $A \simeq L\one$ because $L\one \simeq A \otimes \one \simeq A$.
\end{defn}

If $L$ is smashing then the subcategory $\DD$ inherits a symmetric monoidal structure from $\CC$, making $L$ a symmetric monoidal functor. The tensor product of two objects in $\DD$ is the same as their tensor product in $\CC$. The object $L\one$ comes equipped with an $\E_\infty$-algebra structure in $\DD$, because it is the image of the $\E_\infty$-algebra $\one$ under a symmetric monoidal functor. Moreover, it inherits idempotence from $L$: since $L \circ L \simeq L$, we have $L\one \otimes L\one \simeq L\one$. The natural transformation $\eta$ provides a canonical comparison map $\eta_\one: \one \to L\one$, and the equivalence $L\one \xrightarrow{\simeq} L\one \otimes L\one$ is given by $\eta_\one \otimes \id_{L\one}$. We can package these properties up as a definition, and then the data of a smashing localisation is equivalent to specifying the \emph{idempotent algebra} object $L\one$.

\begin{defn}[\cite{ha} Definition 4.8.2.1]\label{def: idempotent}
    A map $e: \one \to A$ in $\CC$ is \emph{idempotent} if the maps
    \[e\otimes \id_A, \id_A \otimes e: A \to A \otimes A\]
    are both equivalences. In this case $A$ obtains a canonical $\E_\infty$-algebra structure, and we call $A$ an \emph{idempotent algebra} of $\CC$. The map $e$ is the \emph{unit} of $A$. 
\end{defn}

Because $\CC$ is symmetric monoidal it is sufficient to check that just one of $e \otimes \id_A$ and $\id_A \otimes e$ is an equivalence, since postcomposing with the autoequivalence of $A \otimes A$ interchanging the two factors swaps these two maps. As a consequence of idempotence, $e \otimes \id_A \simeq \id_A \otimes e$. Multiplication on $A$ is the inverse to either of these two maps, giving $A$ its $\E_\infty$ structure. We have the following precise correspondence (see \cite{ha} Section 4.8.2, in particular Proposition 4.8.2.9). 

\begin{prop}\label{prop: class of smashing locs}
    A smashing localisation $L$ of $\CC$ determines an idempotent algebra $L\one$ with unit map $\eta_{\one}: \one \to L\one$ which is idempotent. Conversely the unit map $e$ of an idempotent algebra $A \in \CC$ determines a smashing localisation via $LX = A \otimes X$ and $\eta_X = e \otimes \id_X$. Smashing localisations of $\CC$ are in bijection with idempotent maps $\one \to A$ in $\CC$ via these two inverse translations.
\end{prop}

Recall that an $\E_0$-algebra structure on an object $A \in \CC$ is simply a map $\alpha: \one \to A$, so that $\Alg_{\E_0}(\CC) \simeq \CC_{\one/}$. This is the degenerate case of an $\E_n$-algebra structure, and as such there are forgetful functors $\Alg_{\E_n}(\CC) \to \Alg_{\E_0}(\CC)$ for every $n$, which remember only the unit map of the $\E_n$-algebra structure. Let $J_\alpha$ denote the free $\E_1$-algebra on the $\E_0$-algebra $(A, \alpha)$. Explicitly, $J_\alpha$ is an $\E_1$-algebra equipped with a map $c: A \to J_\alpha$ and a homotopy identifying $c \circ \alpha$ with the unit $\overline{\alpha}$ of $J_\alpha$ such that restriction along $c$ gives an equivalence
\[\Map_{\Alg_{\E_1}(\CC)}(J_\alpha, R) \xrightarrow{\simeq} \Map_{\Alg_{\E_0}(\CC)}(A, R)\]
for every $R \in \Alg_{\E_1}(\CC)$. Another way of describing this is that $J_\alpha$ is the image of $(A, \alpha)$ under the left adjoint to the forgetful functor $\Alg_{\E_1}(\CC) \to \Alg_{\E_0}(\CC)$. We can package $c$ together with its homotopy as a map of $\E_0$-algebras $c: (A, \alpha) \to (J_\alpha, \overline{\alpha})$. We study a general method for building smashing localisations which uses this construction as a key step.

\subsection{Main results}\label{subsec: intro real content}

An $\E_0$-algebra $\alpha: \one \to A \in \CC$ is \emph{central} if there exists a homotopy $\alpha \otimes \id_A \simeq \id_A \otimes \alpha: A \to A \otimes A$ (Definition \ref{def: central}). The free $\E_1$-algebra on a central $\E_0$-algebra has idempotent unit (Corollary \ref{cor: Jalpha is idempotent}), hence we obtain a smashing localisation via Proposition \ref{prop: class of smashing locs}. In this situation we say that $\alpha$ \emph{presents} the smashing localisation $LX := J_\alpha \otimes X$ of $\CC$.

Given a smashing localisation $L$, the unit map of the idempotent algebra $L\one$ is central and presents $L$. Thus every smashing localisation can be presented by a central map in a trivial way. It is more interesting to consider cases where $\alpha$ is central but not itself idempotent. In particular, we would like to understand the situation where $A$ is a compact object of $\CC$. Idempotent algebras are rarely compact -- indeed, in $\Sp$ the only compact idempotent algebras are $0$ and the monoidal unit $\S$. We call a localisation $L$ \emph{compactly central} if it can be presented by a central $\E_0$-algebra $\alpha: \one \to A$ where $A$ is compact. We study and classify compactly central localisations of various categories of interest, including $\Sp$, $\Sp_{(p)}$, and $D(R)$.

Requiring a localisation to be compactly central is a finiteness condition. There is an existing notion of finiteness in the literature -- a localisation is finite if its acyclic objects are generated under colimits and desuspensions by the compact acyclics (Miller, \cite{millerfinite}). On $\Sp$ and on $\Sp_{(p)}$, every compactly central localisation is finite (Lemma \ref{lem: compactly central implies finite}) and every finite localisation is smashing (Miller, \cite{millerfinite} Proposition 9; appearing here as Corollary \ref{cor: sp fin implies smashing}). The thick subcategory theorem (Hopkins-Smith, \cite{NilpotenceII} Theorem 7) classifies finite localisations of $\Sp_{(p)}$ as precisely the localisations $L_n^f$ for $-1 \leq n \leq \infty$, where the kernel of $L_n^f$ is generated by type $\geq (n+1)$ $p$-local finite spectra. Here $L_{-1}^f$ is the zero localisation, whose kernel is $\Sp_{(p)}$. It is natural to ask which of these finite localisations are compactly central. Our first result is that $L_n^f$ is compactly central.

\begin{mainthm}[Theorem \ref{thm: main result lnf compactly central}]\label{intro thm: p-local compactly central}
    The localisation $L_n^f$ is compactly central. Thus all finite localisations of $\Sp_{(p)}$ are compactly central.
\end{mainthm}

Theorem \ref{intro thm: p-local compactly central} refines Miller's construction of finite localisations (\cite{millerfinite} Theorem 4). Miller builds $L_n^f \S_{(p)}$ as a sequential colimit $\S_{(p)} \to X_1 \to X_2 \to \cdots$, in which each $X_{k+1}$ is obtained from $X_k$ by killing every map into $X_k$ from a $p$-local finite spectrum of type $\geq n+1$. A compactly central presentation of $L_n^f$ builds $L_n^f \S_{(p)}$ as a sequential colimit of the same shape, $\S_{(p)} \to A \to A^{\otimes 2} \to \cdots$, obtained by repeatedly tensoring with a single map $\alpha: \S_{(p)} \to A$ with $A$ finite (Construction \ref{constr: explicit Jalpha} and Proposition \ref{prop: Jalpha Jalpha' agree}). The cofibre of each map in the sequence is a finite spectrum $A^{\otimes k} \otimes \cof \alpha$ of type $n+1$.

We use Theorem \ref{intro thm: p-local compactly central} to classify all compactly central localisations of $\Sp$. Nonzero finite localisations of $\Sp$ are uniquely determined by their restrictions to $\Sp_{(p)}$ for each prime $p$ (Lemmas \ref{lem: fin loc sp uniqueness} and \ref{lem: fin loc sp existence}). If we independently choose nonzero finite localisations of $\Sp_{(p)}$ for each prime $p$, there always exists a compatible nonzero finite localisation of $\Sp$. This localisation of $\Sp$ is compactly central if and only if we restrict ourselves to making finitely many choices, and take the identity localisation at all other primes. This leads to the following classification.

\begin{mainthm}[Theorem \ref{thm: not compactly central}]\label{intro thm: sp compactly central}
    A nonzero finite localisation $L$ of $\Sp$ is compactly central if and only if $L$ produces the identity when restricted to $\Sp_{(p)}$, for all but finitely many primes $p$.
\end{mainthm}

The following example of a family of finite localisations may help to illustrate Theorem \ref{intro thm: sp compactly central}. For a set $P$ of primes in $\Z$, the localisation of $\Sp$ which inverts the primes in $P$ is smashing, given by tensoring with the Moore spectrum $\S[P^{-1}]$ of $\Z[P^{-1}]$ (\cite{bousfield} Proposition 2.4). This localisation is also finite: its acyclics are generated under colimits and desuspensions by the Moore spectra $\S/p$ for $p \in P$. When $P$ is the set of primes dividing an integer $M$, we write $L_{1/M}$ for this localisation. Rationalisation is the localisation of $\Sp$ obtained by inverting all primes, and $p$-localisation is obtained by inverting all primes except $p$.

Restricting $L_{1/M}$ to $\Sp_{(p)}$ gives rationalisation if $p \in P$ and the identity otherwise, so $L_{1/M}$ is compactly central by Theorem \ref{intro thm: sp compactly central}. On the other hand, $p$-localisation and rationalisation each invert infinitely many primes, so they are finite but not compactly central.

Theorems \ref{intro thm: p-local compactly central} and \ref{intro thm: sp compactly central} can be reformulated in uniform language, using the Balmer spectrum. Write $\CC^\omega$ for the full subcategory of compact objects in $\CC$. The Balmer spectrum $\Spc(\CC^\omega)$ is the space of prime thick tensor ideals of $\CC^\omega$ (\cite{balmerfinclass} Definition 2.1, appearing here as Definition \ref{def: balmer spectrum}). It has a Zariski topology, and arbitrary unions of those closed subsets whose complements are quasicompact are called \emph{Thomason subsets}. To a localisation $L$, one can associate the subset of $\Spc(\CC^\omega)$ consisting of prime ideals which do not contain $\ker L \cap \CC^\omega$, called the \emph{support} of $L$. When $\CC$ is stable and rigid, this assignment restricts to a bijection between finite localisations of $\CC$ that are symmetric monoidal functors and Thomason subsets of $\Spc(\CC^\omega)$ (see Corollary \ref{cor: balmer correspondence} and \cite{balmerfinclass} Theorem 4.10). In this language, and using Balmer's computation of $\Spc(\Sp^\omega)$ (\cite{balmerspectra} Corollary 9.5), Theorems \ref{intro thm: p-local compactly central} and \ref{intro thm: sp compactly central} both state that compactly central localisations are exactly those finite localisations with support a closed Thomason subset. See Subsection \ref{subsec: balmer spectrum} for an introduction to the Balmer spectrum and a detailed description of $\Spc(\CC^\omega)$ in the cases $\CC = \Sp$ and $\CC = \Sp_{(p)}$.

Inspired by the Balmer spectrum perspective, we generalise Theorems \ref{intro thm: p-local compactly central} and \ref{intro thm: sp compactly central} to the setting $\CC = D(X)$ where $X$ is a quasicompact and quasiseparated (qcqs) scheme with an ample family of line bundles (Definition \ref{def: ample family}). We have the following.

\begin{mainthm}[Theorem \ref{thm: closed implies cc for R} and Lemma \ref{lem: balmer central translator}]\label{intro thm: D(R) case}
    Let $R$ be a commutative ring. Compactly central localisations of $D(R)$ are in bijection with closed Thomason subsets of the Balmer spectrum $\Spc(\Perf(R)) \cong \Spec R$, via taking supports. 
\end{mainthm}

\begin{mainthm}[Theorem \ref{thm: closed implies cc ample family} and Lemma \ref{lem: balmer central translator}]\label{intro thm: D(X) with ample family case}
    Let $X$ be a qcqs scheme with an ample family of line bundles. Compactly central localisations of $D(X)$ are in bijection with closed Thomason subsets of the Balmer spectrum $\Spc(\Perf(X)) \cong X$, via taking supports. In particular, this holds for $X$ quasiprojective over an affine base.
\end{mainthm}

A map from the unit to a line bundle in $D(X)$ is always central (Lemma \ref{lem: sections of line bundles are central}). Our proof of Theorem \ref{intro thm: D(X) with ample family case} relies on starting with sections of line bundles, and then building central maps corresponding to finite intersections of their zero loci. The hypothesis that $X$ has an ample family of line bundles is used only to guarantee that there are enough sections of line bundles. Specifically, to ensure we have a sufficient supply of central maps, we use that every closed subset of $X$ with quasicompact complement is a finite intersection of zero loci of sections of line bundles.

\subsection{Structure of the paper}
Subsections \ref{subsec: mon loc} and \ref{subsec: combining smash locs} develop background material on smashing localisations. In Subsection \ref{subsec: central maps} we establish several key properties of central maps, and explain how to use a central map to build a smashing localisation. We also develop a relationship between the fibre of a central map and the kernel of the localisation it presents, which we will later use to compute the localisations presented by central maps between compact objects. In Subsection \ref{subsec: finiteness} we relate finite and compactly central localisations, and explain how any localisation can be universally approximated by a finite localisation (due to Miller, \cite{millerfinite}). This motivates Subsection \ref{subsec: univ smash loc}, where we explain how any localisation can be universally approximated by a smashing localisation. This is independent of the narrative of the remainder of the paper. 

The purpose of Section \ref{sec: the Lnf example} is to prove Theorems \ref{intro thm: p-local compactly central} and \ref{intro thm: sp compactly central}, and we usually work directly in $\Sp_{(p)}$. In Subsection \ref{subsec: props of central maps} we develop algebraic criteria in terms of $K(n)$-homology that a central map between finite spectra must satisfy, and compute the compactly central localisation that such a map presents. We call a map \emph{algebraically central} if it satisfies these algebraic properties but is not necessarily central. In Subsection \ref{subsec: alg central to central} we show that a sufficiently large tensor power of an algebraically central map is central. This proof works in any stable rigid $\CC$ with a family of nilpotence-detecting homological functors, but is phrased here specifically for $\CC = \Sp_{(p)}$, and the generalisation is delayed to Theorem \ref{thm: alg central power is central general}. Taking tensor powers of algebraically central maps provides us with a supply of central maps between $p$-local finite spectra needed in the proof of Theorem \ref{intro thm: p-local compactly central}. Subsection \ref{subsec: passage to spectra} contains the passage from Theorem \ref{intro thm: p-local compactly central} to Theorem \ref{intro thm: sp compactly central}. In Subsection \ref{subsec: construction of algebraically central maps} we explicitly construct a family of algebraically central maps needed for the proof of Theorem \ref{intro thm: p-local compactly central}.

The purpose of Section \ref{sec: the D(X) case} is to prove Theorems \ref{intro thm: D(R) case} and \ref{intro thm: D(X) with ample family case}. In Subsection \ref{subsec: balmer spectrum} we provide background on the Balmer spectrum, and explain the usual correspondence between its Thomason subsets and finite localisations of $\CC$. We also reinterpret Theorems \ref{intro thm: p-local compactly central} and \ref{intro thm: sp compactly central} as statements about the Balmer spectrum. In Subsection \ref{subsec: spectra vs qcqs schemes} we audit the formal inputs to Section \ref{sec: the Lnf example} and discuss their analogues in $\CC$ and in the specific case $\CC = D(X)$ for $X$ a qcqs scheme. Subsection \ref{subsec: cc locs closed support} generalises Subsection \ref{subsec: props of central maps} to any stable rigid $\CC$. In Subsection \ref{subsec: unions and intersections} we embark on a study of how intersections and unions of closed Thomason subsets correspond to operations on central maps, which is used to reduce the number of central maps we must construct by hand. Finally, Subsection \ref{subsec: cc locs building blocks scheme} is analogous to Subsection \ref{subsec: construction of algebraically central maps} and constructs a family of elementary central maps in $D(X)$ using sections of line bundles.

\subsection*{Acknowledgements}
I would like to thank Mike Hopkins for suggesting this line of inquiry and for many subsequent helpful conversations. I am also grateful to Dhilan Lahoti for comments on an early draft of this paper and for a number of mathematical discussions which shaped Section \ref{sec: the D(X) case}. Thanks also to an anonymous referee for thoughtful suggestions to improve the exposition and mathematics of the paper.

\subsection*{Declaration of generative AI use}
The initial draft of this paper was prepared without the use of generative AI. During the editing process, the author used generative AI tools to assist with detecting errors, mathematical oversights, confusing exposition, and missing citations in the manuscript. In particular, the simplified argument in Subsection \ref{subsec: alg central to central} was suggested by AI, although its consequences in Section \ref{sec: the D(X) case} are the author's work. All material in this paper has been carefully reviewed and edited by the author, who takes full responsibility for the content and for any remaining inaccuracies.

\section{Smashing localisations}\label{sec: smashing}

This Section begins with some exposition regarding smashing localisations of a presentably symmetric monoidal $\infty$-category $\CC$. Subsections \ref{subsec: mon loc} and \ref{subsec: combining smash locs} record the basic properties of smashing localisations, compare them to localisations that are compatible with the tensor product in a weaker sense, and explain how smashing localisations can be combined. Subsections \ref{subsec: central maps} and \ref{subsec: finiteness} introduce the machinery of compactly central localisations. Subsection \ref{subsec: univ smash loc} constructs the universal smashing localisation approximating a given localisation. We begin by recalling a basic structural result about smashing localisations.

\begin{prop}[\cite{ha} Proposition 4.8.2.10]\label{prop: smashing loc modules over unit}
    If $L$ is a smashing localisation then the forgetful functor 
    \[\Mod_{L\one}(\CC)^\otimes \to \CC^\otimes\]  
    determines a symmetric monoidal equivalence $\Mod_{L\one}(\CC)^\otimes \xrightarrow{\simeq} (L\CC)^\otimes$.
\end{prop}

Let us now discuss some of the consequences. Given a smashing localisation $L$, every $L$-local object has a unique $L\one$-module structure, and we can detect whether an object $X$ is an $L\one$-module by checking whether the map $X \xrightarrow{e \otimes \id_X} L\one \otimes X$ is an equivalence. It can be helpful to think of the module structure on a local object $X$ as arising from the $L\one$-module structure on $L\one$ itself, transferred via this identification.

When $L$ is smashing the product of a local object with any object is local, since 
\[LX \otimes Y \simeq L\one \otimes X \otimes Y \simeq L(X \otimes Y),\]
and for objects $X, Y \in \CC$ we have natural equivalences $L(X \otimes Y) \simeq LX \otimes LY \simeq LX \otimes Y \simeq X \otimes LY$.

We also state a characterisation of the acyclic objects for a localisation, and more specifically, for a smashing localisation. Recall that a \emph{strongly saturated class} of morphisms in $\CC$ is a collection of morphisms which is closed under all pushouts, has the two out of three property for composition, and is closed under small colimits taken in $\Fun(\Delta^1, \CC)$ (see \cite{htt} Definition 5.5.4.5). For any localisation $L$, the class of $L$-local equivalences is strongly saturated.

\begin{lem}\label{lem: acyclics gen by}
    Let $\CC$ be a presentably symmetric monoidal stable $\infty$-category, and $L$ a localisation of $\CC$. The $L$-local equivalences are generated as a strongly saturated class by local equivalences of the form $X \to LX$ for objects $X \in \CC$. The acyclics are generated under colimits by cofibres of such. If $L$ is smashing with idempotent algebra $L\one$, then the acyclics all have the form $X \otimes \cof(\one \to L\one)$ for $X\in \CC$.

    If $L$ is smashing and $\CC$ is generated by the monoidal unit under colimits and desuspensions, then $\ker L = \langle \cof(\one \to L\one) \rangle$.
\end{lem}

\begin{proof}
    Given any local equivalence $f: A \to B$, the natural transformation $\id_\CC \implies L$ which comes with the localisation $L$ produces a commuting square
    \[\begin{tikzcd}[sep={6em, between origins}]
        A \arrow[d] \arrow[r, "f"]    & B \arrow[d] \\
        LA \arrow[r, "Lf"', "\simeq"] & LB.         
        \end{tikzcd}\]
    Since $Lf$ is an equivalence it is a pushout of the identity on $LA$ and so belongs to \emph{any} strongly saturated class. By including morphisms of the form $A \to LA$, three of the four morphisms in the square lie in the class and so $f$ is forced to belong to the strongly saturated class. We are working stably, so a local equivalence is precisely a map whose cofibre is acyclic. The characterisation of acyclics follows.

    When $L$ is smashing, take the cofibre sequence $\one \xrightarrow{\alpha} L\one \to \cof \alpha$ and tensor with any object $A$. We obtain a new cofibre sequence 
    \[A \to LA \to A \otimes \cof \alpha\]
    so that $\cof(A \to LA) \simeq A \otimes \cof \alpha$. 

    Now suppose the monoidal unit generates $\CC$ under colimits and desuspensions. The full subcategory $\AA$ of objects $X$ such that $X \otimes \cof \alpha \in \langle \cof\alpha \rangle$ is closed under colimits and desuspensions, since $- \otimes \cof \alpha$ preserves both. Since $\one \in \AA$, we find that $\AA = \CC$. Because every acyclic has the form $X \otimes \cof\alpha$ and $\cof\alpha$ is acyclic, it follows that $\ker L = \langle \cof \alpha \rangle$.
\end{proof}

We say that a localisation $L$ is \emph{accessible} if its class of local equivalences is generated as a strongly saturated class by a set of morphisms. Conversely, any strongly saturated class which is generated by a set of morphisms is the class of local equivalences for an accessible localisation (\cite{htt} Section 5.5.4). When $\CC$ is stable, a map is a local equivalence precisely when its cofibre is acyclic. It follows that $L$ is accessible if and only if $\ker L$ is generated as a localising subcategory of $\CC$ by a set of objects.

\subsection{Monoidal localisations}\label{subsec: mon loc}

We now turn to localisations that are compatible with the tensor product in a weaker sense than being smashing. Let $L$ be a localisation of a pointed symmetric monoidal category $\CC$. We are interested in how $L$ interacts with the tensor product on $\CC$, and we would like to understand what compatibility properties $L$ may have with the tensor product. In particular, it will be useful for us to have a criterion under which a localisation that is a symmetric monoidal functor is automatically smashing. Consider the diagram
\begin{equation*}
\begin{tikzcd}[sep={12em,between origins}]
    X \otimes Y \arrow[d, "\eta_{X \otimes Y}"] 
    \arrow[r, "\eta_X \otimes \id_Y"] 
    & LX \otimes Y  \arrow[r, "\id_{LX} \otimes \eta_Y"] 
    \arrow[d, "\eta_{LX \otimes Y}"]
    & LX \otimes LY \arrow[d, "\eta_{LX \otimes LY}"] \\
    L(X \otimes Y) \arrow[r, "L(\eta_X \otimes \id_Y)"']                      
    & L(LX \otimes Y) \arrow[r, "L(\id_{LX} \otimes \eta_Y)"']
    & L(LX \otimes LY)                            
\end{tikzcd}
\tag{$\ddagger$}\label{diag: tensor and loc}
\end{equation*}
which commutes via the unit $\eta: \id_\CC \implies L$ for the localisation $L$. This can be upgraded to a commuting diagram of natural transformations by replacing each vertex with the bifunctor $\CC \times \CC \to \CC$ it defines and each arrow by a natural transformation between the corresponding functors. Looking at just the left square, we get
\[\begin{tikzcd}[sep={12em,between origins}]
    - \otimes - \arrow[r, "\eta_{-} \otimes \id_{-}", Rightarrow] \arrow[d, "\eta(-\otimes -)"', Rightarrow] & L(-) \otimes \id(-) \arrow[d, "\eta(L(-) \otimes \id(-))", Rightarrow] \\
    L(- \otimes -) \arrow[r, "L(\eta_{-} \otimes \id_{-})"', Rightarrow]                              & L(L(-) \otimes \id(-)).                                     
\end{tikzcd}\]

\begin{defn}\label{def: tensor compatible}
    Call $L$ \emph{tensor-compatible} if the natural transformation \hbox{$L(\eta_{-} \otimes \id_{-})$} is a natural equivalence of bifunctors. 
\end{defn}

See \cite{day} Theorem 1.2 for the original characterisation of tensor-compatible localisations. In particular property (4) there is our definition. A modern treatment is given in \cite{ha} Definition 2.2.1.6 and Example 2.2.1.7. 

If $L$ is tensor-compatible and in addition \hbox{$\eta(L(-) \otimes \id(-))$} is a natural equivalence of bifunctors, then all the bifunctors in Diagram \eqref{diag: tensor and loc} except $\otimes$ are naturally equivalent to one another. In this case one can show that $L$ is smashing (Proposition \ref{prop: monoidal + local mult = smashing}).

\begin{lem}\label{lem: equiv monoidal loc}
    The following are equivalent when $L$ is a localisation of a presentably symmetric monoidal stable category $\CC$:
    \begin{enumerate}[label=(\arabic*)]
        \item $L$ is tensor-compatible; \label{cond: my monoidal}
        \item for any local equivalence $f: X \to Y$ and object $Z$, $f \otimes \id_Z: X \otimes Z \to Y \otimes Z$ is also a local equivalence; \label{cond: lurie monoidal}
        \item the $L$-acyclic objects form a tensor ideal, meaning if $X$ is acyclic and $Z$ is any object then $X \otimes Z$ is also acyclic. \label{cond: tensor ideal}
    \end{enumerate}
\end{lem}

Condition \ref{cond: lurie monoidal} is usually called being \emph{compatible with the symmetric monoidal structure on $\CC$}, see \cite{ha} Example 2.2.1.7. Conditions \ref{cond: lurie monoidal} and \ref{cond: tensor ideal} are equivalent only in the case where $\CC$ is stable, as in the unstable setting acyclic objects are not very well behaved. Conditions \ref{cond: my monoidal} and \ref{cond: lurie monoidal} are equivalent even in the unstable setting.

\begin{proof}
    Since $\CC$ is stable, the acyclic objects are precisely the cofibres of local equivalences. Then $\cof(f \otimes \id_Z) \simeq (\cof f) \otimes Z$ so $f \otimes \id_Z$ is a local equivalence if and only if $(\cof f) \otimes Z$ is acyclic. It follows that \ref{cond: lurie monoidal} $\Leftrightarrow$ \ref{cond: tensor ideal}.

    It is clear that Condition \ref{cond: lurie monoidal} $\implies \eta_X \otimes \id_Y: X \otimes Y \to LX \otimes Y$ is a local equivalence $\implies$ Condition \ref{cond: my monoidal}. Conversely, if Condition \ref{cond: my monoidal} holds then all maps of the form $\eta_X \otimes \id_Y$ are local equivalences. Let $f: X_1 \to X_2$ be an arbitrary local equivalence, and consider the commuting diagram 
    \[\begin{tikzcd}[sep={8em,between origins}]
        X_1 \otimes Y \arrow[r, "f \otimes \id_Y"] \arrow[d, "\eta_{X_1} \otimes \id_Y"'] & X_2 \otimes Y \arrow[d, "\eta_{X_2} \otimes \id_Y"] \\
        LX_1 \otimes Y \arrow[r, "Lf \otimes \id_Y"']                                  & LX_2 \otimes Y,                                  
        \end{tikzcd}\]
        where both vertical maps are local equivalences by Condition \ref{cond: my monoidal}. Since $f$ is a local equivalence, $Lf$ is an equivalence, and hence $Lf \otimes \id_Y$ is an equivalence and therefore a local equivalence. Now $f \otimes \id_Y$ fits into a commutative diagram where all the other maps are local equivalences, so it too is a local equivalence due to the closure properties on local equivalences (see \cite{htt} Section 5.5.4). Condition \ref{cond: lurie monoidal} follows.
\end{proof}

A tensor-compatible localisation comes equipped with a natural equivalence 
\[L(X \otimes Y) \to L(LX \otimes LY),\]
using the natural equivalence of Definition \ref{def: tensor compatible} twice. This means we can define a symmetric monoidal structure on the category $L\CC$ of local objects by $A \otimes_{L\CC} B := L(A \otimes_\CC B)$, restricting the bifunctor $L(- \otimes -): \CC \times \CC \to \CC$ to $L\CC$. The tensor product on $L\CC$ does not necessarily agree with the plain restriction of the tensor product on $\CC$ to $L\CC$, because $L\CC$ can fail to be closed under the tensor product on $\CC$. With our chosen monoidal structure on ${L\CC}$, the natural equivalence above becomes $L(X \otimes_\CC Y) \xrightarrow{\simeq} LX \otimes_{L\CC} LY$ and so the localisation $L$ is compatible with the respective monoidal structures on $\CC$ and ${L\CC}$. That is, $L: \CC \to {L\CC}$ becomes a symmetric monoidal functor. In this case the fully faithful functor ${L\CC} \into \CC$ is lax monoidal but not in general monoidal. The endofunctor of $\CC$ induced by the localisation is also lax monoidal.

Note that $\iota$ canonically identifies $\DD$ with $L\CC$ equipped with the monoidal structure $\otimes_{L\CC}$ above, so these two symbols always refer to the same symmetric monoidal category. We use the notation $\DD$ when we wish to strongly emphasise that the symmetric monoidal structure on $L\CC$ need not agree with the one on $\CC$.

\begin{lem}[\cite{ha} Proposition 2.2.1.9]\label{lem: monoidal loc is monoidal}
    A tensor-compatible localisation $L: \CC \to \mathscr{D}$ is a symmetric monoidal functor. The inclusion $\mathscr{D} \into \CC$ is a lax monoidal functor.
\end{lem}

For the proof that $L$ is symmetric monoidal, see \cite{ha}. The result for $\iota$ follows from the result for $L$ because $\iota$ is right adjoint to $L$ and $L$ is monoidal (\cite{ha} Corollary 7.3.2.7).

\begin{lem}\label{lem: qsmashing is really symm mon}
    Let $L: \CC \to \mathscr{D}$ be a tensor-compatible localisation such that \hbox{$\eta(L(-) \otimes L(-))$} is a natural equivalence. That is, the $\CC$-tensor product of two local objects is local. Then $\mathscr{D} \xrightarrow{\iota} \CC$ is a symmetric monoidal functor, except that $\iota$ need not preserve the tensor unit. Specifically, there is a morphism $\one_{\CC} \to \iota(\one_\mathscr{D}) = \iota L(\one_\CC)$ which need not be an equivalence, and all other coherence conditions for a symmetric monoidal functor are satisfied.
\end{lem}

\begin{proof}
    By Lemma \ref{lem: monoidal loc is monoidal}, $\iota$ is lax symmetric monoidal. For local objects $A$ and $B$, the structure map $A \otimes_\CC B \to A \otimes_\DD B$ is the composite of $\eta_{A \otimes_\CC B}: A \otimes_\CC B \to L(A \otimes_\CC B)$ with the equivalence $L(A \otimes_\CC B) \simeq A \otimes_\DD B$ coming from tensor-compatibility of $L$. Since $A \otimes_\CC B$ is local by hypothesis, $\eta_{A \otimes_\CC B}$ is an equivalence too, so the structure map is a natural equivalence as desired.
\end{proof}

\begin{remark}
    What hypotheses on $L$ and $\iota$ are required to make both functors symmetric monoidal? The situation of Lemma \ref{lem: qsmashing is really symm mon} is very close, but we would need $\iota$ to also preserve the monoidal unit.

    If $L$ is smashing and $\iota$ preserves the monoidal unit then $L$ is the identity localisation on $\CC$. More precisely, if the canonical morphism $\one_\CC \to \iota L \one_\CC = \iota \one_\DD$ is an equivalence then $L$ and $\iota$ are inverse symmetric monoidal equivalences between $\CC$ and $\DD$. To see why, note that $\iota LX \simeq \iota L\one \otimes X$ for every $X \in \CC$. Hence if $\one \simeq \iota L\one$, we find that $\iota L X \simeq X$ for all objects $X \in \CC$. Moreover the unit $\id_\CC \implies \iota \circ L$ of the adjunction $L \dashv \iota$ is a natural equivalence. The counit of every localisation adjunction is an equivalence because the right adjoint is fully faithful, so it follows that $\iota$ and $L$ are inverse equivalences. 

    However, there exist non-smashing localisations which preserve the monoidal unit but do not preserve every object of $\CC$. Consider the example $\CC = \Fun(\Delta^1, \Sp)$ with pointwise tensor product, $\DD = \Sp$, and $L$ given by evaluation of a morphism at its target. The subcategory of local objects in $\CC$ consists of equivalences $X \xrightarrow{\simeq} X$, and the monoidal unit in $\CC$ is $\S \xrightarrow{\id} \S$ which is already local. The $L$-acyclic objects are arrows of the form $X \to 0$ for $X \in \Sp$, which form a tensor ideal, and hence $L$ is tensor-compatible. The tensor product of two local objects in $\CC$ is local, because tensoring equivalences produces an equivalence. So this localisation satisfies the hypotheses of Lemma \ref{lem: qsmashing is really symm mon}. Not all objects of $\CC$ are local, so $L$ cannot be smashing. Although the monoidal unit is local, its tensor product with the object $\S \xrightarrow{0} \S$ is not local, so $L$ does not satisfy the stronger hypothesis of Proposition \ref{prop: monoidal + local mult = smashing}.
\end{remark}

\begin{prop}[cf \cite{axiomatic} Lemma 3.3.1 and Definition 3.3.2]\label{prop: monoidal + local mult = smashing}
    If $L$ is a tensor-compatible localisation and the $\CC$-tensor product of a local object with any object is local, then $L$ is smashing.
\end{prop}

\begin{proof}
   We know the product of a local object with any object is local, so $L\one \otimes X$ is local for every $X \in \CC$. Since $L$ is tensor-compatible, $\eta_\one \otimes \id_X: X \to L\one \otimes X$ is a local equivalence. A local equivalence with local target exhibits its target as the localisation, so $LX \simeq L\one \otimes X$ naturally in $X$, and $L$ is smashing.

   Alternatively, combining two of the five natural equivalences in Diagram \eqref{diag: tensor and loc} and setting $X = \one$ produces a natural equivalence $L\one \otimes_\CC Y \to L(\one \otimes_\CC Y) \simeq LY$.
\end{proof}

The Proposition above is useful in settings where every localisation is tensor-compatible, like the category of spectra. 

\begin{prop}[\cite{axiomatic} Lemma 1.4.6]\label{prop: when all locs are monoidal}
    If $\CC$ is presentably symmetric monoidal, stable, and generated by the monoidal unit under colimits and desuspensions then every localisation is tensor-compatible. More generally, if $\CC$ is compactly generated, stable, and presentably symmetric monoidal then every localisation with the property that an acyclic object tensored with a compact object remains acyclic is tensor-compatible.
\end{prop}

\begin{proof}
    Since $\CC$ is stable, $L$ is tensor-compatible precisely if the acyclic objects form a tensor ideal, see Lemma \ref{lem: equiv monoidal loc}. The acyclics are closed under colimits since $L$ is a left adjoint, and under desuspensions by our assumption on localisations of stable categories. Fix an acyclic object $A$. Since $- \otimes A$ preserves colimits and desuspensions, the full subcategory $\AA$ of objects $X$ such that $X \otimes A$ is acyclic is closed under colimits and desuspensions. Since $\one \in \AA$, we find that $\AA = \CC$, so the acyclics form a tensor ideal.

    Now suppose that $\CC$ is compactly generated, and that $Y \otimes A$ is acyclic for every acyclic $A$ and every compact object $Y$. Every object $X$ of $\CC$ is a filtered colimit of compact objects $Y_i$, so $X \otimes A$ is the filtered colimit of the objects $Y_i \otimes A$, each of which is acyclic by assumption. Since the acyclics are closed under colimits, $X \otimes A$ is acyclic, so again the acyclics form a tensor ideal.
\end{proof}

The following Lemma is a version of Proposition \ref{prop: smashing loc modules over unit} for tensor-compatible localisations. Unlike in the case of a smashing localisation, being a module over the local unit does not guarantee that an object is local.

\begin{lem}[\cite{axiomatic} Proposition 3.1.8]\label{lem: monoidal loc modules over unit}
    If $L$ is a tensor-compatible localisation then every local object is canonically an $L\one$-module in $\CC$.
\end{lem}

\begin{proof}
    Each $X \in \CC$ is canonically a $\one$-module, via the natural equivalence $\one \otimes X \xrightarrow{\simeq} X$. Localising, we obtain an equivalence $L(\one \otimes_\CC X) \xrightarrow{\simeq} LX$. Since $L$ is tensor-compatible, we have natural equivalences $L(\one\otimes_\CC X) \xrightarrow{\simeq} L(L\one \otimes_\CC X) \xrightarrow{\simeq} L(L\one \otimes_\CC LX)$ and thus we get a natural equivalence $L\one \otimes_{L\CC} LX = L(L\one \otimes_\CC LX) \xrightarrow{\simeq} LX$. This makes $LX$ into an $L\one$-module in $L\CC$. Since the inclusion $L\CC \into \CC$ is lax monoidal, $LX$ also obtains the structure of an $L\one$-module in $\CC$, with action $L\one \otimes_\CC LX \to L(L\one \otimes_\CC LX) \simeq LX$. We can alternatively think of this as the proof that $L\one$ is the tensor unit for the induced monoidal structure on $L\CC$, when $L$ is tensor-compatible.
\end{proof}

\subsection{Combining smashing localisations}\label{subsec: combining smash locs}

Next we discuss several ways that smashing localisations can be combined to build new smashing localisations. We will eventually use this to produce a right adjoint to the forgetful functor from smashing localisations to localisations. This is some kind of free approximation to a given localisation by a smashing localisation. 

\begin{notation}
Let $L_1 \otimes L_2$ denote the smashing localisation whose idempotent algebra is $L_1\one \otimes_\CC L_2\one$, whenever $L_1$ and $L_2$ are both smashing localisations of $\CC$.
\end{notation}

Note that $L_1\one \otimes_\CC L_2\one$ is automatically an idempotent algebra in $\CC$ because $L_1\one$ and $L_2 \one$ are. Its idempotent unit map $\one \to L_1\one \otimes_\CC L_2\one$ is simply the tensor product of the unit maps of $L_1\one$ and $L_2\one$.

\begin{lem}\label{lem: combining smashing locs}
    Let $L_1$ and $L_2$ be smashing localisations of $\CC$. The local objects for $L_1 \otimes L_2$ are the intersection of the local objects for $L_1$ with the local objects for $L_2$. The acyclics for $L_1 \otimes L_2$ are generated by the union of the acyclics for $L_1$ with the acyclics for $L_2$. It follows that $L_2 \circ L_1 \simeq L_1 \otimes L_2 \simeq L_1 \circ L_2$ when we think of all three as endofunctors of $\CC$. 

    We also obtain smashing localisations $L_1|_{L_2\CC}$ and $L_2|_{L_1\CC}$ computed by tensoring with $L_1\one$ and $L_2\one$ respectively. An $L_1$-local object remains $L_1$-local upon $L_2$-localising, and vice versa. If $L_1$ is finite then $L_1|_{L_2\CC}$ is also finite. 
\end{lem}

\begin{proof}
    It is clear that $L_1\one \otimes L_2\one$ is idempotent and thus defines a smashing localisation. Then a local object for $L_1 \otimes L_2$ has the form $L_1\one \otimes L_2\one \otimes X$, which is evidently local for both $L_1$ and $L_2$, using idempotence of $L_1\one$ and $L_2\one$ respectively. Moreover, any object $Y$ which is local for both $L_1$ and $L_2$ satisfies 
    \[L_1\one \otimes L_2\one \otimes Y \simeq L_1 \circ (L_2 Y) \simeq L_1 Y \simeq Y\]
    and is thus local for $L_1 \otimes L_2$. This proves that $L_1 \otimes L_2(\CC) = L_1 \CC \cap L_2 \CC$. The result for acyclics follows from the result for local objects by recalling that the acyclics are left-orthogonal to the locals.

    Since all three localisations are smashing we have 
    \[L_1 \circ L_2 \simeq L_1\one \otimes L_2\one \otimes - \simeq L_1 \otimes L_2 \simeq L_2\one \otimes L_1\one \otimes - \simeq L_2 \circ L_1\]
    by symmetry of the monoidal structure on $\CC$.

    Assume $L_1$ is finite. Let $X$ be any acyclic for $L_1|_{L_2\CC}$, so $X$ is $L_2$-local and $L_1$-acyclic. Then $X$ can be written as a colimit of compact $L_1$-acyclics, $X \simeq \colim_{i \in I} X_i$, since $L_1$ is finite. Each $X_i$ is compact in $\CC$. $L_2$-localisation preserves colimits because it is smashing, so 
    \[X \simeq L_2 X \simeq \colim_{i \in I} L_2 X_i\]
    and we know $L_1 L_2 X_i \simeq L_2 L_1 X_i \simeq 0$ so $L_2 X_i$ is $L_1$-acyclic for $i \in I$. Moreover, each $L_2 X_i$ is compact in $L_2 \CC$ because smashing localisations send compacts to compacts, although $L_2 X_i$ may not be compact when viewed as an object of $\CC$ via the subcategory inclusion. Thus $X$ is a colimit of compact $L_1|_{L_2\CC}$-acyclics, so $L_1|_{L_2\CC}$ is finite.
\end{proof}

Lemma \ref{lem: combining smashing locs} can be summarised as the following commuting diagram of five smashing localisations, labelled by their idempotent algebras. 
\[\begin{tikzcd}[sep={6em,between origins}]
& L_1\CC \arrow[rd, "L_2\one"] &                     \\
\CC \arrow[rr, "L_1\one \otimes L_2\one"] \arrow[ru, "L_1\one"] \arrow[rd, "L_2\one"'] &                          & L_1\CC \cap L_2 \CC \\
& L_2\CC \arrow[ru, "L_1\one"'] &                    
\end{tikzcd}\]

If $L: \CC \to \DD$ preserves small colimits, the adjoint functor theorem tells us it has a right adjoint $R$. If $L$ is also symmetric monoidal, its right adjoint is lax symmetric monoidal, and hence both functors map $\E_n$-algebras to $\E_n$-algebras, and morphisms of algebras to morphisms of algebras. We thus obtain induced functors between $\E_n$-algebras in $\CC$ and $\E_n$-algebras in $\mathscr{D}$, but it is not immediately clear that these functors again form an adjunction. We have the following result.

\begin{prop}[\cite{ha} Corollary 7.3.2.7]\label{prop: En adjunction for locs}
    Let $L: \CC \to \mathscr{D}$ be a functor between presentably symmetric monoidal $\infty$-categories which is symmetric monoidal and preserves small colimits. Then the adjunction $L \dashv R$ induces an adjunction between $\E_n$-algebras in $\CC$ and $\E_n$-algebras in $\mathscr{D}$. On the level of underlying objects and morphisms, this new adjunction agrees with $L \dashv R$.
\end{prop}

\begin{proof}
    Here is a helpful reference diagram which summarises the functors involved in this proof. It contains three known adjunctions, and we are trying to show that the fourth pair of functors $(L', R')$ is also an adjunction. The two squares involving $L, L'$ both commute, and the square involving $R$ and $\oblv$ commutes, but the square involving $R$ and $\Free$ does not necessarily commute. 

    \[\begin{tikzcd}[sep={8em,between origins}]
	\CC & \DD \\
	{\Alg_{\E_n}(\CC)} & {\Alg_{\E_n}(\DD)}
	\arrow["L"{name=L}, shift left=1.5, from=1-1, to=1-2]
    \arrow["R"{name=R}, shift left=1.5, from=1-2, to=1-1]
    \arrow[from=L, to=R, phantom, "\dashv"{font=\tiny, rotate=-90}]
	\arrow["{\Free_\CC}"'{name=freeC}, shift right=1.5, from=1-1, to=2-1]
    \arrow["{\oblv_\CC}"'{name=oblvC}, shift right=1.5, from=2-1, to=1-1]
    \arrow[from=freeC, to=oblvC, phantom, "\dashv"{font=\tiny}]
	\arrow["{\Free_\DD}"'{name=freeD}, shift right=1.5, from=1-2, to=2-2]
    \arrow["{\oblv_\DD}"'{name=oblvD}, shift right=1.5, from=2-2, to=1-2]
    \arrow[from=freeD, to=oblvD, phantom, "\dashv"{font=\tiny}]
	\arrow["{L'}", shift left=1.5, from=2-1, to=2-2]
	\arrow["{R'}", shift left=1.5, from=2-2, to=2-1]
\end{tikzcd}\]

    Recall there is an adjunction 
    \[\begin{tikzcd}[column sep={huge}]
	    {\Free_\CC: \CC} & {\Alg_{\E_n}(\CC): \oblv_\CC,}
	    \arrow[""{name=right},shift left=1.5, from=1-1, to=1-2]
	    \arrow[""{name=left}, shift left=1.5, from=1-2, to=1-1]
        \arrow[from=left, to=right, phantom, "\dashv"{font=\tiny, rotate=-90}]
    \end{tikzcd}\]
    where $\oblv_\CC$ is the forgetful functor mapping an $\E_n$-algebra to its underlying object. We call an object in the essential image of $\Free_\CC$ a \emph{free $\E_n$-algebra}. The forgetful functor $\oblv_\CC$ is conservative, preserves sifted colimits, and creates small limits. That is, a diagram in $\Alg_{\E_n}(\CC)$ has a limit if and only if its image under $\oblv_\CC$ has a limit, and moreover $\oblv_\CC$ preserves the limit (see \cite{ha} Corollaries 3.2.2.5 and 3.2.3.2). 
    
    Because both $L$ and $R$ are lax symmetric monoidal, they induce functors between $\E_n$-algebras in $\CC$ and $\E_n$-algebras in $\mathscr{D}$. We get
    \[L': \Alg_{\E_n}(\CC) \to \Alg_{\E_n}(\mathscr{D}) \quad \text{and} \quad R': \Alg_{\E_n}(\mathscr{D}) \to \Alg_{\E_n}(\CC).\]
    These induced functors commute with the forgetful functor by definition, i.e. they come equipped with natural equivalences
    \[L \circ \oblv_\CC \simeq \oblv_\DD \circ L' \quad \text{and} \quad \oblv_\CC \circ R' \simeq R \circ \oblv_\DD.\]
    Since $L$ is symmetric monoidal, it also commutes with the free $\E_n$-algebra functor in the sense that there is a natural equivalence $L' \circ \Free_\CC \simeq \Free_\DD \circ L$, compatible with the one for $\oblv$. Note that because $R$ need not be symmetric monoidal, it is not necessarily the case that $R$ commutes with $\Free$. 

    Using the adjunctions $\Free_\CC \dashv \oblv_\CC$, $\Free_\DD \dashv \oblv_\DD$, and $L \dashv R$, as well as the fact that $L$ commutes with $\Free$ and $R$ commutes with $\oblv$, we can compute that $L'$ and $R'$ behave as an adjoint pair when $L'$ is restricted to free algebras. Explicitly, we have a natural equivalence between two mapping space functors on $\CC^{\op} \times \Alg_{\E_n}(\DD)$. Let $A \in \CC, Y \in \Alg_{\E_n}(\DD)$ and $X = \Free_\CC(A) \in \Alg_{\E_n}(\CC)$. Then
    \begin{align*}
        \Map_{\Alg_{\E_n}(\CC)}(X, R'Y) &\simeq \Map_{\Alg_{\E_n}(\CC)}(\Free_\CC(A), R'Y)\\
        &\simeq \Map_{\CC}(A, \oblv_\CC R'Y)\\
        &\simeq \Map_{\CC}(A, R \oblv_\DD Y)\\
        &\simeq \Map_{\DD}(LA, \oblv_\DD Y)\\
        &\simeq \Map_{\Alg_{\E_n}(\DD)}(\Free_\DD LA, Y)\\
        &\simeq \Map_{\Alg_{\E_n}(\DD)}(L' \Free_\CC A, Y)\\
        &\simeq \Map_{\Alg_{\E_n}(\DD)}(L'X, Y).
    \end{align*}
    The composite equivalence 
    \[\Map_{\Alg_{\E_n}(\CC)}(\Free_\CC(A), R'Y) \simeq \Map_{\Alg_{\E_n}(\DD)}(L'\Free_\CC(A), Y),\]
    is natural in $\CC^{\op} \times \Alg_{\E_n}(\DD)$. 

    Any $Z \in \Alg_{\E_n}(\CC)$ can be expressed as a sifted colimit of free $\E_n$-algebras, and this can be done functorially. Write $Z \simeq \colim_m \Free_\CC A_m$ and we obtain a natural equivalence 
    \begin{align*}
    \Map_{\Alg_{\E_n}(\CC)}(Z, R'Y) &\simeq \Map_{\Alg_{\E_n}(\CC)}(\colim_m \Free_\CC A_m, R'Y)\\ 
    &\simeq \lim_m \Map_{\Alg_{\E_n}(\CC)}(\Free_\CC A_m, R'Y)\\
    &\simeq \lim_m \Map_{\Alg_{\E_n}(\DD)}(L' \Free_\CC A_m, Y)\\
    &\simeq \Map_{\Alg_{\E_n}(\DD)}(\colim_m L' \Free_\CC A_m, Y)\\
    &\simeq \Map_{\Alg_{\E_n}(\DD)}(L' \colim_m \Free_\CC A_m, Y)\\
    &\simeq \Map_{\Alg_{\E_n}(\DD)}(L'Z, Y).
    \end{align*}
    In the second-last step we used that $L'$ commutes with sifted colimits, and we wrote $Z$ as a sifted colimit of free algebras. To see that $L'$ commutes with sifted colimits, we simply note that $L'$ commutes with $\oblv$ and both $L$ and $\oblv$ preserve sifted colimits. Indeed $L$ preserves all small colimits because it is a left adjoint. For a sifted colimit $\colim_m X_m \in \Alg_{\E_n}(\CC)$,
    \begin{align*}
        \oblv_\DD L' \colim_m X_m &\simeq L \oblv_\CC \colim_m X_m\\
        &\simeq \colim_m L \oblv_\CC X_m\\
        &\simeq \colim_m \oblv_\DD L' X_m\\
        &\simeq \oblv_\DD \colim_m L' X_m,
    \end{align*}
    and since $\oblv_\DD$ is conservative we conclude that $L' \colim_m X_m \simeq \colim_m L' X_m$. All of this is appropriately natural. 

    Since we have established an equivalence 
    \[\Map_{\Alg_{\E_n}(\CC)}(Z, R'Y) \simeq \Map_{\Alg_{\E_n}(\DD)}(L'Z, Y),\]
    natural in $\Alg_{\E_n}(\CC)^{\op} \times \Alg_{\E_n}(\DD)$, we have shown that $L' \dashv R'$ as desired.
\end{proof}

\begin{lem}[cf \cite{ambidexterity} Propositions 5.1.6 and 5.1.7]\label{lem: En refinement}
    Let $R_2 \in \CAlg(\CC)$ be an idempotent algebra and $R_1 \in \Alg_{\E_n}(\CC)$ for $n \geq 1$. Let $L$ be the smashing localisation of $\CC$ computed by tensoring with $R_2$. Then $\Map_{\E_n}(R_2, R_1) \simeq \Map_{\E_0}(R_2, R_1)$ via forgetting the $\E_n$ structure. This mapping space is contractible if and only if $R_1$ is $L$-local, and otherwise it is empty.
\end{lem}

\begin{proof}
    Since $R_2$ is an idempotent algebra, its unit map $e_2: \one \to R_2$ determines a smashing localisation which is a symmetric monoidal functor $L: \CC \to L\CC$, with right adjoint the subcategory inclusion (Proposition \ref{prop: class of smashing locs} and Lemma \ref{lem: monoidal loc is monoidal}). The local objects for this localisation are precisely the $R_2$-modules in $\CC$, by Proposition \ref{prop: smashing loc modules over unit}. 

    To prove the result, we will establish the following two points.
    \begin{itemize}
        \item If $R_1$ is $L$-local then $\Map_{\E_n}(R_2, R_1) \simeq \Map_{\E_0}(R_2, R_1) \simeq *$.
        \item $\Map_{\E_0}(R_2, R_1)$ is nonempty if and only if $R_1$ is $L$-local.
    \end{itemize}

    This is sufficient because if $\Map_{\E_0}(R_2, R_1)$ is nonempty then both mapping spaces are contractible and $R_1$ is $L$-local, and if $\Map_{\E_0}(R_2, R_1)$ is empty then $R_1$ is not local and the forgetful map $\Map_{\E_n}(R_2, R_1) \to \Map_{\E_0}(R_2, R_1)$ tells us $\Map_{\E_n}(R_2, R_1)$ must also be empty.

    For the first point, we can build a commuting diagram 
    \[\begin{tikzcd}[sep={10em,between origins}]
    {\Map_{\En}(L\one, LR_1)} \arrow[r] \arrow[d, "\simeq"'] & {\Map_{\E_0}(L\one, LR_1)} \arrow[d, "\simeq"] \\
    {\Map_{\E_n}(\one, LR_1)} \arrow[r, "\simeq"]   & {\Map_{\E_0}(\one, LR_1)}                     
    \end{tikzcd}\]
    whose horizontal maps are forgetting the $\E_n$ structure, and vertical maps are guaranteed by the adjunction of Proposition \ref{prop: En adjunction for locs} to be equivalences. The two terms in the bottom row are clearly contractible, as there is a contractible space of $\E_n$ maps from the unit to any $\E_n$-ring, and both $L$ and its adjoint are lax symmetric monoidal, so $LR_1$ is an $\E_n$-algebra in $\CC$.

    As a consequence, the top horizontal map is an equivalence between contractible spaces. If $R_1$ is local, then the canonical map $R_1 \to LR_1$ is an equivalence in $\CC$ and so we can extend this diagram to 
    \[\begin{tikzcd}[sep={10em,between origins}]
        {\Map_{\E_n}(R_2, R_1)} \arrow[d, "\simeq"'] \arrow[r] & {\Map_{\E_0}(R_2, R_1)} \arrow[d, "\simeq"] \\
        {\Map_{\En}(R_2, LR_1)} \arrow[r, "\simeq"]                  & {\Map_{\E_0}(R_2, LR_1).}                 
        \end{tikzcd}\]
        Recall that since $L\one \simeq R_2$, the bottom horizontal map is the same one we just proved was an equivalence, hence the top horizontal map must be an equivalence, between contractible spaces. This proves the first point.

        Now let $f: R_2 \to R_1$ be $\E_0$. To show that $R_1$ is $L$-local, we will check that the canonical map $e_2 \otimes \id_{R_1}: R_1 \to R_2 \otimes R_1 \simeq LR_1$ is an equivalence by providing its inverse. The inverse is given by $m_1 \circ (f \otimes \id_{R_1})$, where $m_i: R_i \otimes R_i \to R_i$ denotes the multiplication. We know $f \circ e_2 \simeq e_1$ because $f$ is $\E_0$, so 
        \[m_1 \circ (f \otimes \id_{R_1}) \circ (e_2 \otimes \id_{R_1}) \simeq m_1 \circ ((f \circ e_2) \otimes \id_{R_1}) \simeq m_1 \circ (e_1 \otimes \id_{R_1}) \simeq \id_{R_1},\]
        because the multiplication on $R_1$ is unital. This computation is summarised in the diagram
        \[\begin{tikzcd}[sep={10em,between origins}]
            R_1 \arrow[r, "e_2 \otimes \id_{R_1} "] 
                \arrow[d, "\id_{R_1}"'] 
                \arrow[rd, "e_1 \otimes \id_{R_1}"] 
            & R_2 \otimes R_1 \arrow[d, "f \otimes \id_{R_1}"] \\
            R_1 & R_1 \otimes R_1. \arrow[l, "m_1"]          
            \end{tikzcd}\]
        For the other composition, let $g = m_1 \circ (f \otimes \id_{R_1})$ as shorthand. Rewrite 
        \[(e_2 \otimes \id_{R_1}) \circ g \simeq e_2 \otimes g \simeq (\id_{R_2} \otimes g) \circ (e_2 \otimes \id_{R_2} \otimes \id_{R_1}),\] 
        and now observe that $e_2 \otimes \id_{R_2} \simeq \id_{R_2} \otimes e_2$ because both are inverse to the multiplication map $m_2$ on the idempotent algebra $R_2$. This means 
        \[(e_2 \otimes \id_{R_1}) \circ g \simeq 
        (\id_{R_2} \otimes g) \circ (\id_{R_2} \otimes e_2 \otimes \id_{R_1}) \simeq \id_{R_2} \otimes (g \circ (e_2 \otimes \id_{R_1})) \simeq \id_{R_2 \otimes R_1},\]
        by reducing to the composition we already checked. The upshot is that $g$ and $e_2 \otimes \id_{R_1}$ are mutual inverses, so $R_1$ is local. Moreover, checking the second composition did not depend on anything about the maps involved -- since $R_2$ was idempotent, once we checked the first composite the second followed purely formally. We will state this explicitly as Lemma \ref{lem: key iso} later. 

        Finally, if $R_1$ is local then by the universal property of localisation the $\E_n$ unit map $\one \to R_1$ factors through $L\one \simeq R_2$ to give an $\E_0$ map $R_2 \to R_1$.
\end{proof}

The next result collects some easy but useful consequences of Lemma \ref{lem: En refinement} for future reference.

\begin{cor}\label{cor: very useful En refinement}
    Let $R_2 \in \CAlg(\CC)$ be idempotent and $L$ denote the localisation of $\CC$ given by tensoring with $R_2$. Let $R_1$ be another object of $\CC$.
    \begin{enumerate}[label=(\arabic*)]
        \item $\Map_{\E_0}(R_2, R_2) \simeq \Map_{\E_\infty}(R_2, R_2) \simeq *$. \label{cor sub: end of idem}
        \item If $R_1$ is $\E_n$ and $L$-local then $\Map_{\E_m}(R_2, R_1) \simeq *$ for all $0 \leq m \leq n$. \label{cor sub: intermediate n}
        \item $R_2$ has a canonical $\E_n$ structure for every $n$, induced by forgetting from its canonical $\E_\infty$ structure which comes from idempotence. This is the unique unital $\E_n$ structure on $R_2$. \label{cor sub: unique En str on idem}
    \end{enumerate}
\end{cor}

\begin{proof}
    First note that \ref{cor sub: end of idem} is a special case of Lemma \ref{lem: En refinement}. For \ref{cor sub: intermediate n}, observe that an $\E_n$-algebra has an induced $\E_m$ structure for $0 \leq m \leq n$ by forgetting. For \ref{cor sub: unique En str on idem}, the identity endomorphism on $R_2$ can be thought of as a map from $R_2$ with the canonical $\E_\infty$ structure induced by idempotence to $R_2$ with some arbitrary unital $\E_n$ structure. This map is $\E_0$, and hence refines to a map of $\E_n$-algebras. Thus the two $\E_n$ structures on $R_2$ agree. This point also holds with $n = \infty$, and explains why the $\E_\infty$ structure on $R_2$ is canonical. 
\end{proof}

   There is a partial order on localisations of a fixed category $\CC$ given by inclusion of subcategories of local objects, or equivalently reverse inclusion of subcategories of acyclic objects: we say $L \leq L'$ if $\ker L \supseteq \ker L'$. If $L \leq L'$ then every $L$-local object is $L'$-local, so $L$ factors through $L'$ as $L \simeq L|_{L'\CC} \circ L'$, where $L|_{L'\CC}$ is a localisation of the category of $L'$-local objects. In particular there is a natural transformation $L' \implies L$. For an introduction to the lattice of localisations of $\Sp$, see \cite{bousfield} and \cite{latticestructure}. We next discuss special properties of the partial order on localisations when restricted to tensor-compatible localisations.

\begin{lem}[cf \cite{axiomatic} Lemma 3.1.5]\label{lem: smash loc partial order via ring map}
    If $L_1$ and $L_2$ are localisations satisfying $L_1 \leq L_2$ then there is an $L_1$-local equivalence $L_2 \one \to L_1 \one$. If $L_1$ and $L_2$ are tensor-compatible then this refines to a map of $\E_\infty$-algebras in $\CC$. If in addition $L_2$ is smashing then the converse holds: if there exists an $\E_\infty$ map $L_2\one \to L_1\one$ then $L_1 \leq L_2$.
\end{lem}

We will most commonly use this result in the case when $L_2$ is smashing and $L_1$ is tensor-compatible. Then the condition $L_1 \leq L_2$ is equivalent to supplying an $\E_\infty$-ring map $\phi: L_2\one \to L_1\one$. Such a map $\phi$ is unique, and automatically an $L_1$-local equivalence.

\begin{proof}
    Since $L_1 \leq L_2$, we know $L_1$-localisation factors through $L_2$-localisation. The comparison map $L_2\one \to L_1\one$ is the component $\eta^{L_1}_{L_2\one}: L_2\one \to L_1L_2\one \simeq L_1\one$ of the unit of the adjunction for $L_1$. In particular $\eta^{L_1}_{L_2\one}$ is an $L_1$-local equivalence. If both localisations are tensor-compatible then they are lax monoidal endofunctors of $\CC$ (Lemma \ref{lem: monoidal loc is monoidal}), so $L_1\one$ and $L_2\one$ inherit $\E_\infty$ structures in $\CC$ from $\one$. Since the units of the adjunctions for $L_1$ and $L_2$ are monoidal natural transformations, $\eta^{L_1}_{L_2\one}$ is a map of $\E_\infty$-algebras in $\CC$, as is the identification $L_1L_2\one \simeq L_1\one$.

    Conversely, suppose $L_2$ is smashing and $\phi: L_2\one \to L_1\one$ is a map of $\E_\infty$-algebras in $\CC$. Every $L_1$-local object is an $L_1\one$-module by Lemma \ref{lem: monoidal loc modules over unit}, hence an $L_2\one$-module by restriction along $\phi$, and thus $L_2$-local by Proposition \ref{prop: smashing loc modules over unit}. So $L_1 \leq L_2$, since we have inclusion of local subcategories.

    Finally, $L_2\one$ is idempotent and $L_1\one$ is $L_2$-local by the previous paragraph. Hence Lemma \ref{lem: En refinement} tells us the space of $\E_0$ maps and the space of $\E_\infty$ maps $L_2\one \to L_1\one$ are both contractible. Since $\phi$ and $\eta^{L_1}_{L_2\one}$ are in these spaces, $\phi \simeq \eta^{L_1}_{L_2\one}$. Thus $\phi$ is an $L_1$-local equivalence, unique with the stated properties.
\end{proof}

\subsection{Central maps}\label{subsec: central maps}

In this Subsection we describe another way to build idempotent algebras, and thus smashing localisations, by starting with a map which is not idempotent but instead satisfies the weaker condition of \emph{centrality} (Definition \ref{def: central}). The free $\E_1$-algebra on a central map turns out to be idempotent, which we prove as Corollary \ref{cor: Jalpha is idempotent}. This perspective on smashing localisations suggests a natural notion of finiteness, which we will compare to the usual definition in Subsection \ref{subsec: finiteness}.

We are going to give two different constructions of the object $J_\alpha$ of Construction \ref{constr: Jalpha}. The second is an explicit colimit in $\CC$, appearing as Construction \ref{constr: explicit Jalpha}. We will then show that the two constructions agree (Proposition \ref{prop: Jalpha Jalpha' agree}), provided that $\CC$ is stable and the input map $\alpha$ has the special property of being central.

\begin{constr}\label{constr: Jalpha}
Let $A\in \CC$ have an $\E_0$-algebra structure given by a map $\alpha: \one \to A$. The forgetful functor $\Oblv^{\E_1}_{\E_0}: \Alg_{\E_1}(\CC) \to \Alg_{\E_0}(\CC)$ has a left adjoint, denoted $\Free_{\E_0}^{\E_1}$. Formally, $\Oblv^{\E_1}_{\E_0}$ is restriction along the map of $\infty$-operads $\E_0 \to \E_1$, so it admits a left adjoint given by operadic left Kan extension (\cite{ha} Corollary 3.1.3.5). Let $J_\alpha$ denote the free $\E_1$-algebra on the $\E_0$-algebra $(A, \alpha)$, i.e. $J_\alpha := \Free_{\E_0}^{\E_1}(A, \alpha)$.

We can unpack the definition of $J_\alpha$ into a universal property. Let $\overline{\alpha}$ denote the unit map of the $\E_1$-algebra structure on $J_\alpha$, which can be obtained as the $\E_0$-algebra structure on $\Oblv^{\E_1}_{\E_0} J_\alpha = \Oblv^{\E_1}_{\E_0} \Free_{\E_0}^{\E_1} (A, \alpha)$. Then $\overline{\alpha}$ factors through $\alpha$ as $c \circ \alpha \simeq \overline{\alpha}$. The map of $\E_0$-algebras $c: A \to J_\alpha$ appearing in this composition is obtained by applying the adjunction $\Free \dashv \Oblv$ to the identity map on $J_\alpha$. Moreover, if $R \in \Alg_{\E_1}(\CC)$ and $f: (A, \alpha) \to \Oblv^{\E_1}_{\E_0}(R)$ is a map of $\E_0$-algebras, then $f \simeq \Oblv^{\E_1}_{\E_0}(\tilde f) \circ c$ for a unique map of $\E_1$-algebras $\tilde f: J_\alpha \to R$. This universal property can be summarised in the following diagram, where $\tilde{f}$ is a map of $\E_1$-algebras and $f, c$ are maps of $\E_0$-algebras. Each map with domain $\one$ in this diagram encodes the $\E_0$-algebra structure of its codomain.
\[\begin{tikzcd}[sep={6em,between origins}]
    \one \arrow[r, "\alpha"] \arrow[rd] & A \arrow[d, "f"] \arrow[r, "c"] & J_\alpha \arrow[ld, "\exists ! \tilde{f}", dashed] \\
                                                & R                          &                                         
    \end{tikzcd}\]
\end{constr}

One can think of $J_\alpha$ as the universal way to build an $\E_1$-algebra for which $\alpha$ is a unit. We choose to emphasise the map $\alpha$ in the notation $J_\alpha$ because the same object $A \in \CC$, when given two different $\E_0$-algebra structures, can produce very different resulting $\E_1$-algebras. As an example, consider the two maps $\id: \S \to \S$ and $p: \S \to \S$, the second of which is the multiplication-by-$p$ map for some prime $p \in \Z$. Then $J_{\id}$ simply recovers the $\E_1$ structure on the sphere, obtained by forgetting from the usual $\E_\infty$ structure. The multiplication-by-$p$ map is central (see Definition \ref{def: central}) and so Proposition \ref{prop: Jalpha Jalpha' agree} and Construction \ref{constr: explicit Jalpha} together allow us to compute that $J_p \simeq \S[1/p]$, the Moore spectrum of $\Z[1/p]$. Thus $J_p$ and $J_{\id}$ do not even have the same underlying object.

\begin{defn}\label{def: central}
A map $\alpha: \one \to A$ is \emph{central} if 
\[\alpha \otimes \id_A \simeq \id_A \otimes \alpha :A \to A \otimes A,\] 
i.e. these two maps are homotopic.
\end{defn}

We are not requiring that either of the two maps itself be an equivalence, just that the maps agree with one another. Any idempotent map is central (see Definition \ref{def: idempotent}), but the converse does not hold in general.

In Corollary \ref{cor: Jalpha is idempotent} we establish that $J_\alpha$ is idempotent, but before we can do so we need a few technical Lemmas.
    
    \begin{lem}\label{lem: key iso}
        Let $\alpha$ be central, and assume we are given a unital map $\mu: A \otimes M \to M$, meaning that the diagram 
        \[\begin{tikzcd}[sep={6em,between origins}]
            M \arrow[r, "\alpha \otimes \id_M"] \arrow[rd, "\id_M"'] & A \otimes M \arrow[d, "\mu"] \\
                                                                     & M                           
            \end{tikzcd}\]
        commutes. Then $\alpha \otimes \id_M$ and $\mu$ are inverse equivalences.
    \end{lem}

    Lemma \ref{lem: key iso} is useful because it allows us to check that a pair of maps are mutually inverse by computing only one of the two compositions. 

    \begin{proof}
        We need only check that the other composition $A \otimes M \xrightarrow{\mu} M \xrightarrow{\alpha \otimes \id_M} A \otimes M$ is the identity. Consider the diagram
        \[\begin{tikzcd}[sep={9em,between origins}]
            A \otimes M \arrow[r, "\mu"] \arrow[d, "\alpha \otimes \id_{A \otimes M}", bend left] \arrow[d, "\id_A\otimes \alpha \otimes \id_M"', bend right] & M \arrow[d, "\alpha \otimes \id_M"] \\
            A \otimes A \otimes M \arrow[r, "\id_A \otimes \mu"']                                                                                             & A \otimes M                        
            \end{tikzcd}\]
        which commutes because the two vertical maps $A \otimes M \to A \otimes A \otimes M$ are homotopic by centrality of $\alpha$. The composite around the top of the diagram is the one we are interested in, and around the bottom is $\id_A \otimes (\mu \circ (\alpha \otimes \id_M)) \simeq \id_A \otimes \id_M$ by assumption. 
    \end{proof}

    \begin{cor}\label{cor: one more iso}
        If $\alpha$ is central then the map $\alpha \otimes \id_{J_\alpha}: J_\alpha \to A \otimes J_\alpha$ is an equivalence. 
    \end{cor}

    By symmetry, it is clear that $\id_{J_\alpha} \otimes \alpha$ must also be an equivalence.

    \begin{proof}
	    It is enough by Lemma \ref{lem: key iso} to provide a map $A \otimes J_\alpha \to J_\alpha$ and check that the composite
        $J_\alpha \xrightarrow{\alpha \otimes \id_{J_\alpha}} A \otimes J_\alpha \longrightarrow J_\alpha$
        is $\id_{J_\alpha}$. Consider the diagram 
        \[\begin{tikzcd}[sep={7em,between origins}]
            J_\alpha    \arrow[d, "\alpha \otimes \id_{J_\alpha}"'] 
                        \arrow[rd, "\overline{\alpha} \otimes \id_{J_\alpha}"'] 
                        \arrow[rrd, "\id_{J_\alpha}"] \\
            A \otimes J_\alpha \arrow[r, "c \otimes \id_{J_\alpha}"']
            & J_\alpha \otimes J_\alpha \arrow[r, "\mu"'] 
            & J_\alpha
            \end{tikzcd}\]
        where $\mu$ is the multiplication map in $J_\alpha$, coming from its $\E_1$-ring structure. This commutes because $\overline{\alpha}$ is the unit in $J_\alpha$. 
    \end{proof}

    \begin{cor}\label{cor: maps into Jalpha modules}
        Suppose $\alpha$ is central and let $N$ be a $J_\alpha$-module in $\CC$. Restriction along $\alpha$ provides an equivalence $\Map_\CC(A, N) \to \Map_\CC(\one, N)$.
    \end{cor}

    \begin{proof}
        By Corollary \ref{cor: one more iso}, $J_\alpha \otimes A \simeq J_\alpha$. This equivalence holds in $\CC$ and also as $J_\alpha$-modules. For modules over an $\E_1$-algebra there is a free-forgetful adjunction, where the forgetful functor has a left adjoint of the form $X \mapsto J_\alpha \otimes X$ for $X \in \CC$ (\cite{ha} Corollary 4.2.4.8). We apply this adjunction result twice to obtain
        \[\Map_\CC(A, N) \simeq \Map_{\Mod_{J_\alpha}(\CC)}(J_\alpha \otimes A, N) \xrightarrow{\id_{J_\alpha} \otimes \alpha, \simeq} \Map_{\Mod_{J_\alpha}(\CC)}(J_\alpha, N) \simeq \Map_\CC(\one, N).\]
        The composite equivalence $\Map_\CC(A, N) \xrightarrow{\simeq} \Map_\CC(\one, N)$ is precomposition with $\alpha$. 
    \end{proof}

    \begin{cor}\label{cor: still central}
        If $\alpha$ is central, so is the induced unit map $\overline{\alpha}: \one \to J_\alpha$.
    \end{cor}

    \begin{proof}
        By the universal property of $J_\alpha$, the maps $\id_{J_\alpha} \otimes \overline{\alpha}$ and $\overline{\alpha} \otimes \id_{J_\alpha}$ are the unique elements of $\Map_{\E_1}(J_\alpha, J_\alpha \otimes J_\alpha)$ extending the $\E_0$ maps
        \begin{align*}
            c \otimes \overline{\alpha}: A &\xrightarrow{\id_A \otimes \alpha} A \otimes A \longrightarrow J_\alpha \otimes J_\alpha,\\ 
            \overline{\alpha} \otimes c: A &\xrightarrow{\alpha \otimes \id_A} A \otimes A \longrightarrow J_\alpha \otimes J_\alpha
        \end{align*}
        respectively. It therefore suffices to provide a homotopy $c \otimes \overline{\alpha} \simeq \overline{\alpha} \otimes c$ in $\Map_{\E_0}(A, J_\alpha \otimes J_\alpha)$.

        It follows from Corollary \ref{cor: maps into Jalpha modules} that $\Map_{\E_0}(A, J_\alpha \otimes J_\alpha) \simeq *$, so existence of such a homotopy is immediate. Explicitly, the pullback square
        \[
        \begin{tikzcd}[sep={8em, between origins}]
            \Map_{\E_0}(A, N) \arrow[r] \arrow[d] & * \arrow[d] \\
            \Map_{\CC}(A, N) \arrow[r, "\simeq"] & \Map_{\CC}(\one, N)
        \end{tikzcd}
        \]
        computes $\Map_{\E_0}(A, N)$. Since the bottom horizontal map is an equivalence, so is the top map, and $\Map_{\E_0}(A, N) \simeq *$ for every $J_\alpha$-module $N$. It follows that any two unital maps $A \to N$ are unitally homotopic.
    \end{proof}

    \begin{lem}\label{lem: idempotent characterisation}
            Let $X \in \CC$. Given maps $e: \one \to X$ and $m: X \otimes X \to X$ where $e$ is central and $m \circ (e \otimes \id_X) \simeq \id_X$, it follows that $X$ is an idempotent algebra. 
        \end{lem}
    
        \begin{proof}
            Since $e$ is central and $m \circ (e \otimes \id_X) \simeq \id_X$, by Lemma \ref{lem: key iso} we know $e \otimes \id_X$ and $m$ are inverse equivalences. Hence $e$ is an idempotent map so $X$ is an idempotent algebra with unit $e$. See Definition \ref{def: idempotent} and \cite{ha} Proposition 4.8.2.9.
        \end{proof}

    \begin{cor}\label{cor: Jalpha is idempotent}
        If $\alpha$ is central then $\overline{\alpha}$ is idempotent and hence $J_\alpha$ is an idempotent algebra. The functor $\CC \to \CC$ given by $E \mapsto J_\alpha \otimes E$ is hence a smashing localisation (Proposition \ref{prop: class of smashing locs}), whose essential image consists of those objects of $\CC$ which admit the structure of a $J_\alpha$-module.
    \end{cor}

    \begin{proof}
        We use the characterisation of idempotent algebras given in Lemma \ref{lem: idempotent characterisation}, and Corollary \ref{cor: still central} supplies the requisite property that $\overline{\alpha}: \one \to J_\alpha$ is central. The other input is just the fact that $\overline{\alpha}$ is the unit for the $\E_1$-multiplication on $J_\alpha$, which comes from the universal property of Construction \ref{constr: Jalpha}.

        Since $J_\alpha$ is an idempotent algebra, its multiplication map $(\overline{\alpha} \otimes \id_{J_\alpha})^{-1}: J_\alpha \otimes J_\alpha \to J_\alpha$ refines to an $\E_\infty$-algebra structure on $J_\alpha$ (see Definition \ref{def: idempotent} and \cite{ha} Proposition 4.8.2.9). By Corollary \ref{cor: very useful En refinement} \ref{cor sub: unique En str on idem} there is a unique $\E_1$-algebra structure on $J_\alpha$ compatible with the $\E_0$-algebra structure that its idempotent algebra structure induces, i.e. compatible with the unit map $\overline{\alpha}$. Hence the $\E_\infty$-algebra structure we produced for $J_\alpha$ must restrict to the same $\E_1$-algebra structure that $J_\alpha$ already had in Construction \ref{constr: Jalpha}.
    
        For the characterisation of the essential image of $J_\alpha$-localisation, refer to Proposition 4.8.2.10 of \cite{ha}, included in this paper as Proposition \ref{prop: smashing loc modules over unit}. 
    \end{proof}

    \begin{notation}\label{ntn: central map presents loc}
    If $\alpha: \one \to A$ is a central map, we say that $\alpha$ \emph{presents} the smashing localisation $L_\alpha := J_\alpha \otimes -$ of Corollary \ref{cor: Jalpha is idempotent}.
    \end{notation}

    \begin{lem}\label{lem: check local on A}
        Let $\alpha$ be central. An object $X \in \CC$ is $J_\alpha$-local if and only if the map $\alpha \otimes \id_X: X \to A \otimes X$ is an equivalence.
    \end{lem}

    \begin{proof}
        Since $J_\alpha$ is idempotent, we know that local objects are precisely $J_\alpha$-modules (Proposition \ref{prop: smashing loc modules over unit}), and an object $X$ is local if and only if $\overline{\alpha} \otimes \id_X: X \to J_\alpha \otimes X$ is an equivalence. We want to show that the same statement is true with $\alpha$ in place of $\overline{\alpha}$.

        Suppose first that $X$ is local, so $\overline{\alpha} \otimes \id_X: X \to J_\alpha \otimes X$ is an equivalence. We have a commuting diagram 
        \[\begin{tikzcd}[sep={10em,between origins}]
            X & {A \otimes X} \\
            {J_\alpha \otimes X} & {A \otimes J_\alpha \otimes X,}
            \arrow["{\alpha \otimes \id_X}", from=1-1, to=1-2]
            \arrow["{\overline{\alpha} \otimes \id_X}"', "\simeq", from=1-1, to=2-1]
            \arrow["{\id_A\otimes \overline{\alpha} \otimes \id_X}", "\simeq"', from=1-2, to=2-2]
            \arrow["{\alpha \otimes \id_{J_\alpha} \otimes \id_X}"', "\simeq", from=2-1, to=2-2]
        \end{tikzcd}\]
        where $\overline{\alpha} \otimes \id_X$ is an equivalence because $X$ is local, and $\alpha \otimes \id_{J_\alpha}$ is an equivalence by Corollary \ref{cor: one more iso}. Hence $\alpha \otimes \id_X$ is an equivalence, because the other three maps in the diagram are.

        Conversely, suppose $\alpha \otimes \id_X$ is an equivalence and let $\varphi: A \otimes X \to X$ be an inverse.
        Since $\CC$ is presentably symmetric monoidal, recall that it has an internal hom $\hom(-, -)$ that is right adjoint to the tensor product, i.e. 
        \[\Map_\CC(A, \hom(B, C)) \simeq \Map_\CC(A \otimes B, C).\]  
        Then $\hom(X, X)$ is automatically an $\E_1$-algebra in $\CC$, and $\varphi$ corresponds via the adjunction to an $\E_0$ map $\widetilde{\varphi}: A \to \hom(X, X)$. Indeed $\widetilde{\varphi}$ is $\E_0$ because the unit of $\hom(X, X)$ corresponds to $\id_X$ under the adjunction, and we know $\varphi \circ (\alpha \otimes \id_X) \simeq \id_X$, so $\widetilde{\varphi} \circ \alpha$ is the unit of $\hom(X, X)$. Then by the universal property of $J_\alpha$, we obtain an $\E_1$ map $\widetilde{\psi}: J_\alpha \to \hom(X, X)$. The adjunction produces a corresponding map $\psi: J_\alpha \otimes X \to X$. Transporting the commuting diagram
        \[\begin{tikzcd}[sep={7em,between origins}]
            \one & A & {J_\alpha} \\
            & {\hom(X, X)}
            \arrow["\alpha", from=1-1, to=1-2]
            \arrow[from=1-1, to=2-2]
            \arrow["c", from=1-2, to=1-3]
            \arrow["{\widetilde{\varphi}}"', from=1-2, to=2-2]
            \arrow["{\widetilde{\psi}}", from=1-3, to=2-2]
        \end{tikzcd}\]
        along the adjunction, we obtain the diagram
        \[\begin{tikzcd}[sep={7em,between origins}]
            X & {A \otimes X} & {J_\alpha \otimes X} \\
            & X,
            \arrow["{\alpha \otimes \id_X}", from=1-1, to=1-2]
            \arrow["{\id_X}"', from=1-1, to=2-2]
            \arrow["{c \otimes \id_X}", from=1-2, to=1-3]
            \arrow["\varphi"', from=1-2, to=2-2]
            \arrow["\psi", from=1-3, to=2-2]
        \end{tikzcd}\]
        and hence $\psi \circ (\overline{\alpha} \otimes \id_X) \simeq \id_X$. By centrality of $\overline{\alpha}$ and Lemma \ref{lem: key iso}, we conclude that $\overline{\alpha} \otimes \id_X$ is an equivalence with inverse $\psi$, so $X$ is local. 
    \end{proof}

    There are some situations where taking the idempotent algebra on a central map gains us nothing new, for example if the central map was already idempotent. A perhaps more interesting example is when the object $A$ has a ring structure. If the unit map for a ring is central, it is automatically idempotent. This perfectly replicates the situation for discrete rings, as studied in \cite{solid}.

    \begin{cor}\label{cor: Jalpha eq Jalphabar}
        If $\alpha$ happens to be idempotent rather than merely central, then $J_\alpha \simeq A$. It follows that in general, $J_{\overline{\alpha}} \simeq J_\alpha$.
    \end{cor}

    \begin{proof}
        When $\alpha$ is idempotent, $A$ itself is an idempotent algebra. We know $J_\alpha$ is also an idempotent algebra, by Corollary \ref{cor: Jalpha is idempotent}. We have an $\E_0$ map $c: A \to J_\alpha$ by construction of $J_\alpha$. Moreover, the commuting diagram 
        \[\begin{tikzcd}[sep={7em, between origins}]
            \one \arrow[rd, "\alpha"'] \arrow[r, "\alpha"] & A \arrow[d, "\id_A"] \arrow[r, "c"] & J_\alpha \arrow[ld, "\exists \phi"] \\
                                                           & A                                 &                                    
            \end{tikzcd}\]
            gives us, by the universal property of $J_\alpha$, an $\E_1$ map $\phi: J_\alpha \to A$. By Lemma \ref{lem: En refinement}, both $c$ and $\phi$ refine uniquely to $\E_\infty$. We already know $\phi \circ c \simeq \id_A$. The other composite is an endomorphism of $J_\alpha$ which is at least $\E_0$, so it must be the identity -- by Corollary \ref{cor: very useful En refinement} the space of $\E_0$ endomorphisms of an idempotent algebra is contractible. Hence $c$ and $\phi$ are inverse $\E_\infty$ maps so $J_\alpha \simeq A$.

        Now in the case of general $\alpha$, the map $\overline{\alpha}$ is central and idempotent, so if we repeat the process again to obtain $J_{\overline{\alpha}}$ we have obtained nothing new, and we find that $J_\alpha \simeq J_{\overline{\alpha}}$. That is, the structure map $J_\alpha \to J_{\overline{\alpha}}$ in the universal property of $J_{\overline{\alpha}}$ is forced to be an equivalence.
    \end{proof}

    \begin{remark}
        Here is another proof of the same fact: we can directly check that $A$ satisfies the universal property of $J_\alpha$ in the case where $\alpha$ is idempotent. We must show that an $\E_0$ map $A \to R_1$ refines uniquely to $\E_1$, where $R_1$ is $\E_1$. But since $A$ is an idempotent algebra, this is part of the content of Lemma \ref{lem: En refinement}.

        One could imagine a similar construction to $J_\alpha$ but producing an $\E_n$-algebra: namely, we could take the left adjoint to the forgetful functor from $\E_n$-algebras to $\E_0$-algebras and apply it to a given $\E_0$-algebra $\alpha: \one \to A$, as in Construction \ref{constr: Jalpha}. Such adjoints exist for general $n$ by the same considerations as for $\E_1$. However, doing this would not give us anything new as $J_\alpha$, when defined as an $\E_1$-algebra, already satisfies the $\E_n$ version of the universal property. Given an $\E_0$ map $J_\alpha \to R$ with $R$ an $\E_n$-algebra, the map refines uniquely to $\E_n$, by Lemma \ref{lem: En refinement} combined with the fact that $J_\alpha$ is idempotent.
    \end{remark}

    \begin{lem}\label{lem: central implies idempotent for a ring}
        If $A \in \Alg_{\E_n}(\CC)$ for $n \geq 1$, and its unit map $\alpha: \one \to A$ is central, then $\alpha$ is automatically idempotent and hence $J_\alpha \simeq A$.
    \end{lem}

    \begin{proof}
        We have a multiplication map $\mu: A \otimes A \to A$ because $A$ is at least $\E_1$. Moreover, $\mu \circ (\alpha \otimes \id_A) \simeq \id_A \simeq \mu \circ (\id_A \otimes \alpha)$. By Lemma \ref{lem: key iso} with $M = A$, it follows that $\mu$ and $\alpha \otimes \id_A$ are inverse equivalences. Hence $A$ is an idempotent algebra, because its unit map $\alpha$ is an idempotent map. Then Corollary \ref{cor: Jalpha eq Jalphabar} tells us $J_\alpha \simeq A$.
    \end{proof}

    \begin{cor}\label{cor: central map universality}
        Every smashing localisation is presented by a central map. Indeed, for a smashing localisation $L$, the unit map $\alpha: \one \to L\one$ is idempotent, central, and presents $L$. That is, $J_\alpha \simeq L\one$ as idempotent algebras in $\CC$.
    \end{cor}

    \begin{proof}
        Centrality of $\alpha$ follows directly from the fact that the localisation $L$ is given by tensoring with $L\one$, and $L\one \otimes L\one \simeq L\one$. The maps $\alpha \otimes \id_{L\one}$ and $\id_{L\one} \otimes \alpha$ are the same after postcomposition with this identification. Then by Corollary \ref{cor: Jalpha eq Jalphabar}, since $L\one$ is already idempotent we have $L\one \simeq J_\alpha$ as idempotent algebras in $\CC$.
    \end{proof}

    We next need a way to relate properties of $J_\alpha$ to properties of the central map $\alpha$ we started with. In light of Lemma \ref{lem: acyclics gen by}, in a nice setting like the category of spectra we know that the acyclics for the smashing localisation $L_\alpha$ are generated under colimits and desuspensions by the single acyclic $\cof(\one \to J_\alpha)$. Thus this cofibre completely characterises the localisation. To compute the localisation presented by a given central map, as we will later want to do, we must understand this cofibre. So we will need a more explicit construction of $J_\alpha$ which allows us to relate the properties of $\cof(\one \to J_\alpha)$ to $\cof \alpha$. 

    \begin{constr}[cf \cite{NilpotenceII} (2.3)]\label{constr: explicit Jalpha}
        Given a map $\alpha: \one \to A$, we build the sequential colimit of powers of $A$, with maps given by tensoring with $\alpha$. That is, consider the diagram
        \[\begin{tikzcd}[column sep={4em}]
            \one \arrow[r, "\alpha"] & A \arrow[r, "\id \otimes \alpha"] & A^{\otimes 2} \arrow[r, "\id^{\otimes 2} \otimes \alpha"] & A^{\otimes 3} \arrow[r, "\id^{\otimes 3} \otimes \alpha"] & \ldots,
        \end{tikzcd}\] 
        and let $J_\alpha'$ denote its colimit in $\CC$. For ease of reference, let $\mathfrak{J}_\alpha$ denote the diagram itself. In the case $\CC = \Sp$, the colimit $J'_\alpha$ is the infinite smash product of $\alpha$ considered in \cite{NilpotenceII} (2.3), where it is used to detect smash nilpotence.
    \end{constr}

    Now we must show that Construction \ref{constr: Jalpha} and Construction \ref{constr: explicit Jalpha} agree when $\alpha$ is central. This will allow us to compute $\cof (\one \to J_\alpha) \simeq \cof (\one \to J_\alpha')$ in terms of $\cof \alpha$. The motivation for Construction \ref{constr: explicit Jalpha} is that $J_\alpha$ is determined by the property that $\alpha$ acts as its identity, so we formally invert multiplication by $\alpha$ to construct it. 
    
    Construction \ref{constr: explicit Jalpha} makes sense even when $\alpha$ is not central, but we should not expect it to compute $J_\alpha$ in this case because we have only formally inverted \emph{right} multiplication by $\alpha$. When $\alpha$ is central, left and right multiplication by $\alpha$ agree, so we should expect all multiplication by $\alpha$ on $J_\alpha'$ to be unital -- and hence $J_\alpha'$ to agree with $J_\alpha$ in this case.

    \begin{lem}\label{lem: Jalpha' localises to Jalpha}
        There is a canonical unit map $e': \one \to J_\alpha'$. If $\alpha$ is central, this unit map is a $J_\alpha$-local equivalence. 
    \end{lem}

    \begin{proof}
        The map $\one \to J_\alpha'$ comes from the diagram $\mathfrak{J}_\alpha$ computing $J_\alpha'$. By definition $\one$ is the initial term in this diagram, and $J_\alpha'$ is the colimit over the diagram, so $J_\alpha'$ comes equipped with a unit map $e': \one \to J_\alpha'$. 

        Localisation with respect to $J_\alpha$ is computed by tensoring with $J_\alpha$, so we must show that $\id_{J_\alpha} \otimes e': J_\alpha \to J_\alpha \otimes J_\alpha'$ is an equivalence. Consider the diagram $J_\alpha \otimes \mathfrak{J}_\alpha$, which computes $J_\alpha \otimes J_\alpha'$ as its colimit. Each term of the diagram $J_\alpha \otimes \mathfrak{J}_\alpha$ is a copy of $J_\alpha$, since $J_\alpha \otimes A^{\otimes n} \simeq J_\alpha$ by Corollary \ref{cor: one more iso} applied inductively. Each map in the diagram $J_\alpha \otimes \mathfrak{J}_\alpha$ is canonically homotopic to the identity on $J_\alpha$ under this identification of terms with $J_\alpha$. To see this, note that $\id_{J_\alpha} \otimes \id_A \simeq \id_{J_\alpha}$ under the identification $\id_{J_\alpha} \otimes \alpha: J_\alpha \xrightarrow{\simeq} J_\alpha \otimes A$ of Corollary \ref{cor: one more iso}, and we have commuting diagrams
            \[\begin{tikzcd}[sep={6em,between origins}]
                {J_\alpha \otimes A^{\otimes n}} && {J_\alpha \otimes A^{\otimes (n+1)}} \\
                & {J_\alpha}
                \arrow["{\id_{J_\alpha} \otimes \id_A^{\otimes n} \otimes \alpha}", from=1-1, to=1-3]
                \arrow["{\id_{J_\alpha} \otimes \alpha^{\otimes n}}", "\simeq"', from=2-2, to=1-1]
                \arrow["{\id_{J_\alpha} \otimes \alpha^{\otimes (n+1)}}"', "\simeq", from=2-2, to=1-3]
            \end{tikzcd}\]
        showing that all the maps in the diagram $J_\alpha \otimes \mathfrak{J}_\alpha$ reduce to $\id_{J_\alpha}$. Such homotopies are canonical in the sense that the space of unital endomorphisms of the idempotent algebra $J_\alpha$ is contractible, by Corollary \ref{cor: very useful En refinement} \ref{cor sub: end of idem}. Hence $J_\alpha \otimes J_\alpha'$ is the colimit of a sequential diagram where all the terms are $J_\alpha$ and all the maps are $\id_{J_\alpha}$, so $J_\alpha \simeq J_\alpha \otimes J_\alpha'$. All the maps from terms of $J_\alpha \otimes \mathfrak{J}_\alpha$ to this colimit are the identity on $J_\alpha$. Moreover, the map 
            \[\id_{J_\alpha} \otimes e': J_\alpha \to J_\alpha \otimes J_\alpha'\]
        identifies with the map $J_\alpha \otimes \one \to J_\alpha \otimes J_\alpha'$ coming from the first term of the diagram $J_\alpha \otimes \mathfrak{J}_\alpha$ mapping to the colimit. Hence $\id_{J_\alpha} \otimes e' \simeq \id_{J_\alpha}$ under our identification of $J_\alpha \otimes J_\alpha'$ with $J_\alpha$, and in particular $\id_{J_\alpha} \otimes e'$ is an equivalence. 
    \end{proof}

    \begin{lem}\label{lem: Jalpha' is local for Jalpha}
        Assume $\CC$ stable. If $\alpha$ is central then $J_\alpha'$ is $J_\alpha$-local. That is, $J_\alpha'$ is local with respect to the smashing localisation computed by tensoring with $J_\alpha$.
    \end{lem}

    See Construction \ref{constr: explicit Jalpha}, Construction \ref{constr: Jalpha}, and Corollary \ref{cor: Jalpha is idempotent}.

    \begin{proof}
        To show that $J'_\alpha$ is $J_\alpha$-local, we must show that $\overline{\alpha} \otimes \id_{J_\alpha'}: J'_\alpha \to J_\alpha \otimes J'_\alpha$ is an equivalence. By Lemma \ref{lem: check local on A}, this reduces to showing that $\alpha \otimes \id_{J'_\alpha}: J'_\alpha \to A \otimes J'_\alpha$ is an equivalence.

        Write $Z = \cof \alpha$, so that there is a cofibre sequence
        \[\one \xrightarrow{\alpha} A \xrightarrow{q} Z \xrightarrow{\delta} \Sigma \one.\]
        Since $\CC$ is stable, $\alpha \otimes \id_{J'_\alpha}$ is an equivalence if and only if its cofibre $Z \otimes J'_\alpha$ vanishes. The tensor product preserves colimits, so $Z \otimes J'_\alpha$ is the colimit of the sequential diagram $Z \otimes \mathfrak{J}_\alpha$ with terms $Z \otimes A^{\otimes n}$ and maps $\id_{Z \otimes A^{\otimes n}} \otimes \alpha$. We will show that the composite of any two consecutive maps in this diagram is nullhomotopic.

        Centrality of $\alpha$ gives
        \begin{align*}
            (\id_Z \otimes \alpha) \circ q & \simeq q \otimes \alpha \simeq (q \otimes \id_A) \circ (\id_A \otimes \alpha)\\ 
            &\simeq (q \otimes \id_A) \circ (\alpha \otimes \id_A) \simeq (q \circ \alpha) \otimes \id_A \simeq 0.
        \end{align*}
        Since $\delta: Z \to \Sigma \one$ is the cofibre of $q$, it follows that $\id_Z \otimes \alpha \simeq h \circ \delta$ for some map $h: \Sigma \one \to Z \otimes A$. The composite of the first two maps in the diagram $Z \otimes \mathfrak{J}_\alpha$ is therefore
        \[\id_Z \otimes \alpha^{\otimes 2} \simeq (h \circ \delta) \otimes \alpha \simeq (h \otimes \id_A) \circ (\id_{\Sigma\one} \otimes \alpha) \circ \delta \simeq (h \otimes \id_A) \circ \Sigma \alpha \circ \delta \simeq 0,\]
        because $\Sigma \alpha \circ \delta$ is the composite of consecutive maps in a cofibre sequence. For general $n$, the composite $Z \otimes A^{\otimes n} \to Z \otimes A^{\otimes (n+2)}$ of two consecutive maps is $\id_{Z \otimes A^{\otimes n}} \otimes \alpha^{\otimes 2}$, which agrees with $(\id_Z \otimes \alpha^{\otimes 2}) \otimes \id_{A^{\otimes n}}$ after permuting tensor factors, and is therefore also nullhomotopic.

        The terms with even index form a cofinal subdiagram, so $Z \otimes J'_\alpha$ is the colimit of a sequential diagram $Y_0 \to Y_1 \to \cdots$ in which every map is nullhomotopic, where $Y_n$ is the $(2n)$th term of $Z \otimes \mathfrak{J}_\alpha$. Such a colimit vanishes. Indeed, since $\CC$ is stable, the colimit is the cofibre of $\id - s: \bigoplus_n Y_{n} \to \bigoplus_n Y_{n}$, where $s$ is induced by the maps in the diagram. Since $s \simeq 0$, the map $\id - s$ is an equivalence and its cofibre is zero. 
    
        Hence $Z \otimes J'_\alpha \simeq 0$, so $\alpha \otimes \id_{J'_\alpha}$ is an equivalence and $J'_\alpha$ is $J_\alpha$-local.
    \end{proof}

    The next Proposition is a direct consequence of Lemmas \ref{lem: Jalpha' is local for Jalpha} and \ref{lem: Jalpha' localises to Jalpha}.

    \begin{prop}\label{prop: Jalpha Jalpha' agree}
        If $\CC$ is stable and $\alpha$ is central, $J_\alpha \simeq J_\alpha'$ as $\E_0$-algebras. That is, Constructions \ref{constr: Jalpha} and \ref{constr: explicit Jalpha} agree. In particular $e'$ is idempotent.
    \end{prop}

    \begin{proof}
        By Lemma \ref{lem: Jalpha' localises to Jalpha} we know $J_\alpha \simeq J_\alpha \otimes J_\alpha'$, and by Lemma \ref{lem: Jalpha' is local for Jalpha} we know $J_\alpha \otimes J_\alpha' \simeq L_\alpha J_\alpha' \simeq J_\alpha'$, since $J_\alpha'$ is already local. Combining these two facts, we see that $J_\alpha \simeq J_\alpha'$. 

        Moreover, we can relate $J_\alpha$ and $J_\alpha'$ via explicit equivalences. We know that the map $\id_{J_\alpha} \otimes e': J_\alpha \to J_\alpha \otimes J_\alpha'$ is an equivalence by Lemma \ref{lem: Jalpha' localises to Jalpha}, and the map $\overline{\alpha} \otimes \id_{J_\alpha'}: J_\alpha' \to J_\alpha \otimes J_\alpha'$ is an equivalence by Lemma \ref{lem: Jalpha' is local for Jalpha}. 

        The commuting diagram 
        \[\begin{tikzcd}[sep={7em, between origins}]
            {J_\alpha'} & {J_\alpha \otimes J_\alpha'} & {J_\alpha} \\
            & \one
            \arrow["{\overline{\alpha} \otimes \id_{J_\alpha'}}", from=1-1, to=1-2]
            \arrow["{\id_{J_\alpha} \otimes e'}"', from=1-3, to=1-2]
            \arrow["{e'}", from=2-2, to=1-1]
            \arrow["{\overline{\alpha}}"', from=2-2, to=1-3]
        \end{tikzcd}\]
    then tells us that the identification $J_\alpha'\simeq J_\alpha$ is as $\E_0$-algebras with respect to the unit maps already established. Since $\overline{\alpha}$ is known to be idempotent and $e'$ identifies with $\overline{\alpha}$ under the equivalence $J_\alpha'\simeq J_\alpha$, we conclude that $e'$ is idempotent.
    \end{proof}

    Now that $J_\alpha$ and $J_\alpha'$ agree as $\E_0$-algebras, we obtain an $\E_\infty$ structure on $J_\alpha'$ transferred from $J_\alpha$. This is the unique $\E_\infty$ structure on $J_\alpha'$ compatible with its $\E_0$ structure, because $e'$ is idempotent -- see Corollary \ref{cor: very useful En refinement}.

    \begin{remark}\label{rmk: explicit maps Jalpha Jalpha'}
        We can now describe explicit maps in each direction between $J_\alpha$ and $J_\alpha'$. There is a canonical map $A \to J_\alpha'$ since $A$ appears as a term in the diagram $\mathfrak{J}_\alpha$, and this map is compatible with $e'$. Then by the universal property of $J_\alpha$, we obtain a commuting diagram 
        \[\begin{tikzcd}[sep={6em, between origins}]
            \one & A & {J_\alpha} \\
            & {J_\alpha',}
            \arrow["\alpha", from=1-1, to=1-2]
            \arrow["{e'}"', from=1-1, to=2-2]
            \arrow["c", from=1-2, to=1-3]
            \arrow[from=1-2, to=2-2]
            \arrow["{\varphi}", from=1-3, to=2-2]
        \end{tikzcd}\]
        with the map $\varphi$ being the unique $\E_1$ map fitting into the diagram. By Lemma \ref{lem: En refinement}, $\varphi$ refines uniquely to $\E_\infty$. The only obstruction to obtaining such a map $\varphi$ directly from Constructions \ref{constr: Jalpha} and \ref{constr: explicit Jalpha} is that we did not yet have an $\E_1$ structure on $J_\alpha'$, so we would have needed to build one explicitly using the diagram $\mathfrak{J}_\alpha$. 

        In the other direction, since we know that $J_\alpha$ is an idempotent algebra we have maps $A^{\otimes n} \to J_\alpha^{\otimes n} \to J_\alpha$ by taking a tensor power of the canonical map $c: A \to J_\alpha$ and postcomposing with the multiplication on $J_\alpha$. We can see that these are compatible with the maps in the diagram $\mathfrak{J}_\alpha$ by using that $c \circ \alpha \simeq \overline{\alpha}$ and $\overline{\alpha}$ is the unit of $J_\alpha$ and thus compatible with the multiplication. 

        Since the diagram $\mathfrak{J}_\alpha$ maps to $J_\alpha$ and $J_\alpha'$ is the colimit of the diagram, we obtain an induced map $J_\alpha' \to J_\alpha$. This induced map is $\E_0$ essentially because the map $c: A \to J_\alpha$ is $\E_0$.
    \end{remark}

    We can use the identification $J_\alpha \simeq J_\alpha'$ to compute $\cof \overline{\alpha}$ in terms of $\cof \alpha$. 
    \begin{lem}\label{lem: baralpha in terms of alpha}
        Let $\CC$ be stable and $\alpha$ be central. Then $\cof \overline{\alpha}$ can be computed as the colimit of a sequential diagram with terms of the form $\cof(\alpha^{\otimes n})$, where $\cof(\alpha^{\otimes (n+1)})$ is an extension of $\cof \alpha$ by $A \otimes \cof(\alpha^{\otimes n})$.
    \end{lem}

    \begin{proof}
        By Proposition \ref{prop: Jalpha Jalpha' agree} we have $\cof \overline{\alpha} \simeq \cof(e')$, so we need only compute $\cof(e')$. Because cofibres and colimits commute, from Construction \ref{constr: explicit Jalpha} we know \hbox{$\cof(e': \one \to J_\alpha')$} is the colimit of the diagram 
        \[\begin{tikzcd}
            0 & {\cof \alpha} & {\cof(\alpha^{\otimes 2})} & {\cof(\alpha^{\otimes 3})} & {\cdots,}
            \arrow[from=1-1, to=1-2]
            \arrow[from=1-2, to=1-3]
            \arrow[from=1-3, to=1-4]
            \arrow[from=1-4, to=1-5]
        \end{tikzcd}\]
        whose maps between cofibres are induced from the maps in the diagram $\mathfrak{J}_\alpha$.

        Consider the diagram
        \[\begin{tikzcd}[sep={8em,between origins}]
            \one & A & {\cof \alpha} \\
            \one & {A^{\otimes (n+1)}} & {\cof(\alpha^{\otimes (n+1)})} \\
            0 & {\cof(\id_A \otimes \alpha^{\otimes n})} & {\cof(\id_A \otimes \alpha^{\otimes n}),}
            \arrow["\alpha", from=1-1, to=1-2]
            \arrow[from=1-1, to=2-1, equal]
            \arrow[from=1-2, to=1-3]
            \arrow["{\id_A \otimes \alpha^{\otimes n}}", from=1-2, to=2-2]
            \arrow[from=1-3, to=2-3]
            \arrow["{\alpha^{\otimes (n+1)}}", from=2-1, to=2-2]
            \arrow[from=2-1, to=3-1]
            \arrow[from=2-2, to=2-3]
            \arrow[from=2-2, to=3-2]
            \arrow[from=2-3, to=3-3]
            \arrow[from=3-1, to=3-2]
            \arrow[from=3-2, to=3-3, equal]
        \end{tikzcd}\]
        whose rows and columns are all cofibre sequences. Thus $\cof(\alpha^{\otimes (n+1)})$ is an extension of $\cof \alpha$ by $\cof(\id_A \otimes \alpha^{\otimes n}) \simeq A \otimes \cof(\alpha^{\otimes n})$. Thus all of the terms in the colimit computing $\cof(e')$ are built from $\cof \alpha$ via iterated extensions and tensoring with $A$.
    \end{proof}
    
    Although this description is not completely explicit, it does allow us to transfer properties of interest from $\alpha$ to $\overline{\alpha}$. We will use this later in Lemma \ref{lem: cof Jbeta correct type} to allow us to explicitly compute the localisation presented by a given central map $\alpha$. 

To conclude our discussion, we would like to give a characterisation of when two central maps present the same smashing localisation. In the following, let $\alpha: \one \to A$ and $\beta: \one \to B$ be central maps in a presentably symmetric monoidal category $\CC$. Our first observation is that if $\alpha$ is central, then $\alpha^{\otimes n}$ is also central. Moreover these two central maps present the same smashing localisation.

\begin{lem}\label{lem: power of central map}
    Let $n \geq 1$. If $\alpha$ is central then $\alpha^{\otimes n}$ is central, and $L_{\alpha} \simeq L_{\alpha^{\otimes n}}$.
\end{lem}

\begin{proof}
    Centrality of $\alpha^{\otimes n}$ is essentially immediate by unwinding the definition. 

    The commuting diagram 
    \[\begin{tikzcd}[sep={7em, between origins}]
            \one & A & {J_\alpha} \\
            & {A^{\otimes n}} & {J_{\alpha^{\otimes n}}}
            \arrow["\alpha", from=1-1, to=1-2]
            \arrow["{\alpha^{\otimes n}}"', from=1-1, to=2-2]
            \arrow["{c_{(1)}}", from=1-2, to=1-3]
            \arrow["{\id_A \otimes \alpha^{\otimes (n-1)}}", from=1-2, to=2-2]
            \arrow["\varphi", dashed, from=1-3, to=2-3]
            \arrow["{c_{(n)}}"', from=2-2, to=2-3]
        \end{tikzcd}\]
    tells us that, since $J_{\alpha^{\otimes n}}$ is $\E_\infty$ and in particular $\E_1$, there is an induced map $\varphi: J_\alpha \to J_{\alpha^{\otimes n}}$ coming from the universal property of $J_\alpha$ and which is the unique $\E_1$ map fitting into this diagram. 

    To obtain a map in the other direction, let $c_n: A^{\otimes n} \to J_\alpha$ be the composite of $c^{\otimes n}$ with the multiplication $J_\alpha^{\otimes n} \to J_\alpha$, as in Remark \ref{rmk: explicit maps Jalpha Jalpha'}. Since $c \circ \alpha \simeq \overline{\alpha}$ is the unit of $J_\alpha$, $c_n$ is an $\E_0$ map. Then using the universal property of $J_{\alpha^{\otimes n}}$, we have a diagram 
    \[\begin{tikzcd}[sep={7em, between origins}]
        \one & {A^{\otimes n}} & {J_{\alpha^{\otimes n}}} \\
        & {J_\alpha}
        \arrow["{\alpha^{\otimes n}}", from=1-1, to=1-2]
        \arrow["{\overline{\alpha}}"', from=1-1, to=2-2]
        \arrow["{c_{(n)}}", from=1-2, to=1-3]
        \arrow["{c_n}"', from=1-2, to=2-2]
        \arrow["\psi", dashed, from=1-3, to=2-2]
    \end{tikzcd}\]
    producing an $\E_1$ map $\psi: J_{\alpha^{\otimes n}} \to J_\alpha$. By Lemma \ref{lem: En refinement}, the maps $\psi$ and $\varphi$ both refine to $\E_\infty$ and are inverse to one another, simply because $J_\alpha$ and $J_{\alpha^{\otimes n}}$ are idempotent. Thus $J_\alpha$ and $J_{\alpha^{\otimes n}}$ are canonically equivalent. Indeed, the $\E_\infty$ and $\E_0$ mapping spaces between them are contractible. Hence their smashing localisations are canonically identified.
\end{proof}

We already have a means of comparing two localisations. Recall that $L \leq L'$ means containment on local subcategories or reverse containment for acyclics. When $L$ and $L'$ are smashing, combining Lemmas \ref{lem: smash loc partial order via ring map} and \ref{lem: En refinement} tells us that $L \leq L'$ is equivalent to supplying an $\E_0$ map $L'\one \to L\one$.

Thus the condition $L_\alpha \leq L_\beta$ is equivalent to the claim that the unit of $J_\alpha$ factors through the unit of $J_\beta$, i.e. $\overline{\alpha} \simeq f \circ \overline{\beta}$ for some comparison map $f$. The space of such comparison maps $f$ is \emph{a priori} either empty or contractible, so a specific choice of $f$ provides no additional data. One could ask whether, given such a factoring $\overline{\alpha} \simeq f \circ \overline{\beta}$, it is possible to produce a factoring on the level of the central maps $\alpha$ and $\beta$ we started with. 

Lemma \ref{lem: power of central map} suggests that we may not always be able to obtain a factoring of the form $g \circ \beta \simeq \alpha$. Indeed, since $L_{\alpha} \simeq L_{\alpha^{\otimes n}}$ we have $L_\alpha \leq L_{\alpha^{\otimes n}}$, so we would in particular need to produce a family of $\E_0$ maps of the form $g: A^{\otimes n} \to A$ with $g \circ \alpha^{\otimes n} \simeq \alpha$. Such maps may not exist, as we see in the following example. 

\begin{example}
    Let $I$ denote the Brown-Comenetz dual of the sphere. We need two properties of this spectrum. First, $I \otimes I \simeq 0$ (see \cite{hoveystrickland} Corollary B.12, or \cite{mathew}) but $I$ itself is nonzero. That is, $I$ is tensor-nilpotent. Second, there is a nonzero map $\S \to I$. Indeed, the defining property of $I$ is that for $X$ any spectrum, $[X, I] \cong \pi_0\Map(X, I)\cong \Hom_\Z(\pi_0 X, \Q/\Z)$. Thus $[\S, I] \cong \Q/\Z$. In particular, $[\S, I]$ is nontrivial so it contains some nonzero map $f$. 

    Any map $f: \S \to I$ is central, simply because of the property that $I \otimes I \simeq 0$. For $f$ to be central we need $f \otimes \id_I \simeq \id_I \otimes f$, but the codomain of these maps is $I \otimes I \simeq 0$ so this is automatic. We claim that there is no $\E_0$ map $I \otimes I \to I$. There are no nonzero maps $0 \simeq I \otimes I \to I$, so the composite $\S \to I \otimes I \to I$ is also nullhomotopic, but the central map $f: \S \to I$ was chosen to be nonzero. The same argument shows there is no $\E_0$ map $I^{\otimes n} \to I$ for any $n \geq 2$, when the unit $f: \S \to I$ is chosen to be nonzero. 
\end{example}

But Lemma \ref{lem: power of central map} also suggests a way to repair this problem: instead of asking that $g \circ \beta \simeq \alpha$, we could ask for a factoring of the form $g \circ \beta \simeq \alpha^{\otimes n}$ for some $n \gg 0$. Replacing $\alpha$ by its tensor power is reasonable since all the tensor powers of a central map present the same localisation. Then we have the following results.

\begin{lem}\label{lem: central map divisibility easy}
    If there exists a map $g: B \to A^{\otimes n}$ such that $g \circ \beta \simeq \alpha^{\otimes n}$, then $L_\alpha \leq L_\beta$.
\end{lem}

\begin{proof}
    Compose $g$ with the canonical $\E_0$ map $A^{\otimes n} \to J_\alpha$, or equivalently with the map $A^{\otimes n} \to J_{\alpha^{\otimes n}}$ (see Lemma \ref{lem: power of central map}). The composite is an $\E_0$ map $\overline{g}: B \to J_\alpha$. Then the universal property of $J_\beta$ produces an $\E_1$ map $J_\beta \to J_\alpha$, and thus by Lemmas \ref{lem: smash loc partial order via ring map} and \ref{lem: En refinement} we learn that $L_\alpha \leq L_\beta$.
\end{proof}

The converse is a little less robust, but still holds given some appropriate compactness hypotheses.
\begin{lem}\label{lem: central map divisibility hard}
    Suppose $\one \in \CC$ is compact and $\CC$ is stable. If $L_\alpha \leq L_\beta$ and $B$ is compact, then there is a map $g: B \to A^{\otimes n}$ such that $g \circ \beta \simeq \alpha^{\otimes n}$.
\end{lem}

\begin{proof}
    Since $L_\alpha \leq L_\beta$ we obtain by Lemma \ref{lem: smash loc partial order via ring map} an $\E_0$ map $J_\beta \to J_\alpha$. The canonical map $B \to J_\beta$ is also $\E_0$, so by composing we produce an $\E_0$ map $f: B \to J_\alpha$. By Proposition \ref{prop: Jalpha Jalpha' agree} we may identify $J_\alpha$ with the colimit of Construction \ref{constr: explicit Jalpha}. Since $B$ is compact, $f$ must factor through some finite stage $A^{\otimes n}$ of the colimit. The forgetful functor from $\E_0$-algebras to underlying objects preserves sifted colimits (\cite{ha} Corollary 3.2.3.2) and the colimit diagram $\mathfrak{J}_\alpha$ is filtered, thus sifted. Hence the same diagram still produces $J_\alpha$ when the colimit is computed in $\E_0$-algebras instead of in $\CC$. The $\E_0$ map $f$ out of the compact $\E_0$-algebra $\one \to B$ therefore factors to produce an $\E_0$ map $B \to A^{\otimes n}$. This is the map $g$ we wanted.
\end{proof}

\begin{prop}\label{prop: compact central map class divisibility}
    Suppose $\one \in \CC$ is compact and $\CC$ is stable. Let $\alpha: \one \to A$ and $\beta: \one \to B$ be central maps with $A$ and $B$ both compact in $\CC$. Then $L_\alpha \simeq L_\beta$ if and only if there exist $n, m \gg 0$ and maps $f: A \to B^{\otimes n}, g: B \to A^{\otimes m}$ such that $f \circ \alpha \simeq \beta^{\otimes n}$ and $g \circ \beta \simeq \alpha^{\otimes m}$.
\end{prop}

\begin{proof}
    Combine Lemmas \ref{lem: central map divisibility easy} and \ref{lem: central map divisibility hard}. 
\end{proof}

Proposition \ref{prop: compact central map class divisibility} suggests an order relation on central maps between compact objects. We say $\alpha \leq \beta$ if $g \circ \beta \simeq \alpha^{\otimes m}$ for some $g$ and $m$, and take the order relation this generates. Two central maps are equivalent under this relation if both $\alpha \leq \beta$ and $\beta \leq \alpha$. One consequence of Proposition \ref{prop: compact central map class divisibility} is that $\alpha$ and $\beta$ are equivalent if and only if they present the same smashing localisation. In Subsections \ref{subsec: props of central maps} and \ref{subsec: alg central to central} we will study a class of maps which we call \emph{algebraically central}, and which have the property that some tensor power of the map is central, although the map itself may not be. We can then extend our equivalence relation to all algebraically central maps, and we know that each equivalence class contains a central map. Two choices of central map present the same localisation if and only if they lie in the same equivalence class. In light of this it is reasonable to think of centrality as really being a property of these equivalence classes, since all the relevant properties are stable under taking arbitrarily large tensor powers. 

\subsection{Finiteness conditions}\label{subsec: finiteness}

    In this Subsection we will compare two notions of finiteness for smashing localisations, and at the end of Section \ref{sec: the Lnf example} we will be able to give a full classification of smashing localisations of spectra satisfying either condition. These two notions turn out to be related but often inequivalent. We first recall a useful alternate characterisation of finite spectra.

    \begin{defn}\label{def: dualisable and rigid}
    An object $X$ of $\CC$ is \emph{dualisable} if there is an object $DX$ with maps $\ev_X: DX \otimes X \to \one$ and $\coev_X: \one \to X \otimes DX$ satisfying the triangle identities (see \cite{ha} Section 4.6.1). Then $\hom(X, Y) \simeq DX \otimes Y$ naturally in $Y$. The monoidal unit is dualisable and self-dual, and a finite tensor product of dualisable objects is dualisable. 
    
    We say that $\CC$ is \emph{rigid} if it is compactly generated and its compact objects are precisely its dualisable objects. In particular, in a rigid category the unit is compact and compact objects are closed under tensor products.
    \end{defn}

    For example, $\Sp$ and $\Sp_{(p)}$ are rigid (\cite{axiomatic} Theorem 2.1.3(d), Example 1.2.3(a), and Theorem 3.5.2). For a finite or $p$-local finite spectrum $X$, $DX \simeq \hom(X, \one)$ is the Spanier-Whitehead dual (see \cite{spanierwhitehead}). So compact, dualisable, and finite spectra all agree. The property of being rigid was named by Balmer and Favi in \cite{balmerfavi} Section 1.

    \begin{defn}[Miller, \cite{millerfinite} Definition 3]
        A localisation is \emph{finite} if the acyclic objects are generated under colimits, or equivalently generated as a localising subcategory, by acyclic compact objects. 
    \end{defn}

    Since the acyclic objects for any localisation are closed under colimits, we are essentially requiring that everything about the localisation be captured by the compact objects in its kernel.

    \begin{defn}\label{def: compactly central}
        A map $\alpha: \one \to A$ is \emph{compact} if its cofibre is a compact object in $\CC$. If $\alpha$ is a compact central map, call the localisation $L_\alpha$ it presents \emph{compactly central}. See Notation \ref{ntn: central map presents loc}.
    \end{defn}

    We call such a map compact because we are usually in a setting where the unit itself is compact, so a map out of the unit with compact cofibre represents a compact object of the slice category under the unit. If the unit in $\CC$ is itself a compact object then compactness of $\alpha$ is of course equivalent to compactness of the codomain $A$. 
    
    A localisation presented by a central map is necessarily smashing, so a compactly central localisation is smashing. In nice settings all finite localisations are smashing, see Lemma \ref{lem: fin implies smashing} and Corollaries \ref{cor: sp fin implies smashing} and \ref{cor: bous fin implies smashing}. This suggests it may be fruitful to investigate the relationship between finite and compactly central localisations.

    \begin{lem}[\cite{axiomatic} Theorem 3.3.3]\label{lem: fin implies smashing}
        Suppose $\CC$ is stable and has the property that all compact objects are dualisable, and $L$ is a localisation of $\CC$ that is both tensor-compatible and finite. Then $L$ is smashing.
    \end{lem}

    \begin{proof}
        First note that for a compact acyclic object $Z$ and a local object $LX$, we have $DZ \otimes LX \simeq 0$. This is because, for an arbitrary object $Y$,
        \[\Map(Y, DZ \otimes LX) \simeq \Map(Y \otimes Z, LX) \simeq *.\]
        Here we are using that $L$ is tensor-compatible and $\CC$ is stable, so $\ker L$ is a tensor ideal (Lemma \ref{lem: equiv monoidal loc} \ref{cond: tensor ideal}), and thus $Y \otimes Z$ is acyclic. The mapping space from an acyclic to a local object is always contractible. Since every mapping space to $DZ \otimes LX$ is contractible, $DZ \otimes LX$ itself is zero by Yoneda.

        Next we claim $LX \otimes Y$ is local for any objects $X$ and $Y$. With $Z$ as before, we have 
        \[\Map(Z, LX\otimes Y) \simeq \Map(\one, DZ \otimes LX \otimes Y) \simeq \Map(\one, 0) \simeq *.\]
        If $A$ is any acyclic, then by finiteness of $L$ it can be represented as an iterated colimit of compact acyclics and so 
        \[\Map(A, LX\otimes Y) \simeq *.\]

        By Proposition \ref{prop: monoidal + local mult = smashing}, the fact that $LX \otimes Y$ is local allows us to conclude that $L$ is smashing. 
    \end{proof}

    \begin{cor}[Miller, \cite{millerfinite} Proposition 9]\label{cor: sp fin implies smashing}
        In $\Sp$, and in any category arising as a smashing localisation of $\Sp$, all finite localisations are smashing.
    \end{cor}

    \begin{proof}
        First prove the result for $\Sp$. Because $\Sp$ is rigid (\cite{axiomatic} Theorem 2.1.3(d) and Example 1.2.3(a)), its compact objects are dualisable. By Lemma \ref{lem: fin implies smashing}, it is sufficient to show that all localisations are tensor-compatible. By Proposition \ref{prop: when all locs are monoidal}, this is true in $\Sp$ since the monoidal unit generates the whole category under colimits and desuspensions. Passing to a smashing localisation of $\Sp$ preserves all colimits and desuspensions, so such a category is also generated by its monoidal unit. Moreover a smashing localisation of a stable rigid category remains rigid, since smashing localisations preserve both compact objects and the symmetric monoidal structure. 
    \end{proof}

    \begin{defn}[\cite{bousfield} Section 1]\label{def: bousfield loc}
        An object $E \in \CC$ determines a \emph{Bousfield localisation} of $\CC$, denoted $L_E$. This is the localisation of $\CC$ whose acyclic objects consist of those $Z \in \CC$ such that $Z \otimes E \simeq 0$. There is an equivalence relation on objects of $\CC$ where two objects lie in the same \emph{Bousfield class} if their Bousfield localisations are equivalent.
    \end{defn}

    Bousfield's work \cite{bousfield} treats Bousfield localisation in the specific case $\CC = \Sp$, and Section 3 of \cite{axiomatic} addresses a general stable $\infty$-category $\CC$.

    \begin{cor}\label{cor: bous fin implies smashing}
        On a stable category $\CC$ whose compact objects are all dualisable, any finite Bousfield localisation is smashing.
    \end{cor}

    \begin{proof}
        By definition, the acyclic objects of a Bousfield localisation form a tensor ideal, so a Bousfield localisation of a stable category is tensor-compatible (see Lemma \ref{lem: equiv monoidal loc}). We conclude by Lemma \ref{lem: fin implies smashing} that a finite Bousfield localisation is smashing.
    \end{proof}

    We now use Lemma \ref{lem: baralpha in terms of alpha} to prove a useful fact about the acyclic objects for a localisation presented by a central map. 

\begin{lem}\label{lem: compact acyclics generate}
    Let $\CC$ be stable and rigid. If $\alpha: \one \to A$ is central and $A$ is compact, then $\cof(\one \to J_\alpha \simeq L\one)$ is a sequential colimit of compact $L$-acyclic objects in $\CC$.
\end{lem}

\begin{proof}
    Since $\CC$ is rigid, $\one$ is compact, so $\cof \alpha$ is compact because $A$ is. Applying Lemma \ref{lem: baralpha in terms of alpha}, it is sufficient to prove that $\cof(\alpha^{\otimes n})$ is compact acyclic in $\CC$ for all $n$. By Lemma \ref{lem: baralpha in terms of alpha} we also know $\cof(\alpha^{\otimes n})$ is built in finitely many steps from the compact objects $A$ and $\cof \alpha$, via iterated extensions and tensor products. By rigidity of $\CC$, finite tensor products of compact objects are compact (see Definition \ref{def: dualisable and rigid}). An extension of finitely many compact objects is compact in any stable category. Therefore $\cof(\alpha^{\otimes n})$ is compact.

    It follows from the proof of Lemma \ref{lem: Jalpha' localises to Jalpha} that every composite map in the diagram $\mathfrak{J}_\alpha$ is an $L$-local equivalence. In particular, $\alpha^{\otimes n}: \one \to A^{\otimes n}$ is one of the composites in this diagram, so it is a local equivalence. The cofibre of a local equivalence is acyclic.
\end{proof}

    \begin{lem}\label{lem: compactly central implies finite}
        Let $\CC$ be stable and rigid. Then any compactly central localisation of $\CC$ is finite.
    \end{lem}
    
    \begin{proof}
        A compactly central localisation $L$ is presented by some central map $\alpha: \one \to A$ whose cofibre is compact. Then $L$ is naturally equivalent to the functor given by tensoring with the idempotent algebra $J_\alpha$. In particular we know $L$ is smashing and therefore tensor-compatible. Let $C = \cof(\one \to J_\alpha)$, and observe that $C$ is acyclic as the cofibre of a local equivalence. For any object $X \in \CC$ we have the fibre sequence 
        \[C \otimes \Sigma^{-1}X \to X \to LX \simeq J_\alpha \otimes X,\]
        computing the acyclic and local parts of $X$. Hence any acyclic has the form $C \otimes X$ for some $X \in \CC$, because every acyclic is the acyclic part of itself. By Lemma \ref{lem: compact acyclics generate} we can write $C \simeq \colim_n C_n$ as a colimit of compact acyclics, and because $\CC$ is compactly generated we can write $X \simeq \colim_{j \in J} B_j$ as a filtered colimit of compact objects, not necessarily acyclic. Then
        \[C \otimes X \simeq \colim_n C_n \otimes \colim_j B_j \simeq \colim_{n, j} C_n \otimes B_j,\]
        and each $C_n \otimes B_j$ is compact acyclic. Acyclicity follows because $L$ is tensor-compatible, so the acyclic objects form a tensor ideal (Lemma \ref{lem: equiv monoidal loc} \ref{cond: tensor ideal}), and $C_n$ itself is $L$-acyclic. The tensor product of two compact objects is compact because $\CC$ is rigid (see Definition \ref{def: dualisable and rigid}). Thus $C \otimes X$ is a filtered colimit of compact acyclics, and therefore $L$ is a finite localisation.
    \end{proof}

    \begin{cor}
        For a localisation $L$ of a stable rigid category $\CC$, 
        \[\text{compactly central} \implies (\text{finite and tensor-compatible}) \implies \text{smashing}.\]
    \end{cor}

    The second implication is often strict. For example, the chromatic localisation $L_n$ is smashing (Hopkins-Ravenel, see \cite{ravenel92} Theorem 7.5.6), but the disproof of the telescope conjecture \cite{telescope} shows that for $n \geq 2$ it is never finite, i.e. $L_n \not\simeq L_n^f$. See the introduction to Section \ref{sec: the Lnf example} for background on $L_n$ and $L_n^f$. The first implication can also be strict. In $\Sp$, being compactly central is a strictly stronger condition than finiteness, but in $\Sp_{(p)}$ they are equivalent (Theorems \ref{thm: not compactly central} and \ref{thm: cf class in p-local spectra}). A nonzero compactly central localisation in $\Sp$ turns out to be one which restricts to the identity at all but finitely many primes. In $\Sp_{(p)}$ we have already inverted all the primes except $p$, so every finite localisation is compactly central. Proving these results requires two key ingredients. One is showing that the localisation $L_n^f$ is compactly central. The other is showing that localisations of $\Sp$ which invert an infinite set of primes -- such as $p$-localisation and rationalisation -- cannot be compactly central.

    There is a universal process by which we can take any localisation and produce a finite localisation. 
    \begin{defn}[Miller, \cite{millerfinite}]\label{def: finitise}
        Let $L$ be a localisation of $\CC$. The \emph{finitisation} of $L$, denoted $L^f$, is a new localisation of $\CC$ which we define by specifying its acyclics. The acyclics of $L^f$ are generated under colimits by the acyclics of $L$ which are compact in $\CC$. 
    \end{defn}

    A full subcategory of $\CC$ is the kernel of an accessible localisation if and only if it is of the form $\langle N \rangle$ for some set $N$ of objects of $\CC$ (see \cite{htt} Section 5.5.4 and Lemma \ref{lem: acyclics gen by}). Since the compact objects of $\CC$ form an essentially small collection, it follows that even when $L$ is not accessible, we obtain an accessible localisation $L^f$. Essentially by definition, $L^f$ is a finite localisation. The following Lemma is also essentially due to Miller (\cite{millerfinite}).

    \begin{lem}
        There is a natural transformation $L^f \implies L$. In fact $L \leq L^f$.
    \end{lem}

    \begin{proof}
        Since the acyclics of $L^f$ are a full subcategory of the acyclics of $L$, we have the reverse inclusion on local objects and thus $L \leq L^f$. Moreover there are subcategory inclusions 
        \[L\CC \into L^f\CC \into \CC.\] 
        This means that every $L$-local object is already $L^f$-local, hence $L^f \circ L \simeq L$. In fact the whole adjunction for $L$ factors through the adjunction for $L^f$, producing a new localisation functor $L|_{L^f\CC}: L^f\CC \to L\CC$ which is left adjoint to the subcategory inclusion $L\CC \into L^f\CC$. The unit of the adjunction for this new localisation gives the natural transformation $L^f \implies L$. Recall that the adjunctions for the two existing localisations give comparisons $\id_\CC \implies L$ and $\id_{\CC} \implies L^f$.
    \end{proof}

    If $L$ is already a finite localisation, then $L \simeq L^f$ because they have the same collection of acyclics, and the comparison map is an equivalence between them, given by the identity on objects and morphisms of $\CC$. Otherwise, $L^f > L$, meaning $L^f$ has more local objects, or equivalently fewer acyclic objects. 

    \begin{lem}\label{lem: finitisation univ prop}
        $L^f$ is the minimal finite localisation satisfying $L \leq L^f$.
    \end{lem}

    \begin{proof}
        Let $F$ be another finite localisation. If $L \leq F$ then the compact $F$-acyclic objects of $\CC$ are also $L$-acyclic. But all compact $L$-acyclic objects are $L^f$-acyclic by definition. Since $F$ is finite, any $F$-acyclic object is generated by the compact $F$-acyclics, hence any $F$-acyclic object is generated by the compact $L^f$-acyclics. Therefore any $F$-acyclic is $L^f$-acyclic and we have $L^f \leq F$. 
    \end{proof}

    \begin{remark}
        This universal property of finitisation can be rephrased as follows. The finite localisation $L^f$ factors $L$ as $L \simeq L|_{L^f\CC} \circ L^f$, where $L|_{L^f\CC}$ is a further localisation of $L^f\CC$. If $F$ is any finite localisation which factors $L$ as $L \simeq L|_{F\CC} \circ F$, then there is a further factorisation $L^f \simeq L^f|_{F\CC} \circ F$. The comparison $F \implies L$ factors as $F \implies L^f \implies L$.
    \end{remark}

    \begin{remark}\label{rmk: setup smashification from finitisation}
        Finitisation gives a method for producing a finite localisation from any localisation. This is universal in the sense that it is the closest approximation by a finite localisation from above. When $\CC$ is reasonably nice, finite localisations are smashing (see Lemma \ref{lem: fin implies smashing} and Corollaries \ref{cor: sp fin implies smashing} and \ref{cor: bous fin implies smashing}). This means we can functorially produce a smashing localisation from any given localisation. 
        
        However, as a smashing localisation $L^f$ need not be universal. There can exist a non-finite smashing localisation sitting between $L^f$ and $L$. Example \ref{ex: smashification of Kn} demonstrates this discrepancy. In Subsection \ref{subsec: univ smash loc}, we develop a construction analogous to finitisation which outputs the closest approximation from above by a smashing localisation (Definition \ref{def: smashification} and Theorem \ref{thm: smashification}). 
    \end{remark}

\subsection{Universal smashing localisations}\label{subsec: univ smash loc}

In this Subsection we outline a construction of the universal smashing localisation approximating a given localisation, with a universal property analogous to that of finitisation. Balmer-Krause-Stevenson \cite{tensorlocframe} is the direct antecedent to these results. They consider arbitrary tensor products of smashing localisations, as we do in Lemma \ref{lem: tensor smashing maximal}, Corollary \ref{cor: tensor smashing maximal multi}, and Lemma \ref{lem: tensor smashing maximal big}. In particular \cite{tensorlocframe} Proposition 2.8, Proposition 5.2, Remark 5.12, and Theorem 1.1 are relevant to the material presented here. They work in a compactly generated rigid setting and they order idempotents in the opposite way to us. They restrict attention to smashing localisations, and do not address any of the statements made here relating to localisations that are not smashing, including Theorem \ref{thm: smashification}.

\begin{lem}[cf \cite{ambidexterity} Proposition 5.1.8, \cite{tensorlocframe} Proposition 2.8]\label{lem: tensor smashing maximal}
    Let $L_i$ be a smashing localisation for $i = 1, 2$. Then $L_1 \otimes L_2$ is the maximal tensor-compatible localisation that satisfies $L_1 \otimes L_2 \leq L_i$ for $i = 1, 2$. Moreover it is smashing, hence it is also the maximal such smashing localisation.
\end{lem}

\begin{proof}
    We use Lemma \ref{lem: smash loc partial order via ring map}. The ring map $L_1\one \xrightarrow{\id_1 \otimes e_2} L_1\one \otimes L_2\one$ tells us that $L_1 \otimes L_2 \leq L_1$, and symmetrically $L_1 \otimes L_2 \leq L_2$. 

    If $L_3 \leq L_1$ and $L_3 \leq L_2$ with $L_3$ some tensor-compatible localisation, then we obtain ring maps $L_1\one \to L_3\one$ and $L_2\one \to L_3\one$. Tensoring these together and postcomposing with the multiplication on $L_3\one$, we have a ring map 
    \[L_1\one \otimes L_2\one \to L_3\one \otimes L_3\one \xrightarrow{\mu_3} L_3\one,\]
    and thus $L_3 \leq L_1 \otimes L_2$. This establishes maximality of $L_1 \otimes L_2$. It is necessary here that $L_1\one \otimes L_2\one$ be idempotent, i.e. we need both $L_1$ and $L_2$ to be smashing. 
\end{proof}

\begin{cor}\label{cor: tensor smashing maximal multi}
    For a finite collection of smashing localisations $\{L_i\}_{i=1}^N$, the product $L := \bigotimes_{i=1}^N L_i$ is both the maximal tensor-compatible and the maximal smashing localisation with the property that $L \leq L_i$ for all $1 \leq i \leq N$.
\end{cor}

Corollary \ref{cor: tensor smashing maximal multi} follows immediately by induction from Lemma \ref{lem: tensor smashing maximal}. We next consider infinite tensor products of smashing localisations.

\begin{defn}[cf \cite{tensorlocframe} Proposition 5.2]
    Let $I$ be an arbitrary indexing set and $L_i$ be a smashing localisation for each $i \in I$. For finite sets $A \subseteq B \subseteq I$, we have a canonical map
    \[\bigotimes_{a \in A} L_a\one \to \bigotimes_{b \in B} L_b\one\]
    given by tensoring together the components $\id_a$ for $a \in A \subseteq B$ and $e_b: \one \to L_b\one$ for $b \in B\setminus A$. Consider the poset of finite subsets of $I$ ordered by inclusion, and define the tensor product of idempotent algebras indexed by $I$ as the filtered colimit over this poset, 
    \[\bigotimes_{i\in I} L_i\one := \underset{\text{fin}\, A \subseteq I}{\colim} \bigotimes_{a \in A} L_a\one.\]
    Let $\bigotimes_{i\in I} L_i$ denote the smashing localisation whose idempotent algebra is $\bigotimes_{i\in I} L_i\one$.
\end{defn}

\begin{lem}[cf \cite{tensorlocframe} Proposition 5.2 and Remark 5.12]\label{lem: tensor smashing maximal big}
    For any indexing set $I$ with each $L_i$ a smashing localisation, $L := \bigotimes_{i \in I} L_i$ is the maximal tensor-compatible and maximal smashing localisation that satisfies $L \leq L_i$ for all $i \in I$.
\end{lem}

\begin{proof}
    A filtered colimit of $\E_\infty$-algebras can be computed on the level of underlying objects (\cite{ha} Corollary 3.2.3.2), and indeed a filtered colimit of idempotent algebras remains idempotent, so $L$ defines a smashing localisation. Since any individual $i \in I$ is a one-element finite subset of $I$, each $L_i$ belongs to the diagram over which $L$ is a colimit and we have maps $L_i\one \to L\one$. Hence $L \leq L_i$ for all $i \in I$.

    If $L'$ is any tensor-compatible localisation satisfying $L' \leq L_i$ for all $i \in I$, then by Lemma \ref{lem: smash loc partial order via ring map} we have ring maps $L_i\one \to L'\one$ for all $i \in I$. By taking finite tensor products of these we obtain ring maps 
    \[\bigotimes_{a \in A} L_a\one \to L'\one\]
    for every finite subset $A$ of $I$, which are compatible with inclusions of finite subsets. This family of maps uniquely determines a map $L\one \to L'\one$ since $L\one$ is the colimit. Hence $L' \leq L$ so $L$ is maximal as claimed. We are using here that $L$ is smashing and $L'$ is tensor-compatible, to apply the full strength of Lemma \ref{lem: smash loc partial order via ring map}. 
\end{proof}

\begin{lem}\label{lem: univ loc is univ mon loc}
    Let $\CC$ be a presentably symmetric monoidal stable category. Given a set of accessible localisations $\{L_i\}_{i \in I}$ on $\CC$, there is a maximal localisation $L$ with the property $L \leq L_i$ for all $i \in I$. There is also a maximal tensor-compatible localisation $L^{mon}$ such that $L^{mon} \leq L_i$ for all $i \in I$. Moreover $L$ and $L^{mon}$ agree when all the $L_i$ are tensor-compatible.
\end{lem}

\begin{proof}
    The condition $L \leq L_i$ is equivalent to requiring that the $L_i$-acyclics be contained in the $L$-acyclics. Moreover the $L$-acyclics must be closed under colimits to determine a localisation. So, take the union of all the subcategories of acyclics for the various localisations $L_i$, and the subcategory this generates under colimits is the $L$-acyclics. This is the same as the smallest localising subcategory of $\CC$ containing the $L_i$-acyclics for all $i \in I$. A localising subcategory generated by a set of objects is the kernel of an accessible localisation (see discussion following Lemma \ref{lem: acyclics gen by}). Each subcategory of $L_i$-acyclics is generated by a set of objects, since the $L_i$ are accessible, so the union of these generating sets generates the $L$-acyclics, given that $I$ is a set. We thus obtain an accessible localisation $L$. Since $\ker L$ was generated under colimits by the union of the kernels of the $L_i$, we see that $L$ is maximal under the conditions $L \leq L_i$.

    Since $\CC$ is stable, a localisation is tensor-compatible if and only if its acyclics form a tensor ideal, by Lemma \ref{lem: equiv monoidal loc} \ref{cond: tensor ideal}. Let $S$ be the union of the generating sets for the $L_i$-acyclics chosen above, and let $G$ be a set of objects generating $\CC$ under colimits and desuspensions. Let $\mathscr{B}$ be the localising subcategory generated by $\{s \otimes g : s \in S, g \in G\}$, which is contained in every localising tensor ideal containing $S$. To construct $L^{mon}$ it suffices to show that $\mathscr{B}$ is a tensor ideal. For fixed $s \in S$, the full subcategory of objects $X$ with $s \otimes X \in \mathscr{B}$ is closed under colimits and desuspensions and contains $G$, so it is all of $\CC$. In particular $S \subseteq \mathscr{B}$. The full subcategory of objects $A$ such that $A \otimes X \in \mathscr{B}$ for every $X \in \CC$ is closed under colimits and desuspensions, and it contains each $s \otimes g$ so it contains $\mathscr{B}$. Hence $\mathscr{B}$ is a tensor ideal, and thus it is the kernel of the tensor-compatible localisation $L^{mon}$.

    Next we want to show that if each $L_i$ is tensor-compatible then $L$ is automatically tensor-compatible. Since $\CC$ is stable, by assumption each subcategory of $L_i$-acyclics is a tensor ideal. Fix an object $X$, and consider the full subcategory $\AA$ of objects $A$ such that $A \otimes X$ is $L$-acyclic. The symmetric monoidal structure on $\CC$ commutes with colimits and the $L$-acyclics are closed under colimits, so $\AA$ is closed under colimits. If $A$ is $L_i$-acyclic then $A \otimes X$ is $L_i$-acyclic because the $L_i$-acyclics form a tensor ideal, and hence $A \otimes X$ is $L$-acyclic. So $\AA$ contains all $L_i$-acyclics, for all $i$. Since the $L$-acyclics are generated under colimits by the $L_i$-acyclics, it follows that $A \otimes X$ is $L$-acyclic for every $L$-acyclic $A$. Thus $L$ is tensor-compatible, so $L$ is naturally equivalent to $L^{mon}$.
\end{proof}

\begin{remark}
    The same proof works in the unstable setting after translating all of the statements about acyclics into corresponding statements about local equivalences. That is, we need to use Condition \ref{cond: lurie monoidal} of Lemma \ref{lem: equiv monoidal loc} instead of Condition \ref{cond: tensor ideal} to characterise tensor-compatible localisations.
\end{remark}

\begin{defn}\label{def: smashification}
    The \emph{smashification} of a localisation $L$ is the minimal smashing localisation $L^{sm}$ which satisfies $L \leq L^{sm}$. 
\end{defn}

\begin{thm}\label{thm: smashification}
    Let $\CC$ be stable. Every localisation $L$ of $\CC$ has a smashification, given by the tensor product over all smashing localisations $L'$ satisfying $L \leq L'$.
\end{thm}

\begin{proof}
    Let $\Idem(L)$ be the collection of all smashing localisations $L'$ with $L \leq L'$, and let $L^{sm?}$ denote the tensor product of all the localisations in $\Idem(L)$. This is well-defined because by Ohkawa's theorem there is a set of Bousfield classes in $\CC$, and any smashing localisation is a Bousfield localisation so $\Idem(L)$ is a set. See \cite{ohkawa} for Ohkawa's original proof in $\Sp$, \cite{dwyerohkawa} for a simpler reformulation of the $\Sp$ case, and Iyengar-Krause \cite{iyengarkrause} Theorem 3.1 for a proof in a general stable setting. By \cite{rosicky} Proposition 6.10 (Theorem 4.9 and Remark 3.4(1) in the arXiv version of the same paper), every presentably symmetric monoidal stable $\infty$-category is \emph{well generated} in the sense used in Iyengar-Krause, so their Theorem 3.1 applies. There is a simpler proof that $\Idem(L)$ is a set in the case where $\CC$ is compactly generated (Krause, \cite{krause-smashing} Corollary to Theorem A). Note that our smashing localisations are a subset of Krause's, with the two notions agreeing in $\Sp$ but not in general. 

    Observe that any smashing localisation $L'$ of a presentable $\infty$-category $\CC$ is accessible. The functor $L'\one \otimes -: \CC \to \CC$ preserves colimits. Hence the class of maps it inverts, i.e. the class of $L'$-local equivalences, is generated by a set of morphisms (\cite{htt} Proposition 5.5.4.16).
    
    By Lemma \ref{lem: tensor smashing maximal big} we know $L^{sm?}$ is smashing, and moreover it is both the maximal smashing localisation and the maximal tensor-compatible localisation bounded above by $\Idem(L)$, i.e. satisfying $L^{sm?} \leq L'$ for all $L' \in \Idem(L)$. By Lemma \ref{lem: univ loc is univ mon loc}, since each $L' \in \Idem(L)$ is smashing and thus accessible and tensor-compatible, and $L^{sm?}$ is maximal among tensor-compatible localisations bounded above by $\Idem(L)$, we see that $L^{sm?}$ is also maximal among \emph{all} localisations bounded above by $\Idem(L)$.

    But $L$ itself is a localisation bounded above by $\Idem(L)$, and $L^{sm?}$ is the maximal such, so we have $L \leq L^{sm?}$ as desired. It is clear from the definition of $\Idem(L)$ that $L^{sm?}$ is then the minimal smashing localisation satisfying $L \leq L^{sm?}$. Hence $L^{sm?} \simeq L^{sm}$ is the smashification of $L$.
\end{proof}

\begin{example}\label{ex: smashification of Kn}
See the introduction to Section \ref{sec: the Lnf example} for background on $K(n)$ and $E_n$ relevant to this example. See also Definition \ref{def: bousfield loc}. For $1 \leq n < \infty$, we show that the smashification of $L_{K(n)}$ is the chromatic localisation $L_n$. We work in $\Sp_{(p)}$.

Since $L_{E_n} \simeq L_{K(0) \vee \cdots \vee K(n)}$ (\cite{hoveychromatic} Corollary 1.12 and \cite{ravenel84} Theorem 2.1(d)), $L_{K(n)} \leq L_n$. Because $p$ acts by zero on $K(n)_*$ and invertibly on $\pi_* H\Q$, it follows that $K(n) \otimes H\Q \simeq 0$, whereas $K(0) \otimes H\Q \simeq H\Q \not\simeq 0$. Thus $H\Q$ is $K(n)$-acyclic but not $E_n$-acyclic, so $L_{K(n)} < L_n$.

Let $L'$ be a smashing localisation with $L_{K(n)} \leq L'$. Then $L' \not\simeq 0$, and if $L'$ is the identity, $L_n \leq L'$. Otherwise write $L_m \leq L' \leq L_m^f$ (\cite{barthel} Proposition 4.4), and then $L_{K(n)} \leq L' \leq L_m^f$ tells us that $m \geq n$. Hence $L_n \leq L_m \leq L'$. Since $L_n$ is smashing by Hopkins-Ravenel (\cite{ravenel92} Theorem 7.5.6), $L_n$ is the smashification of $L_{K(n)}$ as claimed.

In fact, by the disproof of the telescope conjecture in \cite{telescope} we have 
$L_{K(n)} < L_n < L_n^f \simeq L_{K(n)}^f$ for $n \geq 2$. The smashification of $L_{K(n)}$ differs from its finitisation, proving a claim made in Remark \ref{rmk: setup smashification from finitisation}.
\end{example}

\begin{lem}\label{lem: composite of cc locs is cc}
    Let $J_\alpha = \Free_{\E_0}^{\E_1}(A, \alpha)$ for a map $\alpha: \one \to A$, and recall that if $\alpha$ is central then $J_\alpha$ is an idempotent algebra as per Corollary \ref{cor: Jalpha is idempotent}. For a finite family of central maps $\{\alpha_i\}_{i=1}^n$,
    \[\alpha_{[1, n]} := \bigotimes_{i=1}^n \alpha_i \text{ is central, and } J_{\alpha_{[1, n]}} \simeq \bigotimes_{i=1}^n J_{\alpha_i}.\]
    If $\CC$ is stable and rigid, and each $\alpha_i$ is compact, then $\alpha_{[1, n]}$ is compact. Hence for stable rigid $\CC$, the composite of a finite family of compactly central localisations is compactly central.
\end{lem}

\begin{proof}
    Since each $\alpha_i$ is central, we have $\alpha_i \otimes \id_i \simeq \id_i \otimes \alpha_i$ where $\alpha_i: \one \to A_i$ and $\id_i := \id_{A_i}$. Let $A_{[1, n]} = \bigotimes_{i=1}^n A_i$, and $A_{[n, 1]}$ denote the same tensor product with the factors in reverse order. By induction we may assume $\alpha_{[1, n-1]}$ is central. Let $\tau_{n+1}$ denote the permutation on a tensor product of $n+1$ terms which is a rotation bringing the last entry to the front. Then the commuting diagram 
    \[\begin{tikzcd}[column sep={15em,between origins}, row sep={6em, between origins}]
        &  {A_{[1, n]} \otimes A_{[1, n]}} \arrow[d, "{\id_{[1, n-1]} \otimes \tau_{n+1}^{-1}}"]         \\
{A_{[1, n]}}    \arrow[r, "{\id_{[1, n-1]} \otimes \alpha_{[1, n]} \otimes \id_n}" description, bend left] 
                \arrow[r, "{\alpha_{[1, n-1]} \otimes \id_{[1, n]} \otimes \alpha_n}" description, bend right] 
                \arrow[rd, "{\alpha_{[1, n]} \otimes \id_{[1, n]}}" description, bend right] 
                \arrow[ru, "{\id_{[1, n]} \otimes \alpha_{[1, n]}}" description, bend left] 
        & {A_{[1, n-1]} \otimes A_{[1, n]} \otimes A_n} \arrow[d, "{\id_{[1, n-1]} \otimes \tau_{n+1}}"] \\
        & {A_{[1, n]} \otimes A_{[1, n]}}                                                                             
\end{tikzcd}\]

    shows that $\alpha_{[1, n]}$ is central. The vertical composite on the right is the identity on $A_{[1, n]} \otimes A_{[1, n]}$ and the middle two horizontal maps are homotopic to one another by centrality of $\alpha_{[1, n-1]}$ and of $\alpha_n$.

    Having established centrality of $\alpha_{[1, n]}$, we know $J_{\alpha_{[1, n]}}$ is an idempotent algebra. 

    To compare $J_{\alpha_{[1, n]}}$ and $\bigotimes_{i=1}^n J_{\alpha_i}$, first we use the universal property of $J_{\alpha_{[1, n]}}$. Let $c_i: A_i \to J_{\alpha_i}$ denote the canonical map obeying $e_i \simeq c_i \circ \alpha_i$, where $e_i: \one \to J_{\alpha_i}$ is the unit of the $\E_\infty$-ring $J_{\alpha_i}$. We have the commuting diagram 
    \[\begin{tikzcd}[sep={7em, between origins}]
        \one    \arrow[r, "{\alpha_{[1, n]}}"] 
                \arrow[rd, "\bigotimes_{i=1}^n e_i"'] 
        & \bigotimes_{i=1}^n A_i    \arrow[d, "\bigotimes_{i=1}^n c_i", pos=0.2] 
                                    \arrow[r, "{c_{[1, n]}}"] & {J_{\alpha_{[1, n]}}} 
                                    \arrow[ld, "\exists!\phi", dashed] \\
        & \bigotimes_{i=1}^n J_{\alpha_i}
        \end{tikzcd}\]
    and by the universal property of $J_{\alpha_{[1, n]}}$ we obtain the $\E_1$-ring map $\phi$. By Lemma \ref{lem: En refinement}, since $\phi$ is an $\E_0$ map it refines uniquely to an $\E_\infty$ map. Note $\bigotimes_{i=1}^n e_i$ is the unit for the $\E_\infty$-ring $\bigotimes_{i=1}^n J_{\alpha_i}$. 

    Conversely, for any $1 \leq i \leq n$ we have a commuting diagram 
        \[\begin{tikzcd}[sep={7em, between origins}]
            \one    \arrow[r, "\alpha_i"] 
                    \arrow[rdd, "{e_{[1, n]}}"', bend right] 
                    \arrow[rd, "{\alpha_{[1, n]}}"' description] 
            & A_i   \arrow[r, "c_i"] 
                    \arrow[d, "\ell"]          
            & J_{\alpha_i} \arrow[ldd, "\exists! \psi_i", dashed, bend left] \\
            & \bigotimes_{j=1}^n A_j \arrow[d, "{c_{[1, n]}}" description] \\
            & {J_{\alpha_{[1, n]}}}                                              
            \end{tikzcd}\]
    with $\ell$ denoting the obvious map $\ell = \bigotimes_{j=1}^{i-1}\alpha_j \otimes \id_i \otimes \bigotimes_{j=i+1}^n \alpha_j$. We get induced $\E_1$ maps $\psi_i: J_{\alpha_i} \to J_{\alpha_{[1, n]}}$ for each $i$, which as before refine uniquely to $\E_\infty$. Then we may tensor together the $\psi_i$ to obtain an $\E_\infty$ map $\psi: \bigotimes_{i=1}^n J_{\alpha_i} \to J_{\alpha_{[1, n]}}$. 

    The maps $\psi$ and $\phi$ are automatically inverse, because each composite is an $\E_0$ endomorphism of an idempotent algebra, and by Corollary \ref{cor: very useful En refinement} idempotent algebras have no non-identity $\E_0$ endomorphisms.

    Now assume $\CC$ is stable and rigid. We claim that for compact maps $f: \one \to A$ and $g: \one \to B$, $f \otimes g$ is also compact. It follows inductively that if each $\alpha_i$ is compact, $\alpha_{[1, n]}$ must be compact. We can write $f \otimes g \simeq (\id_A \otimes g) \circ f$ and the cofibre of a composite is an extension of the two individual cofibres, so $\cof(f \otimes g)$ is an extension of $\cof f$ and $A \otimes \cof g$. Since $f$ and $g$ are both compact, $\cof f$ and $\cof g$ are compact. Since $\CC$ is rigid, $\one$ is compact and the tensor product of two compact objects is compact (see Definition \ref{def: dualisable and rigid}). Thus $A$, $B$, and $A \otimes \cof g$ are compact. Finally $\cof(f \otimes g)$ is compact because it is the extension of two compact objects.
\end{proof}

We can interpret smashification (resp. finitisation) as the right adjoint to the forgetful functor from smashing (resp. finite) localisations to all localisations. To make this precise, let $\Loc_\CC$ denote the category of localisations on $\CC$, thinking of each localisation $L_i$ as an endofunctor of $\CC$ equipped with its unit $\eta_i: \id_\CC \implies L_i$. Define the mapping space from $L_1$ to $L_2$ to be the space of natural transformations $L_1 \implies L_2$ which are compatible with the units, i.e. $\Map_{\Loc_\CC}(L_1, L_2) := \Map((L_1, \eta_1), (L_2, \eta_2))$ in the category $\End(\CC)_{\id_\CC/}$.

\begin{lem}[\cite{axiomatic} Lemma 3.1.5]
    If $L_1$ and $L_2$ are localisations of $\CC$, then the mapping space $\Map_{\Loc_\CC}(L_1, L_2)$ is either contractible or empty. It is nonempty precisely when $L_2 \leq L_1$. Hence $\Loc_\CC$ is a poset, whose partial order recovers the one we already defined.
\end{lem}

\begin{proof}
    We prove two claims: first, if $\Map_{\Loc_\CC}(L_1, L_2)$ is nonempty then $L_2 \leq L_1$; and second, if $L_2 \leq L_1$ then $\Map_{\Loc_\CC}(L_1, L_2)$ is contractible. 

    Let $\phi: L_1 \implies L_2$ be compatible with the units. If $Y$ is $L_2$-local then $\phi_Y \circ (\eta_1)_Y \simeq (\eta_2)_Y$ and this map is an equivalence in $\CC$, so $Y$ is a retract of the $L_1$-local object $L_1 Y$ and is therefore $L_1$-local. Thus $L_2 \leq L_1$. 

    Now suppose $L_2 \leq L_1$, so $L_2$ takes values in $L_1$-local objects. Postcomposition with $L_1$ is a localisation of $\End(\CC)$ onto the full subcategory of functors with $L_1$-local values, with unit induced by $\eta_1$. Precomposition with $\eta_1$ therefore gives an equivalence $\Map_{\End(\CC)}(L_1, L_2) \xrightarrow{\simeq} \Map_{\End(\CC)}(\id_\CC, L_2)$. The pullback square
    \[\begin{tikzcd}[sep={12em, between origins}]
        \Map_{\Loc_\CC}(L_1, L_2) \arrow[r] \arrow[d] & * \arrow[d, "\eta_2"] \\
        \Map_{\End(\CC)}(L_1, L_2) \arrow[r, "- \circ \eta_1", "\simeq"'] & \Map_{\End(\CC)}(\id_\CC, L_2)
    \end{tikzcd}\]
    computes $\Map_{\Loc_\CC}(L_1, L_2)$, so $\Map_{\Loc_\CC}(L_1, L_2) \simeq *$.
\end{proof}

As we have seen, there are forgetful inclusions of full subposets
\[\Loc_\CC^{fin, \otimes} \subseteq \Loc^{sm}_\CC \subseteq \Loc^{\otimes}_\CC \subseteq \Loc_\CC,\]
where $\Loc_\CC^{fin, \otimes}$ denotes finite and tensor-compatible localisations, and $\Loc^{sm}$ the smashing ones. Here we are assuming $\CC$ is stable and rigid. The property that $L^{sm}$ is the minimal smashing localisation satisfying $L \leq L^{sm}$ is equivalent to existence of a map $L^{sm} \implies L$ together with the condition that any smashing localisation $L'$ with a map $L' \implies L$ factors as $L' \implies L^{sm} \implies L$. That is, $\Map(L', L^{sm}) \simeq \Map(L', L)$ whenever $L'$ is a smashing localisation. Since all the mapping spaces involved are contractible or empty, any factorisation is automatically unique. Smashification is therefore right adjoint to the forgetful functor from smashing localisations to all localisations on $\CC$. Similarly, finitisation is right adjoint to the forgetful functor from finite localisations $\Loc_{\CC}^{fin}$ to all localisations.

\section{\texorpdfstring{$L_n^f$}{Lnf} is a compactly central localisation}\label{sec: the Lnf example}

In this Section, we chiefly work in the category $\Sp_{(p)}$ of $p$-local spectra. Since $p$-localisation is smashing, many useful properties of $\Sp$ regarding localisations carry over to our setting. The key results of this Section include a classification of all compact central maps in $p$-local spectra, and computation of the compactly central smashing localisations they present. In particular, we show that $L_{n}^f$ localisation is compactly central, as described in Subsection \ref{subsec: finiteness}. This is generally a stronger condition than finiteness. At the end of the Section, we classify all compactly central localisations in $\Sp_{(p)}$ and in $\Sp$. 

We start with a brief recollection on Morava $K$-theory. Fix a prime $p$. For $0 < n < \infty$, the $n$-th Morava $K$-theory $K(n)$ is a $p$-local complex-oriented ring spectrum with coefficient ring $K(n)_* = \F_p[v_n^{\pm 1}]$, where $|v_n| = 2(p^n - 1)$, and by convention $K(0) = H\Q$ and $K(\infty) = H\F_p$. Since $K(n)_*$ is a graded field, there is a K\"unneth isomorphism $K(n)_*(X \otimes Y) \cong K(n)_*X \otimes_{K(n)_*} K(n)_*Y$. The Morava $K$-theories $K(n)$ for $0 \leq n \leq \infty$ jointly detect nilpotence, and underlie the thick subcategory theorem (\cite{NilpotenceII} Theorems 3 and 7). A $p$-local finite spectrum which is $K(n)$-acyclic is automatically $K(n-1)$-acyclic (\cite{ravenel84} Theorem 2.11), so we say that a $p$-local finite spectrum $X$ has \emph{type} $m < \infty$ if $K(n)_*X = 0$ for $n < m$ and $K(n)_*X \neq 0$ for $m \leq n \leq \infty$. If $K(n)_* X = 0$ for all $n < \infty$ then $X \simeq 0$, and we say $0$ has type $\infty$. We write $E_n$ for Morava $E$-theory at height $n$, an even periodic ring spectrum with $\pi_0 E_n = W(\F_{p^n})[[u_1, \dots, u_{n-1}]]$. The spectrum $E_n$ has the same Bousfield class (Definition \ref{def: bousfield loc}) as $K(0) \vee K(1) \vee \cdots \vee K(n)$ (see \cite{hoveychromatic} Corollary 1.12 and \cite{ravenel84} Theorem 2.1(d)), and the Bousfield localisation $L_n := L_{E_n}$ at $E_n$ is smashing (Hopkins-Ravenel, see \cite{ravenel92} Theorem 7.5.6). For the basic properties of $K(n)$ claimed here, see \cite{ravenel92} Proposition 1.5.2. For more details on these spectra and their role in chromatic homotopy theory, see \cite{ravenel92} Chapter 7.

We next discuss the relationship between $K(n)$, $E_n$, and localisation. The kernel of a localisation $L$ is a thick subcategory. If $L$ is in addition finite, then its finite acyclics generate all acyclics and form a thick subcategory of $p$-local finite spectra. The proper nontrivial thick subcategories of $p$-local finite spectra are precisely the full subcategories on spectra of type $\geq n+1$, for some $0 \leq n < \infty$ (see \cite{NilpotenceII} Theorem 7). The full subcategory on spectra of type $\geq n + 1$ is also the kernel of $K(n)_*$ on $p$-local finite spectra, for $0 \leq n \leq \infty$.

Since $E_n$ has the same Bousfield class as $K(0) \vee K(1) \vee \cdots \vee K(n)$, a spectrum is $E_n$-acyclic if and only if it is $K(i)$-acyclic for every $0 \leq i \leq n$. That is, chromatic localisation $L_n = L_{E_n}$ can be computed via 
\[L_{E_n} \simeq L_{K(0) \vee K(1) \vee \cdots \vee K(n)} =: L_{K(\leq n)}.\] 
As with any localisation, $L_n$ has a finitisation $L^f_n$ (see Definition \ref{def: finitise}) which is defined by taking the acyclics to be the finite acyclics of $L_n$. Since any $K(n)$-acyclic $p$-local finite spectrum is automatically $K(n-1)$-acyclic, a $p$-local finite spectrum is $E_n$-acyclic if and only if it is $K(n)$-acyclic, and so 
\[L^f_n \simeq L^f_{K(\leq n)} \simeq L^f_{K(n)}\] 
because these localisations have identical sets of finite acyclics. Thus the finite acyclics of $L^f_n$ are those finite spectra $X$ for which $X \otimes K(n) \simeq 0$, which are precisely the type $\geq n + 1$ finite spectra. Note that $L_n$ itself is not finite in general: it follows from the disproof of the telescope conjecture \cite{telescope} that $L_n \not\simeq L^f_n$ for $n \geq 2$.

Every smashing localisation of $\Sp_{(p)}$ sits between some $L_n$ and its finitisation: for a smashing localisation $L$ which is neither $0$ nor the identity, there exists $0 \leq n < \infty$ such that $\ker L^f_n \subseteq \ker L \subseteq \ker L_n$, or equivalently $L_n \leq L \leq L^f_n$ (see \cite{barthel} Proposition 4.4). There is a general philosophy proceeding from this fact that to prove some result for all smashing localisations of $\Sp_{(p)}$, one need only consider the case of the smashing localisations $L_n$. As an example, here are two equivalent statements of the telescope conjecture. The first is that $L_n \simeq L^f_n$ for every $n$. The second, suggested by Bousfield (\cite{bousfield} Conjecture 3.4), is that $L \simeq L^f$ for every smashing localisation $L$. They are equivalent because if $L_n \leq L \leq L^f_n$ and $L_n \simeq L_n^f$, then $L \simeq L_n^f \simeq L^f$ must be finite.

In its most classical form (\cite{ravenel84} Conjecture 10.5), the telescope conjecture asserts that the telescope $T(n)$ of a $v_n$-self map on a $p$-local finite spectrum of type $n$ has the same Bousfield class as $K(n)$, i.e. $L_{T(n)} \simeq L_{K(n)}$, for all $n$. For a proof that this classical form is equivalent to the statement that $L_n \simeq L^f_n$ for every $n$, see \cite{barthel} Section 3. See also \cite{millerfinite} Proposition 14 for a characterisation of $L_n^f$ in terms of telescopes. The basic idea of the proof is to identify $L_n \simeq L_{K(0) \vee K(1) \vee \cdots \vee K(n)}$ and $L_n^f \simeq L_{T(0) \vee T(1) \vee \cdots \vee T(n)}$.

\subsection{Properties of compact central maps}\label{subsec: props of central maps}
In this Subsection, we find some general properties which a compact central map must satisfy. In Theorem \ref{thm: computing finite smashing loc from central map}, we show how to compute the compactly central smashing localisation presented by a compact central map. We work entirely in $\Sp_{(p)}$, with monoidal unit $\S_{(p)}$.

Let $\beta: \S_{(p)} \to A$ be a compact central map of spectra. By compactness of $\S_{(p)}$ in $p$-local spectra, $A$ is thus compact. What can we say about the $K(n)$-homology of $\beta$? 

\begin{lem}\label{lem: central map 0 or iso on Kn}
    For a compact central map $\beta: \S_{(p)} \to A$, $\cof \beta$ is a type $m$ finite spectrum, for some $0 \leq m \leq \infty$. If $m < \infty$, then $K(n)_*\beta$ is an isomorphism for $0 \leq n < m$ and $K(n)_* \beta = 0$ for $n \geq m$.
\end{lem}

Equivalences are central, so we must consider the degenerate case where $\beta$ is an equivalence, $m = \infty$, and $\cof \beta \simeq 0$. All of the results in this Subsection are immediate when $\beta$ is an equivalence. The vanishing range of $K(n)_*\cof \beta$ looks slightly different in the $m = \infty$ and $m < \infty$ cases. If $\cof \beta \simeq 0$ then $K(n)_* \cof \beta = 0$ for every $0 \leq n \leq \infty$, including $n = m = \infty$. However, if $\cof \beta$ has type $m < \infty$ then $K(m)_* \cof \beta \neq 0$.

\begin{proof}
    Since $\cof \beta$ is finite, it must have some type $m$ with $0 \leq m \leq \infty$. If $m = \infty$ then $\cof \beta \simeq 0$ and $\beta$ is an equivalence. Otherwise, if $\cof \beta$ has type $m < \infty$ then $K(n)_* \cof \beta = 0$ for $0 \leq n < m$ and $K(n)_* \cof \beta \neq 0$ for $n \geq m$. So $K(n)_* \beta$ is an isomorphism for $0 \leq n < m$ and $K(n)_* \beta$ is not an isomorphism for $m \leq n \leq \infty$. We must show that in the range $m \leq n \leq \infty$, the condition that $\beta$ is central forces $K(n)_* \beta = 0$. 

    Since $\beta$ is central and $K(n)$ is a field, the induced map $K(n)_*\beta$ is a central map of vector spaces. That is, the two maps 
    \[\id \otimes K(n)_*\beta, K(n)_* \beta \otimes \id: K(n)_*(A) \to K(n)_* (A) \otimes K(n)_* (A)\]
    are \emph{equal} as maps of $\pi_* K(n)$-vector spaces. But a map of vector spaces $\one \to V$ cannot be central unless the map is zero or $\dim V \leq 1$, so we conclude that either $K(n)_*A$ is a 1-dimensional $\pi_* K(n)$-vector space or $K(n)_*\beta = 0$. A map between 1-dimensional vector spaces is 0 or an isomorphism, so centrality forces $K(n)_*\beta$ to be either zero or an isomorphism for each $n$.

    We already know that $K(n)_*\beta$ is not an isomorphism for $m \leq n \leq \infty$, so we conclude that $K(n)_* \beta = 0$ for each such $n$.
\end{proof}

\begin{remark}
    For a central map $\beta: \S_{(p)} \to X$ with $X$ any spectrum, not necessarily finite, the above argument shows that each $K(n)_* \beta$ is either zero or an isomorphism. However, in the case where $X$ is not finite, $\cof \beta$ is also not finite. Thus the argument to show that $K(n)_* \beta$ is an isomorphism for some range $0 \leq n < m$ and then zero afterwards no longer goes through. 

    For any smashing localisation $L$ of $\Sp_{(p)}$, the $K(n)$-homology part of the conclusion of Lemma \ref{lem: central map 0 or iso on Kn} still holds for the unit map $\S_{(p)} \to L\S_{(p)}$. This follows from the fact that $L$ sits between some $L_n$ and its finitisation, as explained at the start of this Section.
\end{remark}

\begin{defn}\label{def: algebraically central}
    Call a compact map $f: \S_{(p)} \to A$ \emph{algebraically central} if either $f$ is an equivalence, or $f$ satisfies the conclusion of Lemma \ref{lem: central map 0 or iso on Kn} for some integer $m < \infty$. That is, $K(n)_* f$ is an isomorphism for $n < m$ and $K(n)_* f = 0$ for $n \geq m$. The \emph{type} of $f$ is the integer $m$, or $\infty$ if $f$ is an equivalence.
\end{defn}

The following Lemma justifies the use of the term \emph{type} in the Definition.

\begin{lem}\label{lem: type algebraically central}
    Let $\alpha$ be a compact type $m$ algebraically central map. Then $\cof \alpha$ is a type $m$ finite spectrum.
\end{lem}

\begin{proof}
    If $m = \infty$ then $\alpha$ is an equivalence and $\cof \alpha \simeq 0$. Now assume $m < \infty$. $K(n)_*\cof \alpha = 0$ for $n < m$, since this is the range where $K(n)_*\alpha$ is an isomorphism. $K(n)_*\cof \alpha$ is nonvanishing for $n \geq m$, since the cofibre of the zero map is the direct sum of the codomain and the suspension of the domain. Thus $\cof \alpha$ is type $m$. Finiteness of $\cof \alpha$ is by assumption since $\alpha$ is compact.
\end{proof}

 Theorem \ref{thm: alpha is central} gives a partial converse to Lemma \ref{lem: central map 0 or iso on Kn}, allowing us to produce a central map from an algebraically central map. The content of Theorem \ref{thm: alpha is central} is that while an algebraically central map need not be central, a sufficiently large tensor power of an algebraically central map is central.

 Next we compute the localisation $L_\beta$ presented by a compact central map $\beta$, to obtain a supply of compactly central localisations. Since $L_\beta$ is finite (Lemma \ref{lem: compactly central implies finite}), the thick subcategory theorem gives $L_\beta \simeq L^f_n$ for some $-1 \leq n \leq \infty$. As explained at the start of this Section, the finite acyclics of $L^f_n$ are precisely the type $\geq n + 1$ finite spectra. To identify $n$ it suffices, for example, to find a single type $n + 1$ finite acyclic and a single type $n$ finite spectrum which is not acyclic. We have the following results. 

\begin{lem}\label{lem: acyclic example}
    Let $\beta$ be compact central and $\cof \beta$ have type $m+1$. Since $\cof \beta$ is acyclic for $L_\beta$, all $p$-local finite spectra of type $m+1$ are acyclic, and $L_{\beta}$ agrees with $L_{n}^f$ for some $n \leq m$.
\end{lem}

\begin{proof}
    The central map $\beta$ is a local equivalence for $L_{\beta}$ essentially by construction of the localisation. Indeed, this fact follows from Lemma \ref{lem: Jalpha' localises to Jalpha}, since all of the maps in the diagram $\mathfrak{J}_\beta$ of Construction \ref{constr: explicit Jalpha} are local equivalences. Hence $\cof \beta$ is acyclic. This gives a bound on the possible types of finite acyclics for $L_\beta$. Specifically, since type $m+1$ finite spectra are acyclic, all finite spectra of type $\geq m+1$ must be acyclic, so $L_\beta \simeq L_n^f$ for some $-1 \leq n \leq m$. Here $L_{-1}^f$ is by convention taken to be the $0$ localisation, whose kernel contains all type $\geq 0$ $p$-local finite spectra, i.e. all $p$-local finite spectra.
\end{proof}

Recall the following notation from Subsection \ref{subsec: central maps}. For a map $\beta: \one \to A$ in $\CC$, we use $\overline{\beta}$ to denote the unit map of $J_\beta = \Free_{\E_0}^{\E_1}(A, \beta)$. The next Lemma computes the Morava $K$-theory of $J_\beta$ as a colimit, which is essentially \cite{NilpotenceII} Lemma 2.4.

\begin{lem}\label{lem: cof Jbeta correct type}
    Let $\beta: \S_{(p)} \to A$ be compact central. If $\cof \beta$ has type $m < \infty$ then $\cof \overline{\beta}$ satisfies: $K(n)_* \cof \overline{\beta} = 0$ for $n < m$ and $K(m)_* \cof \overline{\beta} \neq 0$.
\end{lem}

We might say that $\cof \overline{\beta}$ has the same type as $\cof \beta$, although one must keep in mind that in general $\cof \overline{\beta}$ is not finite because $J_\beta$ is not finite. Indeed, if $J_\beta$ is finite then $J_\beta \simeq \S_{(p)}$ or $J_\beta \simeq 0$, and the localisation $\beta$ presents is one of the two trivial ones. For a spectrum $X$ which is not necessarily finite, we define its type to be the minimal integer $n$ such that $K(n)_* X \neq 0$. Note that zero is no longer the only spectrum of type $\infty$, as any tensor-nilpotent spectrum has this type. Moreover, it is not necessarily the case that a spectrum $X$ of type $n$ has $K(m)_* X \neq 0$ for all $m > n$, although of course this is true when $X$ is finite.

\begin{proof}
    We use the description of $\cof \overline{\beta}$ in terms of $\cof \beta$ given in Lemma \ref{lem: baralpha in terms of alpha}. 

    Since $\cof \beta$ is finite of type $m$, we know $K(n)_* \cof \beta = 0$ for $0 \leq n < m$. Then by the K\"{u}nneth formula for $K(n)$-homology, $K(n)_* (A \otimes \cof \beta) = 0$ for $0 \leq n < m$ also. By Lemma \ref{lem: baralpha in terms of alpha} we know $\cof(\beta^{\otimes 2})$ is an extension of $\cof \beta$ and $A \otimes \cof \beta$, so $K(n)_* \cof(\beta^{\otimes 2}) = 0$ for $0 \leq n < m$ also. Proceeding in this way, we find inductively that 
    \[K(n)_* (A \otimes \cof (\beta^{\otimes j})) = 0 = K(n)_* \cof (\beta^{\otimes j}),\]  
    for $0 \leq n < m$ and every $j > 0$. Moreover, since $A$ and $\cof \beta$ are both finite, $A \otimes \cof \beta$ is finite. Our inductive argument then shows that each $\cof(\beta^{\otimes j})$ and each $A \otimes \cof(\beta^{\otimes j})$ is finite, and they all have type at least $m$. 

    Since $\cof \overline{\beta}$ is thus the colimit of $p$-local finite type $\geq m$  spectra, we get $K(n)_* \cof \overline{\beta} = 0$ for $0 \leq n < m$. Put another way, $K(n)_* \cof \overline{\beta}$ can be computed as the colimit of a diagram where all the terms vanish, so it must vanish. 

    Even if we knew that each $\cof(\beta^{\otimes j})$ was type exactly $m$, i.e. $K(m)_* \cof(\beta^{\otimes j}) \neq 0$ for every $j$, we would not be able to conclude directly that $K(m)_* \cof \overline{\beta}$ was nonzero -- indeed, the colimit of a diagram of nonzero terms can vanish. Instead we take a slightly different approach. 

    Recall Construction \ref{constr: explicit Jalpha}, and consider the diagram $K(m)_* \mathfrak{J}_\beta$ whose colimit computes $K(m)_* J_\beta$. Since $\beta$ is central it is in particular algebraically central (Lemma \ref{lem: central map 0 or iso on Kn}), and so $K(m)_* \beta = 0$. Hence in the diagram $K(m)_* \mathfrak{J}_\beta$, all of the maps are zero. The colimit of this diagram therefore vanishes, so $K(m)_* J_\beta = 0$. Then since $K(m)_* \S_{(p)} \neq 0$ for all $m$, we conclude that 
    \[K(m)_* \cof(\overline{\beta}: \S_{(p)} \to J_{\beta}) \cong K(m)_* \Sigma \S_{(p)} \neq 0.\] 
\end{proof}

Indeed, it follows from the proof of Lemma \ref{lem: cof Jbeta correct type} that $K(n)_* \cof \overline{\beta} \cong K(n)_* \Sigma \S_{(p)}$ for all $n \geq m$, and $K(n)_* \cof \overline{\beta} = 0$ for $0 \leq n < m$. This makes sense since $J_\beta$ is an $\E_\infty$-ring, so we know by Hahn's theorem that there exists some $m$ such that $K(n)_*J_\beta = 0$ if and only if $m \leq n < \infty$ (\cite{hahnEinfty} Theorem 1.1).

\begin{thm}\label{thm: computing finite smashing loc from central map}
    Let $\beta$ be a compact central map presenting the smashing localisation $L_\beta$ (see Notation \ref{ntn: central map presents loc}). If $\cof \beta$ has type $\infty$, then $L_\beta \simeq \id_{\Sp_{(p)}}$ is the trivial localisation where every object is already local. If $\cof \beta$ has type $0$, then $L_\beta \simeq 0$ where 0 is the only local object. If $\cof \beta$ has type $m$ for $0 < m < \infty$, then $L_\beta \simeq L_{m-1}^f$ is the finitisation of $L_{m-1}$.
\end{thm}

\begin{proof}
    If $\cof \beta$ has type $\infty$ then $\cof \beta \simeq 0$, so $\beta$ is an equivalence, and $A \simeq \S_{(p)}$. An autoequivalence of $\S_{(p)}$ is idempotent, so the smashing localisation it presents is computed by tensoring with $\S_{(p)}$ along this autoequivalence and is therefore an autoequivalence of $\Sp_{(p)}$, with every object already local. 
    
    If $\cof \beta$ has type 0 then by Lemma \ref{lem: central map 0 or iso on Kn} we know $K(n)_* \beta = 0$ for every $n$. The collection $\{K(n)\}_{n=0}^\infty$ jointly detects nilpotence by \cite{NilpotenceII} Theorem 3, so $\beta$ is tensor-nilpotent. But $\beta$ is an $L_\beta$-local equivalence by construction, so any power of $\beta$ is also a local equivalence, and hence the map $\beta^{\otimes N} \simeq 0: \S_{(p)} \to A^{\otimes N}$ is a local equivalence for $N \gg 0$. Then $L_\beta \S_{(p)} \simeq 0$ and so $L_\beta \simeq 0$, because it localises the unit to zero. This is analogous to the classical fact that a localisation of a discrete ring yields the zero ring if and only if at least one nilpotent element is inverted. 

    We know from the thick subcategory theorem that $L_\beta \simeq L_n^f$ for some $n$ or $L_\beta \simeq 0$, and Lemma \ref{lem: acyclic example} tells us $n \leq m-1$ when $\cof \beta$ has type $m > 0$. Alternatively it follows from Lemma \ref{lem: cof Jbeta correct type} that $n \leq m-1$, since $\cof \overline{\beta}$ is acyclic of type $m$. But in fact, since $\beta$ presents $L_\beta$, knowing $\cof(\overline{\beta}: \S_{(p)} \to J_\beta)$ is type $m$ completely determines that $n = m-1$. 
    
    Since the localisation $L_\beta$ is compactly central, it is both finite -- by Lemma \ref{lem: compactly central implies finite} -- and smashing. In $\Sp_{(p)}$, this means any acyclic is generated under colimits and desuspensions by the single acyclic object $\cof(\overline{\beta}: \S_{(p)} \to J_\beta)$, see Lemma \ref{lem: acyclics gen by}. Indeed, every acyclic has the form $X \otimes \cof \overline{\beta}$ for some object $X$. But $\cof \overline{\beta}$ has type $m$ so $X \otimes \cof \overline{\beta}$ remains type $\geq m$, in the sense that $K(n)_* (X \otimes \cof \overline{\beta})$ vanishes for all $0 \leq n < m$. Thus all acyclics are of type $\geq m$. This is a restatement of the fact that $p$-local type $\geq m$ spectra form a tensor ideal. Hence the minimal type of any finite acyclic is $m$, witnessed by $\cof \beta$ as per Lemma \ref{lem: acyclic example}, and the collection of finite acyclics is precisely the type $\geq m$ finite spectra. Then $L_\beta \simeq L_{m-1}^f$.
\end{proof}

We have seen so far that any compact central map must be algebraically central. We have also seen how to compute the smashing localisation presented by a compact central map. It depends only on the type of the finite spectrum that is the cofibre of the central map.

\subsection{Central maps from algebraically central maps}\label{subsec: alg central to central}

In this Subsection, we approach the problem of classifying compactly central localisations from the other direction. We would like to construct a reliable supply of compact central maps. We first consider the following question: how far is an algebraically central map from being central? We prove in Theorem \ref{thm: alpha is central} that a sufficiently large tensor power of an algebraically central map must be central. This is a partial converse to Lemma \ref{lem: central map 0 or iso on Kn} and allows us to more easily produce central maps. In Theorem \ref{thm: main result lnf compactly central} we exploit this to show that $L_n^f$ is compactly central, and we then obtain our first classification result, Theorem \ref{thm: cf class in p-local spectra}, that all finite localisations of $p$-local spectra are actually compactly central. We again work in $\Sp_{(p)}$. 

The arguments given here were inspired by Hopkins-Smith \cite{NilpotenceII}, and use similar manipulations. In particular, our Theorem \ref{thm: alpha is central} is an analogue of their Lemma 3.5, and our proof of Proposition \ref{prop: alg central is central main input} uses the trick appearing as their Corollary 2.5 (i). The main difference is that our $A$ need not be a ring spectrum, so in place of their Lemma 3.4 applied to left and right multiplication we use a uniformisation trick (Construction \ref{constr: averaging idempotents}). The goal of Construction \ref{constr: averaging idempotents} and Lemma \ref{lem: Kn behaviour of e} is to build a tensor-nilpotent map by tensoring together maps that each vanish on $K(n)_*$ for some range of values of $n$, mirroring the Hopkins-Smith proof of the thick subcategory theorem. 

\begin{constr}\label{constr: averaging idempotents}
    In this Subsection, let $\alpha: \S_{(p)} \to A$ denote a compact algebraically central map. Fix an integer $r \geq 2$. Let $X = A^{\otimes r}$, and let $\sigma: X \to X$ be the cyclic permutation of the tensor factors of $X$ moving the first factor to the end, so that $\sigma^{\circ r} = \id_X$. Endomorphisms of $X$ in $\Sp_{(p)}$ form an $\E_1$-algebra, and $[X, X] = \pi_0 \End_{\Sp_{(p)}}(X)$ is an ordinary ring. Define
    \[e' = \sum_{i=0}^{r-1} \sigma^{\circ i} \text{ and } e = r\id_X - e' \in [X, X].\]
\end{constr}
In the above Construction, it would make sense to choose $r$ not divisible by $p$, and then define $e', e$ with an extra factor of $1/r$. This would simplify the statement of Proposition \ref{prop: alg central is central main input} and the proof of Corollary \ref{cor: almost central arb permutation}. The reason we allow $r$ to be arbitrary is that later we will generalise our arguments to any stable rigid $\CC$ possessing a nilpotence theorem in an appropriate sense. In such a setting, there is no guarantee that the ring $[X, X]$ has any invertible integer $r$. We give all arguments in this Subsection in a form that immediately generalises to such $\CC$ after replacing $K(n)_*$ with a set of nilpotence-detectors in $\CC$.

We will be working with $e$, and $e'$ is introduced only because it is a convenient stepping stone for establishing properties of $e$. For simplicity of notation, anytime we write $\alpha \otimes e$ or similar we mean lift $e \in [X, X]$ to some map $f: X \to X$ in $\Sp_{(p)}$ whose homotopy class is $e$, tensor with $\alpha$, and then pass to homotopy classes. In particular any expression containing $e$ is to be thought of as a homotopy class unless it is explicitly stated otherwise. 

\begin{lem}\label{lem: e composition facts}
    In $[X, X]$ we have: $e' \circ e' = re'$, $e \circ e = re$, $e'\circ \sigma = e'$, and $e \circ (\sigma - \id_X) = r(\sigma - \id_X)$.
\end{lem}

\begin{proof}
    Because $\sigma^{\circ r} = \id_X$, precomposition with $\sigma$ permutes the terms of the sum defining $e'$. Hence $e'\circ \sigma = e'$. It follows that $e' \circ (\sigma - \id_X) = 0$, and hence $e \circ (\sigma - \id_X) = r(\sigma - \id_X)$. It also follows that $e' \circ e' = r e'$, and hence $e \circ e = (r\id_X - e') \circ (r\id_X - e') = r^2\id_X - 2re' + re' = re$.
\end{proof}

\begin{lem}\label{lem: Kn behaviour of e}
    If $K(n)_* \alpha$ is an isomorphism then $K(n)_* e = 0$. Hence $K(n)_* (\alpha \otimes e) = 0$ for all $0 \leq n \leq \infty$.
\end{lem}

\begin{proof}
    Fix $n$ for which $K(n)_* \alpha$ is an isomorphism. The K\"unneth formula for $K(n)_*$ tells us that $K(n)_*(\alpha^{\otimes r})$ is also an isomorphism. Reordering a tensor product of maps is the same as permuting the tensor factors in both source and target, and $\alpha$ has source $\S_{(p)}$, so $\sigma \circ \alpha^{\otimes r} \simeq \alpha^{\otimes r}$. Hence $K(n)_* \sigma$ is the identity on $K(n)_* X$. Therefore $K(n)_* e = r\id_{K(n)_* X} - r \id_{K(n)_* X}  = 0$.

    Since $\alpha$ is algebraically central, for each $0 \leq n \leq \infty$ we have either $K(n)_* \alpha = 0$, or $K(n)_* \alpha$ is an isomorphism and therefore by the first paragraph $K(n)_* e = 0$. In both cases $K(n)_* (\alpha \otimes e) = 0$.
\end{proof}

\begin{prop}\label{prop: alg central is central main input}
    If $\alpha$ is compact and algebraically central then $r^\ell \alpha^{\otimes \ell} \otimes (\sigma - \id_X) \simeq 0$ for $\ell \gg 0$. 
\end{prop}

Note the map in the Proposition statement takes the form $A^{\otimes r} \to A^{\otimes (\ell + r)}$.

\begin{proof}
    The spectra $\{K(n)\}_{0 \leq n \leq \infty}$ collectively detect nilpotence (\cite{NilpotenceII} Theorem 3), so Lemma \ref{lem: Kn behaviour of e} tells us that $(\alpha \otimes e)^{\otimes \ell} = 0$ for $\ell \gg 0$. Permuting the order of terms in a tensor product of maps does not change whether or not the map is nullhomotopic, so $\alpha^{\otimes \ell} \otimes e^{\otimes \ell} = 0$ also. 

    The next part of this argument exploits the $\E_1$-algebra structure on $R := \hom(X, X)$, and is essentially \cite{NilpotenceII} Corollary 2.5 (i). The $\ell$-fold multiplication $\mu_\ell: R^{\otimes \ell} \to R$ induces $\ell$-fold composition on $[X, X] \cong [\S_{(p)}, R]$. Let $\widetilde{e}: \S_{(p)} \to R$ be the image of $e$ in $[\S_{(p)}, R]$, and write $\widetilde{e}^{\circ \ell}$ for the image of $e^{\circ \ell}$ in $[\S_{(p)}, R]$. Since $X$ is dualisable (Definition \ref{def: dualisable and rigid}), for any $Y$ and $\ell$ we have a bijection $[\S_{(p)}, Y \otimes R^{\otimes \ell}] \cong [X^{\otimes \ell}, Y \otimes X^{\otimes \ell}]$. Under this bijection, $\beta \otimes \widetilde{e}^{\otimes \ell}$ corresponds to $\beta \otimes e^{\otimes \ell}$ for any $\beta: \S_{(p)} \to Y$. Tracing these identifications, we have
   \[\begin{tikzcd}[column sep=small, row sep=small]
        {[X^{\otimes \ell}, A^{\otimes \ell} \otimes X^{\otimes \ell}]} \arrow[r, phantom, "\textstyle\cong"] & {[\S_{(p)}, A^{\otimes \ell} \otimes R^{\otimes \ell}]} \arrow[r, "\left(\id_A^{\otimes \ell} \otimes \mu_\ell\right) \circ -"] &[4em] {[\S_{(p)}, A^{\otimes \ell} \otimes R]} \arrow[r, phantom, "\textstyle\cong"] & {[X, A^{\otimes \ell} \otimes X]} \\
        \alpha^{\otimes \ell} \otimes e^{\otimes \ell} \arrow[r, leftrightarrow] & \alpha^{\otimes \ell} \otimes \widetilde{e}^{\otimes \ell} \arrow[r, mapsto] & \alpha^{\otimes \ell} \otimes \widetilde{e}^{\circ \ell} \arrow[r, leftrightarrow] & \alpha^{\otimes \ell} \otimes e^{\circ \ell}.
    \end{tikzcd}\]
    For $\ell \gg 0$ we have $\alpha^{\otimes \ell} \otimes e^{\otimes \ell} = 0$ and thus $\alpha^{\otimes \ell} \otimes e^{\circ \ell} = 0$. By Lemma \ref{lem: e composition facts} we know $e^{\circ \ell} = r^{\ell-1} e$, so $r^{\ell-1} \alpha^{\otimes \ell} \otimes e = 0$. The same Lemma also gives $e \circ (\sigma - \id_X) = r(\sigma - \id_X)$, so
    \[0 = r^{\ell-1} (\alpha^{\otimes \ell} \otimes e) \circ (\sigma - \id_X) = r^\ell \alpha^{\otimes \ell} \otimes (\sigma - \id_X) \in [X, A^{\otimes \ell} \otimes X].\]
\end{proof}

\begin{cor}\label{cor: almost central arb permutation}
    Let $\ell_2$, $\ell_3$ be large enough that Proposition \ref{prop: alg central is central main input} holds for $r = 2, 3$ respectively. For any $N \geq \max(\ell_2, \ell_3)+2$, $M \geq 1$, and permutation $\rho$ of $N+M$ tensor factors, we have $\rho \circ (\alpha^{\otimes N} \otimes \id_A^{\otimes M}) \circ \rho^{-1} \simeq \alpha^{\otimes N} \otimes \id_A^{\otimes M}$.
\end{cor}

\begin{proof}
    Note that $\rho \circ (\alpha^{\otimes N} \otimes \id_A^{\otimes M}) \circ \rho^{-1}$ is the tensor product of $N$ copies of $\alpha$ and $M$ copies of $\id_A$ in the order given by $\rho$, i.e. reordering a tensor product of maps is the same as permuting the tensor factors in both source and target.

    For each of $r = 2, 3$, we have the following argument. Precompose the homotopy $r^{\ell_r} \alpha^{\otimes \ell_r} \otimes \id_X \simeq r^{\ell_r} \alpha^{\otimes \ell_r} \otimes \sigma$ of Proposition \ref{prop: alg central is central main input} with the map $\alpha \otimes \id_A \otimes \alpha^{\otimes (r-2)}$ to produce a homotopy
    \begin{align*}
    r^{\ell_r} \alpha^{\otimes (\ell_r+1)} \otimes \id_A \otimes \alpha^{\otimes (r-2)} &\simeq r^{\ell_r} \alpha^{\otimes \ell_r} \otimes (\sigma \circ (\alpha \otimes \id_A \otimes \alpha^{\otimes (r-2)}))\\
    &\simeq r^{\ell_r} \alpha^{\otimes \ell_r} \otimes \id_A \otimes \alpha^{\otimes (r-1)}. \tag{$\sharp$} \label{eq: symmetry centrality permute}
    \end{align*}
    The second identification is by definition of $\sigma$ as the cyclic permutation of $A^{\otimes r}$ sending the first factor to the end.

    Let $\e = \alpha \otimes \id_A - \id_A \otimes \alpha$. It follows that $r^{\ell_r} \alpha^{\otimes \ell_r} \otimes \e \otimes \alpha^{\otimes (r-2)} \simeq 0$ because this is the difference between the first and last maps of \eqref{eq: symmetry centrality permute}. Permuting the order of terms in a tensor product of maps does not change whether or not the map is nullhomotopic, so $r^{\ell_r}\alpha^{\otimes (\ell_r+r-2)} \otimes \e \simeq 0$. Because $N -1 \geq \max(\ell_2, \ell_3)+1 \geq \ell_r+r-2$ for both $r = 2$ and $r = 3$, $\alpha^{\otimes (N-1)} \otimes \e$ is both $2$- and $3$- power torsion. Hence it is nullhomotopic. 

    Now, any permutation of $\psi := \alpha^{\otimes N} \otimes \id_A^{\otimes M}$ can be reordered to produce $\psi$ by repeatedly applying transpositions that replace $\id_A \otimes \alpha$ by $\alpha \otimes \id_A$. At each step, the difference between the maps before and after applying the transposition is $-u \otimes \e \otimes u'$, where $u$ and $u'$ together contain $N-1$ copies of $\alpha$ and $M-1$ copies of $\id_A$. Permuting the order of the tensor product in $u \otimes \e \otimes u'$ produces $\alpha^{\otimes (N-1)} \otimes \e \otimes \id_A^{\otimes (M-1)}$, which is nullhomotopic. It follows that $u \otimes \e \otimes u' \simeq 0$, so permuting the terms of $\psi$ does not alter its homotopy class.
\end{proof}

\begin{thm}\label{thm: alpha is central}
    If $\alpha$ is compact and algebraically central then $\alpha^{\otimes N} \otimes \id_{A}^{\otimes N} \simeq \id_A^{\otimes N} \otimes \alpha^{\otimes N}$ for $N \gg 0$.
\end{thm}

\begin{proof}
    Apply Corollary \ref{cor: almost central arb permutation} with $M = N$ and $\rho$ being the permutation swapping the first $N$ tensor factors with the last $N$.
\end{proof}

\begin{cor}\label{cor: some central maps exist}
    For each $0 \leq m \leq \infty$, there is a compact central map $\beta_m: \S_{(p)} \to A$ with $\cof \beta_m$ of type $m$.
\end{cor}

\begin{proof}
    The two extremal cases $\beta_0: \S_{(p)} \to 0$ and $\beta_\infty = \id_{\S_{(p)}}$ are evidently central. For each $0 < m < \infty$, Proposition \ref{prop: algebraically central existence} gives us an algebraically central map $\alpha_m$ of type $m$, and by Theorem \ref{thm: alpha is central} some large power $\alpha_m^{\otimes N}$ is central. A large tensor power of a finite spectrum remains a finite spectrum, so we need only check that $\alpha_m^{\otimes N}$ has cofibre of type $m$. It follows from the K\"unneth formula for $K(n)_*$ that $\cof(\alpha_m^{\otimes N})$ has the same type as $\cof \alpha_m$.
\end{proof}

\begin{thm}\label{thm: main result lnf compactly central}
    $L_n^f$-localisation is compactly central.
\end{thm}

\begin{proof}
    Corollary \ref{cor: some central maps exist} gives us a supply of compact central maps. Theorem \ref{thm: computing finite smashing loc from central map} computes that the smashing localisation presented by a compact central map with cofibre of type $m$ is $L_{m-1}^f$.
\end{proof}

In fact, we have proven the following stronger result.

\begin{thm}\label{thm: cf class in p-local spectra}
    A localisation of $\Sp_{(p)}$ is finite if and only if compactly central.
\end{thm}

\begin{proof}
    We already know from Lemma \ref{lem: compactly central implies finite} that all compactly central localisations are finite. For the converse, note that the thick subcategory theorem classifies all finite localisations on $p$-local spectra. These are the localisations $L_n^f$ for $0 \leq n < \infty$ and the two trivial localisations, namely the localisation for which all objects are acyclic, and the one for which all objects are local. In Theorem \ref{thm: main result lnf compactly central} we saw that each $L_n^f$ is compactly central, and the compact central maps $\S_{(p)} \to 0$ and $\id: \S_{(p)} \to \S_{(p)}$ present the two trivial localisations so these are also compactly central.
\end{proof}

As mentioned in the Introduction, Theorem \ref{thm: main result lnf compactly central} refines Miller's construction of finite localisations (\cite{millerfinite} Theorem 4). At this point we can fulfil an earlier promise to give a classification of compact central maps.

\begin{thm}\label{thm: classification of compact central maps}
    Let $\beta: \S_{(p)} \to A$ be a compact map in $p$-local spectra. Then $\beta^{\otimes N}$ is central for $N \gg 0$ if and only if $\beta$ is algebraically central.
\end{thm}

\begin{proof}
    Theorem \ref{thm: alpha is central} gives the reverse implication. For the forwards direction, if $\beta^{\otimes N}$ is central for $N \gg 0$ then by Lemma \ref{lem: central map 0 or iso on Kn} $\beta^{\otimes N}$ is algebraically central. Hence $\beta$ itself must be algebraically central. This last claim reduces to the linear algebra fact that if some tensor power of a linear map is zero or an isomorphism, the linear map itself must be zero or an isomorphism respectively.
\end{proof}

\subsection{Compactly central localisations of spectra}\label{subsec: passage to spectra}

Now that we understand compactly central localisations of $p$-local spectra, we turn our attention to $\Sp$ as a whole. We first need to relate finite localisations on $\Sp$ to those on $\Sp_{(p)}$. 

Any finite localisation $F$ of $\Sp$ is smashing. Since $p$-localisation is also smashing, we may compose the two to obtain a new smashing localisation $L_{(p)} \circ F$. By Lemma \ref{lem: combining smashing locs} we also get a smashing localisation of $p$-local spectra corresponding to $F$, computed by the same idempotent algebra as $F$, since $F$ necessarily maps $p$-local spectra to $p$-local spectra. Do this for each prime $p$, and we get a family of localisations $\{F|_{\Sp_{(p)}}\}_{p ~\text{prime}}$ which together determine $F$. Note that each $F|_{\Sp_{(p)}}$ is a finite localisation, again by Lemma \ref{lem: combining smashing locs}.

The next three results are a reformulation of the classification of finite localisations of $\Sp$ by Thomason subsets of $\Spc(\Sp^\omega)$ (see \cite{balmerfinclass}, and \cite{balmerspectra} Corollary 9.5). We include proofs here for completeness. Lemmas \ref{lem: fin loc sp uniqueness} and \ref{lem: fin loc sp existence} each give one direction of the argument, and Corollary \ref{cor: fin locs on sp} states the full classification.

\begin{lem}\label{lem: fin loc sp uniqueness}
    Let $F$ be a finite localisation of $\Sp$. Then $F$ can be recovered from its behaviour on $p$-local spectra. That is, the collection of finite localisations $\{F_{(p)}\}_{p~\text{prime}}$ uniquely determines $F$, where $F_{(p)} := F|_{\Sp_{(p)}}$. If $F$ is nonzero then $F_{(p)}$ is nonzero for every $p$. 
\end{lem}

\begin{proof}
    Each $F_{(p)}$ determines $F_{(p)} \S_{(p)} \simeq L_{(p)} F\S$, and $F$ is determined by $F\S$ because it is smashing. Thus the result amounts to reconstructing $F\S$ from its $p$-localisations. This can be done using the arithmetic fracture square for $F\S$. We use the version of the arithmetic fracture square with $p$-localisations in place of $p$-completions. We have the homotopy pullback diagram
    \[\begin{tikzcd}[sep={8em, between origins}]
        F\S \arrow[r] \arrow[d]       & L_\Q F\S \arrow[d]             \\
        \prod_p L_{(p)} F\S \arrow[r] & L_\Q \prod_p L_{(p)} F\S.
    \end{tikzcd}\]
    Compare Sullivan \cite{sullivan} Proposition 1.10 and Meta Proposition 1.12 as precedents for the localisation version of arithmetic fracture, and see also \cite{bousfield} Proposition 2.9 for the $p$-completion version. We now consider two cases.

    If $F_{(p)} \simeq 0$ for a single prime $p$, or if $L_\Q F \S \simeq 0$, then $0 \simeq L_\Q F\S \simeq L_\Q F_{(q)} \S_{(q)}$ for every prime $q$. By the thick subcategory theorem, it follows that $F_{(q)} \simeq 0$, because $L_\Q L_n^f \simeq L_0^f \not\simeq 0$ for every $n \geq 0$. Hence the fracture square tells us that $F \simeq 0$, because the other three terms all vanish.

    Assume now that $F \not\simeq 0$, and hence $L_\Q F\S \not\simeq 0$ and $F_{(p)} \not\simeq 0$ for every prime $p$. Since $F\S$ is an idempotent algebra, $L_\Q F\S \simeq H\Q \otimes F\S$ is a nonzero idempotent algebra over $H\Q$ and the only such is $H\Q$ itself. Indeed, $\Mod_{H\Q}(\Sp) \simeq D(\Q)$ by \cite{ha} Theorem 7.1.2.13, and the only localisations of $D(\Q)$ are $0$ and the identity by \cite{neemanchromatictower} Theorem 2.8.
    
    To recover $F\S$ as the pullback, we need two maps. The map 
    \[\prod_p L_{(p)}F\S \to L_\Q \prod_p L_{(p)}F\S\]
    is simply rationalisation, which is determined given each $F_{(p)}$. We claim the map
    \[\phi: H\Q \to L_\Q \prod_p L_{(p)} F\S\]
    is also uniquely determined. Since $F\S$ is an idempotent algebra, the entire fracture square refines to a diagram of $\E_\infty$-algebras. Then the map $\phi$ is a ring map from $H\Q$ to some other $\E_\infty$-algebra, but there is a contractible space of such maps because $H\Q$ is idempotent, see Lemma \ref{lem: En refinement}.
\end{proof}

\begin{lem}\label{lem: fin loc sp existence}
    Any family $\{F_p\}_{p~\text{prime}}$ where each $F_p$ is a nonzero finite localisation of $\Sp_{(p)}$ has a corresponding finite localisation $F$ of $\Sp$, such that $F|_{\Sp_{(p)}} \simeq F_p$ for all $p$.
\end{lem}

We know from Lemma \ref{lem: fin loc sp uniqueness} that if $F|_{\Sp_{(p)}} \simeq 0$ for any individual prime $p$ then $F \simeq 0$, so we are free to ignore this trivial case. Here is the general proof strategy: we can define a localisation by specifying a generating set for its acyclics, and a finite localisation by its compact acyclics. Given localisations $F_p$, we take a generating set of compact acyclics for each and lift them to compact spectra. Declare the union of all these acyclics to be a generating set of acyclics for $F$, and then $F$ is evidently finite. 

An immediate problem with this program is that compact objects in $p$-local spectra need not be compact when viewed as objects of spectra via the subcategory inclusion. For example, the unit $\S_{(p)}$ is compact as a $p$-local spectrum but is not a finite spectrum. So, if we tried to perform this lifting procedure with $F_p \simeq 0$ we would take the generating acyclic $\S_{(p)}$, lift it to $\S$, and then $F \simeq 0$ because the unit is acyclic -- regardless of the choice of localisation we wanted at other primes $q \neq p$. Proving the Lemma thus amounts to observing that when $F_p \not\simeq 0$, we can find a generating set of compact acyclics for $F_p$ which are already compact as spectra. Taking these as the acyclics for $F$, our choices at distinct primes are independent.

\begin{proof}
    Choosing a nonzero finite localisation $F_p$ amounts, by the thick subcategory theorem, to choosing $n_p$ with $0 \leq n_p \leq \infty$ and taking $F_p := L_{n_p}^f$. Note $L_{\infty}^f \simeq \id_{\Sp_{(p)}}$, and we are not allowing $F_p \simeq 0$.

    The acyclics for $L_{n_p}^f$ are generated by any single $p$-local finite spectrum of type $n_p + 1$, so we want to choose a single $p$-local finite spectrum $T_{p, n_p+1}$ of type $n_p + 1$ which is compact in spectra. If $n_p = \infty$ we take $T_{p, n_p+1} \simeq 0$. Otherwise, take $T_{p, n_p+1}$ to be a generalised Smith-Toda complex of the kind in Construction \ref{constr of X}, built from copies of $\S/p$, and such that $L_{(q)} T_{p, n_p+1} \simeq 0$ for every prime $q \neq p$ while $L_{(p)}T_{p, n_p+1}$ is a $p$-local finite spectrum of type $n_p + 1$. See the introduction to Subsection \ref{subsec: construction of algebraically central maps} for background on Smith-Toda complexes. We are using the vanishing property established there, and note that $L_{(p)}T_{p, n_p+1} \simeq T_{p, n_p+1}$ because $T_{p, n_p+1}$ is \emph{a priori} $p$-local. Let $F$ be the finite localisation of $\Sp$ whose acyclics are generated by the $T_{p, n_p+1}$ as $p$ ranges over all primes.

    Recall that for a set $N$ of objects in $\Sp$, $\langle N \rangle$ denotes the smallest localising subcategory of $\Sp$ containing $N$. It remains to check that $F|_{\Sp_{(p)}} \simeq F_p$ for each prime $p$. We do this by comparing their kernels. Since $F|_{\Sp_{(p)}}$ is computed by the same idempotent algebra as $F$, its acyclics are the $p$-local objects of $\AA = \ker F = \langle T_{q, n_q+1} : q \text{ prime} \rangle$. We claim that $\AA \cap \Sp_{(p)} = \langle T_{p, n_p+1} \rangle$.

    Since $T_{p, n_p+1}$ is $p$-local and $F$-acyclic, and both $\AA$ and $\Sp_{(p)}$ are closed under colimits in $\Sp$, we have $\langle T_{p, n_p+1} \rangle \subseteq \AA \cap \Sp_{(p)}$. Conversely, since $L_{(p)}$ is exact and preserves colimits, $\mathscr{B} = \{X \in \Sp \mid L_{(p)}X \in \langle T_{p, n_p+1} \rangle\}$ is a localising subcategory of $\Sp$. It contains $T_{p, n_p+1}$, and it contains $T_{q, n_q+1}$ for $q \neq p$ because $L_{(p)}T_{q, n_q+1} \simeq 0$, so $\AA \subseteq \mathscr{B}$. Hence every $X \in \AA \cap \Sp_{(p)}$ satisfies $X \simeq L_{(p)}X \in \langle T_{p, n_p+1} \rangle$.

    Thus $\ker(F|_{\Sp_{(p)}}) = \AA \cap \Sp_{(p)} = \langle T_{p, n_p+1} \rangle = \ker(L^f_{n_p})$ as full subcategories of $\Sp_{(p)}$, including in the special case when $n_p = \infty$ and $T_{p, n_p+1} \simeq 0$. A localisation is determined by its kernel, so $F|_{\Sp_{(p)}} \simeq L^f_{n_p} \simeq F_p$.
\end{proof}

\begin{cor}\label{cor: fin locs on sp}
    A nonzero finite localisation of $\Sp$ is classified by parameters 
    \[\{n_p~|~0 \leq n_p \leq \infty\}_{p ~\text{prime}}.\]
    Given any such family of parameters, there is a unique finite localisation $F$ that satisfies $F|_{\Sp_{(p)}} \simeq L_{n_p}^f$ for all primes $p$.
\end{cor}

\begin{proof}
    Lemma \ref{lem: fin loc sp existence} gives existence of a localisation corresponding to any given parameter family, and Lemma \ref{lem: fin loc sp uniqueness} tells us every finite localisation has such parameters and establishes uniqueness. 
\end{proof}

\begin{example}
    Note that $p$-localisation itself corresponds to the parameters $n_q = 0$ for $q \neq p$ and $n_p = \infty$. Rationalisation is given by choosing $n_p = 0$ for all $p$, since $L_0 \simeq L_0^f \simeq L_{H\Q}$. If we $p$-localise and then $q$-localise for $p \neq q$, we recover the parameters for rationalisation. Composing two finite localisations of $\Sp$ is achieved by taking the pointwise minimum of their parameters. 
\end{example}

Next we must understand how to lift compactly central localisations from $\Sp_{(p)}$ to $\Sp$.
\begin{lem}\label{lem: lifting cc from spp to sp}
    Each nonzero compactly central localisation $L_n^f$ of $\Sp_{(p)}$ can be lifted to a compactly central localisation $F^{n, p}$ of $\Sp$ with the property that $F^{n, p}|_{\Sp_{(p)}} \simeq L_n^f$ and $F^{n, p}|_{\Sp_{(q)}} \simeq \id_{\Sp_{(q)}}$ for all primes $q \neq p$.
\end{lem}

We already know by Lemmas \ref{lem: fin loc sp uniqueness} and \ref{lem: fin loc sp existence} that there is a unique finite localisation $F^{n, p}$ of $\Sp$ having the desired restrictions at each prime. We must establish that $F^{n, p}$ is compactly central. 

\begin{proof}
    It is clear that the identity localisation of $\Sp$ is compactly central because it is presented by $\id_\S$. Thus we may assume $0 \leq n < \infty$. Fix $n$ and $p$. By Proposition \ref{prop: algebraically central existence} and Lemma \ref{lem: lift alg central spp to sp}, we have an algebraically central map $\alpha_{n+1}$ in $\Sp_{(p)}$ with cofibre of type $n+1$, and a lift $\widetilde{\alpha}_{n+1}: \S \to \widetilde{A}_{n+1} \in \Sp$ satisfying:
    \begin{itemize}
        \item $L_{(p)}\widetilde{\alpha}_{n+1} \simeq \alpha_{n+1}$;
        \item $L_{(q)}\widetilde{\alpha}_{n+1}$ is an equivalence for primes $q \neq p$;
        \item $\cof \widetilde{\alpha}_{n+1} \simeq \cof \alpha_{n+1}$ and this cofibre is compact in both $\Sp_{(p)}$ and $\Sp$. 
    \end{itemize}

    By Theorem \ref{thm: alpha is central} we know $\alpha_{n+1}^{\otimes N}$ is central for $N \gg 0$. Then $\widetilde{\alpha}_{n+1}^{\otimes N}$ is also central because it is central after localising at any prime. Note that any equivalence $\S_{(q)} \xrightarrow{\simeq} \S_{(q)}$ is central. The cofibre of $\widetilde{\alpha}_{n+1}^{\otimes N}$ is a finite spectrum so the localisation $G$ this map presents (see Notation \ref{ntn: central map presents loc}) is compactly central. Since $p$-localisation is smashing and $L_{(p)}\widetilde{\alpha}_{n+1} \simeq \alpha_{n+1}$, it follows from Theorem \ref{thm: computing finite smashing loc from central map} that $G|_{\Sp_{(p)}} \simeq L_n^f$. The central map $L_{(q)}\widetilde{\alpha}_{n+1}$ is an equivalence so it presents the identity localisation, and hence $G|_{\Sp_{(q)}} \simeq \id_{\Sp_{(q)}}$. Therefore $G \simeq F^{n, p}$ is a compactly central localisation of $\Sp$.
\end{proof}

We are now prepared to classify compactly central localisations of $\Sp$.

    \begin{thm}\label{thm: not compactly central}
        A nonzero finite localisation $F$ of $\Sp$ is compactly central if and only if $F|_{\Sp_{(p)}} \simeq \id_{\Sp_{(p)}}$ for all but finitely many primes $p$. That is, the parameters of $F$ satisfy $n_p = \infty$ for all but finitely many primes $p$.
    \end{thm}

    \begin{proof}
        Consider a nonzero compactly central localisation of $\Sp$, presented by the compact central map $\alpha: \S \to A$. Let $P$ be the set of prime numbers $p$ such that there is $p$-torsion in $H_* (\cof \alpha)$. Since $A$ is finite and $\S$ is finite, the long exact sequence on homology tells us that $H_* (\cof \alpha)$ is finitely generated as an abelian group. Since $L_\alpha\cof \alpha \simeq 0$, $L_\alpha$ is nonzero, and $\cof \alpha$ is a finite spectrum, we know that $L_\Q\cof\alpha \simeq 0$ too. Thus $H_* (\cof \alpha)$ is rationally trivial so has no free summands, and it is a torsion group. Then $H_* (\cof \alpha)$ is finite and $P$ is just the finite set of prime numbers dividing its order. Let $M = \prod_{p \in P} p$ be the product of all these primes. Now by the universal property of localisation, we have a factorisation
        \[\begin{tikzcd}[sep={6em, between origins}]
            \S \arrow[rr] \arrow[rd] &                        & {L_{\frac1M}\S} \\
                                     & L_\alpha \S \arrow[ru] &          
            \end{tikzcd}\]
        because the generating $L_\alpha$-acyclic object $\cof \alpha$ is also acyclic for the localisation $L_{\frac1M}$. But this means $L_\alpha \geq L_{\frac1M}$ in the partial order, and so 
        \[L_\alpha|_{\Sp_{(p)}} \geq L_{\frac1M}|_{\Sp_{(p)}} \simeq \begin{cases}L_0^f & p \mid M\\ L_\infty^f & p\nmid M \end{cases}\]
        for every prime $p$. The partial order for finite localisations of $\Sp_{(p)}$ is the same as the partial order on the natural numbers, so we conclude that the parameters for $L_\alpha$ satisfy the system
        \[\begin{cases}
            n_p = \infty &p \nmid M\\
            n_p \geq 0 & p \mid M.
        \end{cases}\]
        In particular, $n_p = \infty$ for all but finitely many primes.

        Conversely, let $P$ be a finite set of primes and choose $n_p$ for each $p \in P$. We must construct a compactly central localisation $F$ such that 
        \[F|_{\Sp_{(p)}} \simeq \begin{cases}L_{n_p}^f & p \in P,\\\id_{\Sp_{(p)}} & p \notin P.\end{cases} \tag{$\dagger$} \label{eq: restr conds}\]
        Given any individual prime $p$ and choice $n_p$, there is a compactly central localisation $F^{n_p, p}$ of $\Sp$ such that $F^{n_p, p}|_{\Sp_{(p)}} \simeq L_{n_p}^f$ and $F^{n_p, p}|_{\Sp_{(q)}} \simeq \id_{\Sp_{(q)}}$ for all primes $q \neq p$. This is the content of Lemma \ref{lem: lifting cc from spp to sp}. The tensor product of a finite collection of compact central maps is again a compact central map, so take $F := \bigotimes_{p \in P} F^{n_p, p}$. Lemma \ref{lem: composite of cc locs is cc} tells us $F$ is compactly central. Now Corollary \ref{cor: tensor smashing maximal multi} implies that among all finite localisations, $F$ is maximal subject to the restriction that $F \leq F^{n_p, p}$ for all $p \in P$. We know by Corollary \ref{cor: fin locs on sp} that a finite localisation satisfying \eqref{eq: restr conds} exists, and $F$ is bounded above by the conditions of \eqref{eq: restr conds}, so maximality tells us $F$ is exactly as desired.
    \end{proof}

\subsection{Construction of algebraically central maps}\label{subsec: construction of algebraically central maps}
The purpose of this Subsection is to establish Proposition \ref{prop: algebraically central existence} and the closely associated analogue Lemma \ref{lem: lift alg central spp to sp}, by explicitly constructing a family of algebraically central maps in $\Sp_{(p)}$. 

\begin{prop}\label{prop: algebraically central existence}
    For each $0 < m < \infty$, there exists an algebraically central map 
    \[\alpha_m: \S_{(p)} \to A_m \in \Sp_{(p)}\] 
    of type $m$, where $A_m$ is a finite complex. That is,
    \[K(n)_* \alpha_m = \begin{cases} \text{isomorphism} & n < m\\ 0 & n \geq m. \end{cases}\]
\end{prop}

To prove this Proposition, we simply construct a suitable map $\alpha_m$. We use generalised Smith-Toda complexes to do so. For $0 < m < \infty$, recall that a generalised Smith-Toda complex of type $m$ is a $p$-local finite spectrum of the form 
\[X_m = \S/(p^{i_0}, v_1^{i_1}, \ldots, v_{m-1}^{i_{m-1}})\] 
for positive integers $i_0, \ldots, i_{m-1}$. Such spectra exist for suitable choices of exponents (\cite{NilpotenceII} Theorem 9 and \cite{hoveystrickland} Proposition 4.22). They are built inductively, starting from $X_0 = \S_{(p)}$ and taking $X_{k+1}$ to be the cofibre of a self map $v_k^{i_k}: \Sigma^{i_k |v_k|} X_k \to X_k$, where $|v_k| = 2(p^k - 1)$. We take $i_0 = 1$ and $v_0$ to be multiplication by $p$, so that $X_1 = \S_{(p)}/p \simeq \S/p$. Each $v_k^{i_k}$ is a $v_k$-self map, meaning that it is an isomorphism on $K(k)$-homology and nilpotent on $K(n)$-homology for all $n \neq k$ with $0 \leq n \leq \infty$. By increasing $i_k$ appropriately, we may assume that $v_k^{i_k}$ is zero on $K(n)$-homology for $n \neq k$. For a detailed account of this construction, see \cite{hoveystrickland} Section 4. See also \cite{BHV} Definition 6.2 and Theorem 6.3 for a concise statement in the language used here.

We use the following properties of $X_m$, each of which follows by induction from the cofibre sequences $\Sigma^{i_k |v_k|} X_k \to X_k \to X_{k+1}$. First, $X_m$ has type $m$. Second, for $m \leq n \leq \infty$ each of the maps $v_k^{i_k}$ with $k < m$ induces zero on $K(n)$-homology (see \cite{NilpotenceII} Theorem 9), so $K(n)_* X_m$ is a direct sum of $2^m$ suspensions of $K(n)_* \S_{(p)}$, one for each cell of $X_m$. The projection onto the top cell $X_m \to \Sigma^{N_m} X_0 \simeq \Sigma^{N_m} \S_{(p)}$, given by composing the maps $X_{k+1} \to \Sigma^{i_k |v_k| + 1} X_k$ from the cofibre sequences, is a split epimorphism on $K(n)$-homology, where $N_m = m + \sum_{k < m} i_k |v_k|$. Third, $X_1 = \S/p$ is a finite spectrum, and $X_m$ lies in the thick subcategory of $\Sp$ generated by $\S/p$. Hence $X_m$ is compact in $\Sp$, and $L_{(q)} X_m \simeq 0$ for every prime $q \neq p$. The original Smith-Toda complexes $V(n) = \S/(p, v_1, \ldots, v_n)$ with exponent $1$ were considered in \cite{smith} and \cite{toda}, and are only known to exist for $n$ small relative to $p$.

\begin{constr}\label{constr of X}
    Fix $0 < m < \infty$. Let $X_m$ denote a type $m$ generalised Smith-Toda complex. Let $Y_m = \Sigma^{-N_m} X_m$, so that projection onto the top cell gives a map $q_m: Y_m \to \S_{(p)}$. Define $A_m$ and $\alpha_m$ by the cofibre sequence 
    \[Y_m \xrightarrow{q_m} \S_{(p)} \xrightarrow{\alpha_m} A_m. \tag{$\star$}\label{eq: fib seq def X A alpha}\]
\end{constr}

\begin{proof}[Proof of Proposition \ref{prop: algebraically central existence}]
    Fix $0 < m < \infty$. We show that the map $\alpha_m$ of Construction \ref{constr of X} is algebraically central. Its codomain $A_m$ is a finite $p$-local spectrum because it is the cofibre of a map between finite $p$-local spectra. Since $Y_m$ has type $m$, $K(n)_* Y_m = 0$ for $n < m$. It follows from the cofibre sequence \eqref{eq: fib seq def X A alpha} that $K(n)_* \alpha_m$ is an isomorphism for $n < m$. For $n \geq m$, $K(n)_* q_m$ is a split epimorphism as explained before Construction \ref{constr of X}, so \eqref{eq: fib seq def X A alpha} shows that $K(n)_* \alpha_m = 0$.
\end{proof}

We also require an analogue to Proposition \ref{prop: algebraically central existence} which holds in $\Sp$.

\begin{lem}\label{lem: lift alg central spp to sp}
    For each $0 < m  < \infty$, the map $\alpha_{m} \in \Sp_{(p)}$ of Proposition \ref{prop: algebraically central existence} lifts to a map $\widetilde{\alpha}_{m}: \S \to \widetilde{A}_{m} \in \Sp$ satisfying:
    \begin{itemize}
        \item $L_{(p)}\widetilde{\alpha}_{m} \simeq \alpha_{m}$;
        \item $L_{(q)}\widetilde{\alpha}_{m}$ is an equivalence for primes $q \neq p$;
        \item $\cof \widetilde{\alpha}_{m} \simeq \cof \alpha_{m}$ and this cofibre is compact in both $\Sp_{(p)}$ and $\Sp$. 
    \end{itemize}
\end{lem}

\begin{proof}
    Run through Construction \ref{constr of X} in $\Sp$. The generalised Smith-Toda complex $X_m$ is defined the same in $\Sp$ as in $\Sp_{(p)}$, because for $m > 0$ one builds $X_m$ from copies of $\S_{(p)}/p \simeq \S/p$ as an iterated mapping cone. In particular, $X_m$ as constructed in $\Sp_{(p)}$ is already compact in $\Sp$, as is $Y_m$. Let $\widetilde{q}_m: Y_m \to \S$ be the projection onto the top cell, and define $\widetilde{\alpha}_m = \cof(\widetilde{q}_m): \S \to \widetilde{A}_m$. Since $Y_m$ and $\S$ are compact spectra, $\widetilde{A}_m$ is compact. Note that $A_m$ of Construction \ref{constr of X} is not compact in $\Sp$ because $\S_{(p)}$ is not compact in $\Sp$. 
    
    By construction $L_{(p)} \widetilde{q}_m \simeq q_m$. Hence $L_{(p)} \widetilde{A}_m \simeq A_m$ and $L_{(p)} \widetilde{\alpha}_m \simeq \alpha_m$, because $p$-localisation is smashing so commutes with (co)fibres. Moreover $\cof \widetilde{\alpha}_m \simeq \Sigma Y_m \simeq \cof \alpha_m$. 
    
    For a prime $q \neq p$, we have $L_{(q)} Y_m \simeq \Sigma^{-N_m} L_{(q)} X_m \simeq 0$ because $X_m$ is built from copies of $\S/p$ and $L_{(q)}$ inverts $p$. Thus $L_{(q)} \widetilde{\alpha}_m$ is an equivalence, and as a consequence $L_{(q)} \widetilde{A}_m \simeq \S_{(q)}$. We assume $m \geq 1$ because we do not have a sensible definition of $X_0$ in $\Sp$, since $\S_{(p)}$ is not compact in $\Sp$ and $\S$ is not killed by $p$-inversion.
\end{proof}

\section{Compactly central localisations of \texorpdfstring{$D(X)$}{D(X)}}\label{sec: the D(X) case}

Our goal in this Section is to generalise Theorems \ref{intro thm: p-local compactly central} and \ref{intro thm: sp compactly central} to $D(X)$. We succeed for any quasicompact and quasiseparated (qcqs) scheme possessing an ample family of line bundles, including a quasiprojective scheme over an affine base and the affine case itself. Many of the results in this Section work in any stable rigid category $\CC$ with a nilpotence-detecting family (Definition \ref{def: nilpotence-detecting family}). The main obstruction to full generality that we encounter is producing a robust supply of algebraically central maps. In $\Sp_{(p)}$, we built such maps explicitly from Smith-Toda complexes, which are well understood (see Subsection \ref{subsec: construction of algebraically central maps}). In $D(R)$, we use maps $R \to R$ given by multiplication by a ring element. Given an ample family of line bundles we use their sections, because any map from the unit to a line bundle is central. 

First we must understand how the Balmer spectrum allows us to interpret Theorems \ref{intro thm: p-local compactly central} and \ref{intro thm: sp compactly central} as essentially the same statement, and then we can work towards generalising that statement to other stable rigid categories $\CC$. 

\subsection{The Balmer spectrum}\label{subsec: balmer spectrum}

We define the Balmer spectrum $\Spc(\CC^\omega)$ of compact objects in a stable rigid category $\CC$, recall some of its properties, and discuss $\Spc(\Sp^\omega)$ and $\Spc(\Sp_{(p)}^\omega)$. The Balmer spectrum is of interest to us because it classifies finite tensor-compatible localisations of $\CC$, and we are trying to understand which localisations among those are compactly central. At the end of this Subsection we analyse what Theorems \ref{intro thm: p-local compactly central} and \ref{intro thm: sp compactly central} say about the Balmer spectrum, and produce a restatement of these results which makes sense to generalise. 

\begin{defn}[\cite{balmerfinclass} Definitions 2.1 and 4.1]\label{def: balmer spectrum}
    Let $\CC$ be stable and rigid. A \emph{thick tensor ideal} of $\CC^\omega$ is a thick subcategory $\mathscr{I} \subseteq \CC^\omega$ such that $E \otimes F \in \mathscr{I}$ for all $E \in \mathscr{I}$ and $F \in \CC^\omega$. A thick tensor ideal $\mathscr{P}$ is \emph{prime} if $\mathscr{P} \neq \CC^\omega$ and $E \otimes F \in \mathscr{P}$ implies that $E \in \mathscr{P}$ or $F \in \mathscr{P}$. The \emph{Balmer spectrum} $\Spc(\CC^\omega)$ is the set of prime thick tensor ideals of $\CC^\omega$. The \emph{support} of an object $E \in \CC^\omega$ is
    \[\Supp(E) = \{\mathscr{P} \in \Spc(\CC^\omega) : E \notin \mathscr{P}\},\]
    and $\Spc(\CC^\omega)$ is given the topology in which the subsets $\Supp(E)$ for $E \in \CC^\omega$ form a basis of closed subsets. A thick tensor ideal $\mathscr{I}$ is \emph{radical} if $E^{\otimes n} \in \mathscr{I}$ for some $n \geq 1$ implies $E \in \mathscr{I}$.
\end{defn}

Supports of objects of $\CC^\omega$ are well understood. The following two Lemmas record the key properties we require. 

\begin{lem}[\cite{balmerfinclass} Lemma 2.6, Definition 3.1, and Theorem 3.2]\label{lem: support properties}
    Let $\CC$ be stable and rigid, and let $E, F, G \in \CC^\omega$.
    \begin{enumerate}[label=(\arabic*)]
        \item $\Supp(0) = \emptyset$ and $\Supp(\one) = \Spc(\CC^\omega)$.
        \item $\Supp(E \oplus F) = \Supp(E) \cup \Supp(F)$.
        \item $\Supp(\Sigma E) = \Supp(E)$. \label{lem sub: supp of suspension}
        \item If $E \to F \to G$ is a cofibre sequence then $\Supp(F) \subseteq \Supp(E) \cup \Supp(G)$.
        \item $\Supp(E \otimes F) = \Supp(E) \cap \Supp(F)$. \label{lem sub: supp of tensor}
    \end{enumerate}
\end{lem}

\begin{lem}[\cite{balmerfinclass} Proposition 2.14]\label{lem: qc opens are supports}
    Let $\CC$ be stable and rigid. A subset $Y \subseteq \Spc(\CC^\omega)$ is closed with quasicompact complement if and only if $Y = \Supp(E)$ for some $E \in \CC^\omega$.
\end{lem}

We have the following observation, which we record as a Lemma because it will be used several times in this Section. 

\begin{lem}[cf \cite{balmerfinclass} Lemma 3.4 and Notation 4.7]\label{lem: support ideals}
    Let $\CC$ be stable and rigid. For any $Y \subseteq \Spc(\CC^\omega)$, the objects $E \in \CC^\omega$ with $\Supp(E) \subseteq Y$ form a radical thick tensor ideal of $\CC^\omega$. Consequently, if $\mathscr{I}$ is the thick tensor ideal generated by a set $\mathscr{E}$ of objects of $\CC^\omega$, then
    \[\bigcup_{E \in \mathscr{I}} \Supp(E) = \bigcup_{E \in \mathscr{E}} \Supp(E).\]
\end{lem}

\begin{proof}
    Let $\mathscr{I}_Y$ denote the objects $E \in \CC^\omega$ with $\Supp(E) \subseteq Y$. By Lemma \ref{lem: support properties}, $\mathscr{I}_Y$ contains $0$, is closed under suspension and desuspension, and contains the third term of any cofibre sequence whose other two terms lie in $\mathscr{I}_Y$. It is closed under retracts by additivity of supports, since a retract $E'$ of $E$ satisfies $E \simeq E' \oplus E''$ for some $E''$. It is a tensor ideal because $\Supp(E \otimes F) = \Supp(E) \cap \Supp(F) \subseteq \Supp(E)$, and it is radical because $\Supp(E^{\otimes n}) = \Supp(E)$.

    For the second statement, $\supseteq$ holds because $\mathscr{E} \subseteq \mathscr{I}$. For $\subseteq$, take $Y = \bigcup_{E \in \mathscr{E}} \Supp(E)$. Then $\mathscr{I}_Y$ is a thick tensor ideal containing $\mathscr{E}$, so it contains $\mathscr{I}$.
\end{proof}

Balmer's classification of finite tensor-compatible localisations of $\CC$ in terms of the Balmer spectrum identifies each localisation with the union of supports of its compact acyclic objects, see Corollary \ref{cor: balmer correspondence} (which is a corollary to \cite{balmerfinclass} Theorem 4.10). Such unions of supports have an identifying topological property as a result of Lemma \ref{lem: qc opens are supports}. 

\begin{defn}[cf \cite{thomason} Theorem 3.15]\label{def: thomason subset}
    Let $\CC$ be stable and rigid. A subset $Y \subseteq \Spc(\CC^\omega)$ is a \emph{Thomason subset} if it is a union $Y = \bigcup_{i \in I} Y_i$ of closed subsets $Y_i$, each of which has quasicompact complement.
\end{defn}

A slightly stronger result than Lemma \ref{lem: qc opens are supports} holds for Thomason subsets. Any Thomason subset which is closed must itself have quasicompact complement. 

\begin{lem}[cf \cite{balmerfinclass} Lemma 2.13]\label{lem: closed thomason subsets}
    Let $\CC$ be stable and rigid. All closed Thomason subsets of $\Spc(\CC^\omega)$ have the form $\Supp(E)$ for some $E \in \CC^\omega$.
\end{lem}

\begin{proof}
    By Lemma \ref{lem: qc opens are supports}, a subset of $\Spc(\CC^\omega)$ has the form $\Supp(E)$ if and only if it is closed with quasicompact complement. It remains to show that every closed Thomason subset $Y$ has quasicompact complement. By Lemma \ref{lem: qc opens are supports} write $Y = \bigcup_{i \in I} \Supp(E_i)$ for some $E_i \in \CC^\omega$, and since $Y$ is closed we can also write $Y = \bigcap_{j \in J} \Supp(F_j)$ for some $F_j \in \CC^\omega$. Let $\mathscr{K}$ be the thick tensor ideal generated by the $E_i$, and let $\mathscr{S}$ be the collection of finite tensor products of the $F_j$, including the empty product $\one$. If $\mathscr{K} \cap \mathscr{S}$ were empty, then by \cite{balmerfinclass} Lemma 2.2 there would be a prime $\mathscr{P}$ containing $\mathscr{K}$ and disjoint from $\mathscr{S}$. Such a $\mathscr{P}$ lies in $\Supp(F_j)$ for every $j$ but in no $\Supp(E_i)$, which contradicts $\bigcap_j \Supp(F_j) = \bigcup_i \Supp(E_i)$. Some $F = F_{j_1} \otimes \cdots \otimes F_{j_m}$ therefore lies in $\mathscr{K}$. Hence $F$ lies in the thick tensor ideal generated by $E = E_{i_1} \oplus \cdots \oplus E_{i_k}$ for some $i_1, \ldots, i_k \in I$. By Lemma \ref{lem: support ideals}, $\Supp(F) \subseteq \Supp(E)$. By Lemma \ref{lem: support properties},
    \[Y \subseteq \Supp(F) \subseteq \Supp(E) \subseteq Y.\]
    Hence $Y = \Supp(E)$.
\end{proof}

We now have the necessary background to state Balmer's classification result. 

\begin{thm}[\cite{balmerfinclass} Theorem 4.10, generalising \cite{thomason} Theorem 3.15]\label{thm: balmer classification}
    Let $\CC$ be stable and rigid. The assignment $Y \mapsto \{E \in \CC^\omega : \Supp(E) \subseteq Y\}$ is an order-preserving bijection from Thomason subsets of $\Spc(\CC^\omega)$ to radical thick tensor ideals of $\CC^\omega$, with inverse $\mathscr{I} \mapsto \bigcup_{E \in \mathscr{I}} \Supp(E)$.
\end{thm}

Superficially, the statement of this correspondence does not seem to relate to finite localisations. To understand the translation, recall that the kernel of a finite tensor-compatible localisation is a localising tensor ideal, and in particular thick. To obtain a classification result for finite localisations, we must understand any possible discrepancy between radical thick tensor ideals and localising tensor ideals.

\begin{lem}[\cite{balmerfinclass} Remark 4.3 and Proposition 4.4]\label{lem: compact acyclics radical}
    Let $\CC$ be stable and rigid, and let $L$ be a tensor-compatible localisation of $\CC$. Then the compact acyclics $\ker L \cap \CC^\omega$ form a radical thick tensor ideal of $\CC^\omega$.
\end{lem}

\begin{proof}
    Since $L$ is tensor-compatible, $\ker L$ is a tensor ideal by Lemma \ref{lem: equiv monoidal loc} \ref{cond: tensor ideal}, and it is localising, so in particular thick. Since $\CC$ is rigid, $\CC^\omega$ is closed under tensor products, so $\mathscr{I} = \ker L \cap \CC^\omega$ is a thick tensor ideal of $\CC^\omega$.

    To see that $\mathscr{I}$ is radical, first suppose that $E \in \CC^\omega$ satisfies $E^{\otimes 2} \in \mathscr{I}$. Since $E$ is dualisable, the triangle identities factor $\id_E$ as $(\id_E \otimes \ev_E) \circ (\coev_E \otimes \id_E)$ through $E \otimes DE \otimes E$, which lies in $\mathscr{I}$ because it is the tensor product of $E^{\otimes 2}$ with $DE$, up to reordering. So $E$ is a retract of an object of $\mathscr{I}$, and hence $E \in \mathscr{I}$. Now suppose $E^{\otimes n} \in \mathscr{I}$ for some $n \geq 2$. Since $2n - 2 \geq n$ and $\mathscr{I}$ is a tensor ideal, $(E^{\otimes (n-1)})^{\otimes 2} \simeq E^{\otimes (2n-2)}$ lies in $\mathscr{I}$, so $E^{\otimes (n-1)} \in \mathscr{I}$. By induction on $n$ we conclude that $E \in \mathscr{I}$.
\end{proof}

\begin{lem}[cf \cite{balmerfavi} Theorem 4.1 and Proposition 6.2]\label{lem: finite locs and tensor ideals}
    Let $\CC$ be stable and rigid. The assignment $L \mapsto \ker L \cap \CC^\omega$ is a bijection from finite tensor-compatible localisations of $\CC$ to thick tensor ideals of $\CC^\omega$, whose inverse sends $\mathscr{I}$ to the localisation with kernel $\langle \mathscr{I} \rangle$. Every finite tensor-compatible localisation of $\CC$ is smashing, and every thick tensor ideal of $\CC^\omega$ is radical.
\end{lem}

\begin{proof}
    Let $L$ be a finite tensor-compatible localisation. By Lemma \ref{lem: compact acyclics radical}, $\ker L \cap \CC^\omega$ is a thick tensor ideal. Since $L$ is finite, $\ker L = \langle \ker L \cap \CC^\omega \rangle$. A localisation is determined by its kernel, so the assignment $L \mapsto \ker L \cap \CC^\omega$ is injective.

    Conversely, let $\mathscr{I}$ be a thick tensor ideal of $\CC^\omega$. By the discussion following Lemma \ref{lem: acyclics gen by}, $\langle \mathscr{I} \rangle$ is the kernel of an accessible localisation $L$, which is finite by construction. The proof of Lemma \ref{lem: univ loc is univ mon loc}, with $\mathscr{I}$ as the generating set of acyclics and $\CC^\omega$ as the set generating $\CC$ under colimits and desuspensions, shows that $\langle \mathscr{I} \rangle$ is a tensor ideal. By Lemma \ref{lem: equiv monoidal loc} \ref{cond: tensor ideal}, $L$ is therefore tensor-compatible. Since $\CC$ is compactly generated and $\mathscr{I}$ is a thick subcategory of $\CC^\omega$, $\langle \mathscr{I} \rangle \cap \CC^\omega = \mathscr{I}$ (\cite{neemanENS} Lemma 2.2), i.e. $\ker L \cap \CC^\omega = \mathscr{I}$.

    Every finite tensor-compatible localisation is smashing by Lemma \ref{lem: fin implies smashing}, since compact objects are dualisable in a rigid category. Every thick tensor ideal of $\CC^\omega$ has the form $\ker L \cap \CC^\omega$, so it is radical by Lemma \ref{lem: compact acyclics radical}.
\end{proof}

The previous two results allow us to convert Balmer's classification result into a bijection between Thomason subsets of the Balmer spectrum and finite tensor-compatible localisations. 

\begin{cor}\label{cor: balmer correspondence}
    Let $\CC$ be stable and rigid. Sending a finite tensor-compatible localisation $L$ of $\CC$ to the Thomason subset $\bigcup_{E \in \ker L \cap \CC^\omega} \Supp(E)$ is a bijection from finite tensor-compatible localisations of $\CC$ to Thomason subsets of $\Spc(\CC^\omega)$.
\end{cor}

\begin{proof}
    Compose the two bijections of Lemma \ref{lem: finite locs and tensor ideals} and Theorem \ref{thm: balmer classification}.
\end{proof}

\begin{notation}\label{ntn: L_S}
    Let $\CC$ be stable and rigid. If $Y \subseteq \Spc(\CC^\omega)$ is Thomason, write $L_Y$ for the finite tensor-compatible localisation of $\CC$ corresponding to $Y$ in Corollary \ref{cor: balmer correspondence}, and say $L_Y$ has \emph{support} $Y$. The compact acyclics of $L_Y$ are the $E \in \CC^\omega$ with $\Supp(E) \subseteq Y$, and $L_{Y_1} \leq L_{Y_2}$ if and only if $Y_2 \subseteq Y_1$. As mentioned in the introduction, $Y$ is precisely the set of prime ideals of $\CC^\omega$ not containing $\ker L_Y \cap \CC^\omega$.
\end{notation}

Now that we understand the relationship between the Balmer spectrum and finite localisations, we are ready to reinterpret Theorems \ref{intro thm: p-local compactly central} and \ref{intro thm: sp compactly central} using the Balmer spectrum. 

By the thick subcategory theorem (\cite{NilpotenceII} Theorem 7), the points of $\Spc(\Sp^\omega)$ are the primes $\PP_{p,n}$ for $p \in \Z$ a prime number and $0 \le n \le \infty$. The prime ideal $\PP_{p,n}$ consists of finite spectra which are acyclic for the Morava $K$-theory $K(n)$ at $p$ (see \cite{balmerspectra} Corollary 9.5). As $p$ varies, the prime ideals $\PP_{p,0}$ coincide, being the ideal of finite spectra which vanish rationally. This is the generic point of $\Spc(\Sp^\omega)$. The space $\Spc(\Sp_{(p)}^\omega)$ is the chain of points $\PP_{p,0} \supset \PP_{p,1} \supset \cdots \supset \PP_{p,\infty}$ at the single prime number $p$. The prime $\PP_{p,\infty} \in \Spc(\Sp_{(p)}^\omega)$ is the zero ideal, and is a closed point. The prime $\PP_{p,\infty} \in \Spc(\Sp^\omega)$ consists of all finite spectra whose $p$-localisation vanishes, and is also a closed point. In both categories, the point $\PP_{p,\infty}$ is not the support of any object because its complement is not quasicompact.

The support of a nonzero $p$-local finite spectrum $X$ of type $n < \infty$ is the set of primes $\PP_{p,n} \supset \PP_{p,n+1} \supset \cdots \supset \PP_{p,\infty}$, because $X \in \PP_{p, n-1}$ and $X \notin \PP_{p, n}$. Also $\Supp(0) = \emptyset$. Hence for $-1 \leq n \leq \infty$, the support of $L_n^f$ is $\Supp(X_{p, n+1})$, where $X_{p, n+1}$ is $p$-local finite of type $n+1$. All the compactly central localisations of Theorems \ref{intro thm: p-local compactly central} and \ref{intro thm: sp compactly central} have supports which are closed Thomason subsets of the appropriate Balmer spectrum.

By Lemma \ref{lem: closed thomason subsets}, the closed Thomason subsets of $\Spc(\Sp^\omega)$ are exactly the supports of finite spectra, which are computed in \cite{balmerspectra} Corollary 9.5 (d). If $X$ is not torsion, $\Supp(X) = \Spc(\Sp^\omega)$. If $X$ is torsion, $\Supp(X) = \bigcup_p \{\PP_{p,m} : n_p \leq m \leq \infty\}$, where $p$ runs over the finitely many primes with $X_{(p)} \not\simeq 0$ and $1 \leq n_p < \infty$ is the type of $X_{(p)}$. Since every Thomason subset of $\Spc(\Sp_{(p)}^\omega)$ is closed, Theorems \ref{intro thm: p-local compactly central} and \ref{intro thm: sp compactly central} establish bijections between closed Thomason subsets of $\Spc(\CC^\omega)$ and compactly central localisations of $\CC$ for $\CC = \Sp_{(p)}, \Sp$ respectively. This suggests the following direction of investigation.

\begin{question}\label{key question}
For which categories $\CC$ is the map sending a compactly central localisation of $\CC$ to its support a bijection onto closed Thomason subsets of $\Spc(\CC^\omega)$?
\end{question}

We will resolve part of this question in Lemma \ref{lem: balmer central translator}, for general stable rigid $\CC$. The full bijection is proved in Theorem \ref{thm: closed implies cc for R} for $D(R)$, and in Theorem \ref{thm: closed implies cc ample family} for $D(X)$ where $X$ possesses an ample family of line bundles, including any quasiprojective scheme over an affine base. 

\subsection{Comparing spectra to \texorpdfstring{$D(X)$}{D(X)}}\label{subsec: spectra vs qcqs schemes}

Let us consider the main inputs to the proof of Theorem \ref{intro thm: p-local compactly central}. We used the nilpotence theorem for $\Sp_{(p)}$, and the K\"unneth formula for the nilpotence-detecting spectra $K(n)$. We used the thick subcategory theorem, which classifies finite localisations of $\Sp_{(p)}$, so in effect we used that the Balmer spectrum of $\Sp_{(p)}^\omega$ is known. We took tensor powers of algebraically central maps to produce central ones. Finally, we constructed an explicit supply of algebraically central maps using Smith-Toda complexes. Many of these inputs are readily available in $D(X)$, which is strong motivation to consider generalising our results from $\Sp$ to $D(X)$. In this Subsection we record the relevant properties of $D(X)$ and write down a general definition capturing when a rigid stable category $\CC$ has an appropriate nilpotence-detecting family (Definition \ref{def: nilpotence-detecting family}). 

\begin{notation}\label{ntn: D(X)}
    Let $R$ denote a commutative ring, and $X$ a qcqs scheme. We write $D(X)$ for the derived $\infty$-category of complexes of $\OO_X$-modules with quasicoherent cohomology, with symmetric monoidal structure given by the derived tensor product $\otimes_{\OO_X}$, and unit $\one = \OO_X$. When $X = \Spec R$ this recovers the derived $\infty$-category $D(R)$ of $R$-modules. For a point $x \in X$ with residue field $k(x)$, write $E \otimes_{\OO_X} k(x) \in D(k(x))$ for the derived pullback of $E \in D(X)$ along $\Spec k(x) \to X$, and $k(x)_* E$ for its homotopy groups. A map $f: E \to F$ in $D(X)$ induces a map $k(x)_* f: k(x)_* E \to k(x)_* F$ of graded $k(x)$-vector spaces. Since $- \otimes_{\OO_X} k(x)$ is symmetric monoidal and $k(x)$ is a field, there is a K\"unneth isomorphism $k(x)_*(E \otimes F) \cong k(x)_* E \otimes_{k(x)} k(x)_* F$.
\end{notation}

We collect some standard results about $D(X)$ here. As always, $X$ is a qcqs scheme and $R$ a commutative ring.
\begin{enumerate}[label=(\roman*)]
    \item $D(X)$ is a presentably symmetric monoidal stable $\infty$-category (\cite{sag} Propositions 2.2.4.1 and 2.2.4.2 and Corollary 2.2.6.2). \label{enum: DX is stable}
    \item $D(X)$ is compactly generated and its compact objects are the perfect complexes (\cite{bondalvandenbergh} Theorem 3.1.1, and see \cite{neeman-duality} Proposition 2.5 for the case where $X$ is separated). \label{enum: DX compactly generated}
    \item A complex in $D(X)$ is dualisable if and only if it is perfect (\cite{stacks} Tags 0FPC, 0FPD, 08DQ, and 08E4). Combined with \ref{enum: DX compactly generated}, $D(X)$ is therefore rigid. \label{enum: DX is rigid}
    \item If $X = \Spec R$, $D(X)$ is generated by the unit $R = \one$ under colimits and desuspensions (\cite{ha} Remark 7.1.1.16, Definition 7.2.4.1, and Proposition 7.2.4.2). By Proposition \ref{prop: when all locs are monoidal}, every localisation of $D(R)$ is then tensor-compatible. \label{enum: DR every loc tensor compatible}
\end{enumerate}

It is known that $\Spc(\Perf(X)) \cong X$. Indeed there is a concrete identification. The following Theorem relies on Thomason's classification of thick tensor ideals (\cite{thomason} Theorem 3.15). It was proved by Balmer for topologically Noetherian schemes (\cite{balmerfinclass} Corollary 5.6) and by Buan-Krause-Solberg for all qcqs schemes (\cite{bks} Theorem 9.5).

\begin{thm}[\cite{bks} Theorem 9.5, see also \cite{balmerfinclass} Corollary 5.6]\label{thm: Spc Perf X}
    Let $X$ be a qcqs scheme. The map 
    \begin{align*}
    X &\to \Spc(\Perf(X))\\
    x &\mapsto \mathscr{P}_x = \{E \in \Perf(X) : E_x \simeq 0\}
    \end{align*}
    is a homeomorphism. It sends a point $x \in X$ to the prime ideal of perfect complexes with vanishing stalk at $x$.
\end{thm}

\begin{lem}[\cite{bks} Section 9.2 and \cite{thomason} Lemma 3.3]\label{lem: supp via residue fields}
    Let $X$ be a qcqs scheme and $E \in \Perf(X)$. Under the homeomorphism of Theorem \ref{thm: Spc Perf X}, $\Supp(E) = \{x \in X : k(x)_* E \neq 0\}$.
\end{lem}

\begin{proof}
    By Theorem \ref{thm: Spc Perf X}, $x \in \Supp(E)$ if and only if $E \notin \mathscr{P}_x$ if and only if the stalk $E_x$ is nonzero. The stalk $E_x$ is a perfect complex over the local ring $\OO_{X,x}$, and $E \otimes_{\OO_X} k(x) \simeq E_x \otimes_{\OO_{X,x}} k(x)$. By \cite{thomason} Lemma 3.3, $E_x \simeq 0$ if and only if $E_x \otimes_{\OO_{X,x}} k(x) \simeq 0$ if and only if $k(x)_* E = 0$.
\end{proof}

\begin{notation}[cf Notation \ref{ntn: L_S}]\label{ntn: identify Spc with X}
    For a qcqs scheme $X$, use Theorem \ref{thm: Spc Perf X} to identify $\Spc(\Perf(X))$ with $X$. Thus Thomason subsets of $\Spc(\Perf(X))$ are subsets of $X$, and for a Thomason subset $Y \subseteq X$ we write $L_Y$ for the finite tensor-compatible localisation of $D(X)$ whose support is $Y$.
\end{notation}

By Lemma \ref{lem: closed thomason subsets}, the closed Thomason subsets of $\Spc(\Perf(X))$ are closed subsets of $X$ whose complements are quasicompact. These can be described explicitly for a qcqs scheme $X$.

\begin{lem}\label{lem: closed thomason subsets of a scheme}
    Let $X$ be a qcqs scheme. Closed Thomason subsets of $X$ are finite intersections of subspaces $X \setminus U$ where $U$ is an affine open in $X$.
\end{lem}

\begin{proof}
    Let $Z \subseteq X$ be a closed Thomason subset. By Lemma \ref{lem: closed thomason subsets}, the complement $X \setminus Z$ is quasicompact. The affine opens form a basis for the topology on $X$, so $X \setminus Z = U_1 \cup \cdots \cup U_n$ for some affine opens $U_i$, and hence $Z = \bigcap_{i=1}^n (X \setminus U_i)$.
\end{proof}

We now understand how one side of the correspondence in Theorems \ref{intro thm: p-local compactly central} and \ref{intro thm: sp compactly central} may be transported to $D(X)$. We know precisely what the finite localisations and closed Thomason subsets are. Next we want a nilpotence theorem.

\begin{thm}[\cite{thomason} Theorem 3.6]\label{thm: thomason nilpotence}
    Let $X$ be a qcqs scheme and let $f: E \to F$ be a map in $\Perf(X)$. Then $f^{\otimes n} \simeq 0$ for some $n \geq 1$ if and only if $k(x)_* f = 0$ for every point $x \in X$. That is, the residue fields of $X$ jointly detect tensor-nilpotence of maps between compact objects of $D(X)$.
\end{thm}

This is closely analogous to the situation in $\Sp_{(p)}$, and motivates a general framework for settings where our methods may apply. Recall that a graded field is a graded commutative ring in which every nonzero homogeneous element is invertible. Recall also that a homological functor $h_i$ sends a cofibre sequence in $\CC$ to a long exact sequence of graded modules, such that $h_i(\Sigma E)$ is $h_i(E)$ with degrees shifted up by one.

\begin{defn}\label{def: nilpotence-detecting family}
    Let $\CC$ be stable and rigid. A \emph{nilpotence-detecting family} for $\CC$ is a collection $\{h_i\}_{i \in I}$ with the following properties. Each $h_i$ is a homological functor from $\CC$ to graded modules over a graded field $k_i$. Each $h_i$ comes with an isomorphism $h_i(\one) \cong k_i$ and natural K\"unneth isomorphisms $h_i(E \otimes F) \cong h_i(E) \otimes_{k_i} h_i(F)$, compatible with the unit and associativity equivalences of $\CC$. Finally, the $h_i$ jointly detect tensor-nilpotence of maps between compact objects: if $f: E \to F$ is a map in $\CC^\omega$ with $h_i(f) = 0$ for every $i \in I$, then $f^{\otimes n} \simeq 0$ for some $n \geq 1$.
\end{defn}

\begin{defn}[cf Definition \ref{def: algebraically central}]\label{def: algebraically central general}
    Let $\CC$ be stable and rigid with a nilpotence-detecting family $\{h_i\}_{i \in I}$. A map $f: \one \to A$ in $\CC$ is \emph{algebraically central} if for every $i \in I$, the map $h_i(f)$ is either zero or an isomorphism.
\end{defn}

\begin{example}\label{ex: algebraically central for schemes}
    Let $X$ be a qcqs scheme. The functors $\{k(x)_*\}_{x \in X}$ of Notation \ref{ntn: D(X)} form a nilpotence-detecting family for $D(X)$ by Theorem \ref{thm: thomason nilpotence}. With respect to this nilpotence-detecting family, a map $f: \one \to A$ in $D(X)$ is algebraically central if for every point $x \in X$, the map $k(x)_* f$ is either zero or an isomorphism.
\end{example}

\begin{example}\label{ex: alg central consistent}
    The Morava $K$-theories $K(n)_*$ for $0 \leq n \leq \infty$ form a nilpotence-detecting family for $\Sp_{(p)}$ by \cite{NilpotenceII} Theorem 3. For a compact map $f: \S_{(p)} \to A$ in $\Sp_{(p)}$, asking that each $K(n)_* f$ be zero or an isomorphism recovers Definition \ref{def: algebraically central}, because $\cof f$ is finite and so has a type.
\end{example}

We have now introduced all the key inputs in the $D(X)$ setting. It is time to look at the behaviour of compact central maps in a stable rigid $\infty$-category $\CC$ with a nilpotence-detecting family.

\subsection{Supports of compactly central localisations}\label{subsec: cc locs closed support}

Generalising Subsection \ref{subsec: props of central maps}, we show that central maps are algebraically central. We compute their support in the $\CC = D(X)$ case. We also prove that the support of any compactly central localisation is a closed Thomason subset of the Balmer spectrum, resolving one direction of Question \ref{key question}.

\begin{lem}[cf Lemma \ref{lem: central map 0 or iso on Kn}]\label{lem: central implies alg central for schemes}
    Let $\CC$ be stable and rigid with a nilpotence-detecting family. Every central map $\one \to A$ in $\CC$ is algebraically central.
\end{lem}

\begin{proof}
    Recall from Definition \ref{def: nilpotence-detecting family} that we have nilpotence-detecting functors $h_i$ on $\CC$ for $i \in I$. Let $f: \one \to A$ be central and fix $i \in I$. The K\"unneth isomorphisms are natural and compatible with the unit equivalences, so applying $h_i$ to the homotopy $f \otimes \id_A \simeq \id_A \otimes f$ shows that $h_i(f): k_i \to h_i(A)$ is a central map of graded $k_i$-modules. As in the proof of Lemma \ref{lem: central map 0 or iso on Kn}, such a map is either zero or an isomorphism.
\end{proof}

\begin{lem}\label{lem: supp of alg central}
    Let $X$ be a qcqs scheme and let $\alpha: \one \to A$ be a compact algebraically central map in $D(X)$. Then $\Supp(\cof \alpha) = \{x \in X : k(x)_* \alpha = 0\}$.
\end{lem}

\begin{proof}
    By Lemma \ref{lem: supp via residue fields}, $x \in \Supp(\cof \alpha)$ if and only if $k(x)_* \cof \alpha \neq 0$. By the long exact sequence associated to the cofibre sequence $\one \xrightarrow{\alpha} A \to \cof \alpha$, we have $k(x)_* \cof \alpha = 0$ if and only if $k(x)_* \alpha$ is an isomorphism. Since $\alpha$ is algebraically central and $k(x)_* \one \cong k(x)$ is nonzero, $k(x)_* \alpha$ fails to be an isomorphism exactly when it is zero.
\end{proof}

\begin{lem}[cf Theorem \ref{thm: computing finite smashing loc from central map}]\label{lem: balmer central translator}
    Let $\CC$ be stable and rigid and let $\alpha: \one \to A$ be a compact central map in $\CC$. Then the compactly central localisation $L_\alpha$ presented by $\alpha$ has support equal to the closed subset $\Supp(\cof \alpha) \subseteq \Spc(\CC^\omega)$. That is, $L_\alpha \simeq L_{\Supp(\cof \alpha)}$. In particular, the support of $L_\alpha$ is closed Thomason.
\end{lem}

\begin{proof}
    By Lemma \ref{lem: compactly central implies finite}, $L_\alpha$ is finite. It is smashing, so it is also tensor-compatible. Let $\mathscr{I}$ be the thick tensor ideal of $\CC^\omega$ generated by $\cof \alpha$. We show that $\ker L_\alpha \cap \CC^\omega = \mathscr{I}$.

    By the proof of Lemma \ref{lem: compact acyclics generate}, each $\cof(\alpha^{\otimes n})$ is a compact $L_\alpha$-acyclic. By Lemma \ref{lem: compact acyclics radical}, $\ker L_\alpha \cap \CC^\omega$ is a thick tensor ideal. It contains $\cof \alpha$, so $\mathscr{I} \subseteq \ker L_\alpha \cap \CC^\omega$. By Lemma \ref{lem: baralpha in terms of alpha}, $\cof(\alpha^{\otimes (n+1)})$ is an extension of $\cof \alpha$ by $A \otimes \cof(\alpha^{\otimes n})$, so by induction on $n$ each $\cof(\alpha^{\otimes n})$ lies in $\mathscr{I}$. The same Lemma expresses $\cof(\one \to J_\alpha)$ as a sequential colimit of the objects $\cof(\alpha^{\otimes n})$, and then the proof of Lemma \ref{lem: compactly central implies finite} shows that every $L_\alpha$-acyclic is a filtered colimit of objects $\cof(\alpha^{\otimes n}) \otimes B$ with $B$ compact. Each of these lies in $\mathscr{I}$, so $\ker L_\alpha = \langle \mathscr{I} \rangle$, and $\ker L_\alpha \cap \CC^\omega = \mathscr{I}$ by Lemma \ref{lem: finite locs and tensor ideals}.

    The support of $L_\alpha$ is therefore $\bigcup_{E \in \mathscr{I}} \Supp(E) = \Supp(\cof \alpha)$, applying Lemma \ref{lem: support ideals}. Hence $L_\alpha \simeq L_{\Supp(\cof \alpha)}$ by Corollary \ref{cor: balmer correspondence}.
\end{proof}

\begin{lem}[cf proof of Corollary \ref{cor: some central maps exist}]\label{lem: support of tensor power}
    Let $\CC$ be stable and rigid, and let $\beta: \one \to B$ be a compact map in $\CC$. Then $\Supp(\cof(\beta^{\otimes N})) = \Supp(\cof \beta)$ for every $N \geq 1$. In particular, if $\beta^{\otimes N}$ is central then $L_{\beta^{\otimes N}}$ has support $\Supp(\cof \beta)$.
\end{lem}

\begin{proof}
    We show that a thick tensor ideal $\mathscr{I}$ of $\CC^\omega$ contains $\cof(\beta^{\otimes N})$ if and only if it contains $\cof \beta$. Applying this to prime ideals gives the equality of supports. The last statement then follows from Lemma \ref{lem: balmer central translator}, since $\beta^{\otimes N}$ is compact because $\CC$ is rigid.

    By Lemma \ref{lem: finite locs and tensor ideals} there is a smashing localisation $L$ of $\CC$ with $\ker L \cap \CC^\omega = \mathscr{I}$. The maps $\beta$ and $\beta^{\otimes N}$ are compact and $L$ is exact, so $\cof \beta \in \mathscr{I}$ if and only if $L\beta$ is an equivalence. Likewise for $\beta^{\otimes N}$. Since $L$ is smashing, $L(\beta^{\otimes N}) \simeq (L\beta)^{\otimes N}$. It therefore suffices to show that a map $f: \one \to B$ in a stable symmetric monoidal $\infty$-category is an equivalence if and only if $f^{\otimes N}$ is.

    Assume $f^{\otimes N}$ is an equivalence. Then $f$ has a left inverse, because $f^{\otimes N}$ factors as $(f^{\otimes(N-1)} \otimes \id_B) \circ f$. Write $B \simeq \one \oplus \cof f$, with $f$ the inclusion of the first summand. It follows that $f^{\otimes N}$ is the inclusion of the summand $\one^{\otimes N}$ of $(\one \oplus \cof f)^{\otimes N}$, and one of the complementary summands is $\one^{\otimes(N-1)} \otimes \cof f \simeq \cof f$. Since $f^{\otimes N}$ is an equivalence, all the complementary summands vanish, so $\cof f \simeq 0$ and $f$ is an equivalence.
\end{proof}

Now that we can compute the support of the localisation presented by a central map, we could try to start directly constructing families of central maps. Instead, we first study methods for combining compact central maps and turning algebraically central maps into central ones. This reduces the number of input maps that we must construct.

\subsection{Unions and intersections of supports}\label{subsec: unions and intersections}

Closed Thomason subsets of the Balmer spectrum may be combined by taking finite intersections and unions, and the resulting subset is still closed Thomason. We investigate what the result of these operations is on the level of localisations. In particular, we would like to understand how to combine compactly central localisations via methods corresponding to intersection and union on the level of supports. 

\begin{notation}\label{ntn: alpha_S}
    Let $\CC$ be stable and rigid and $Y \subseteq \Spc(\CC^\omega)$ a Thomason subset. If $L_Y$ is compactly central, let $\alpha_Y: \one \to A_Y$ denote a choice of compact central map presenting $L_Y$. Then $Y = \Supp(\cof \alpha_Y)$ by Lemma \ref{lem: balmer central translator}.
\end{notation}

Taking unions of the supports of compactly central localisations is simple. This operation corresponds to tensoring together the compact central maps, which preserves both compactness and centrality if the union is finite. 

\begin{lem}\label{lem: union cc}
    Let $\CC$ be stable and rigid, and let $S, T \subseteq \Spc(\CC^\omega)$ be Thomason subsets such that $L_S$ and $L_T$ are compactly central. Then $L_{S \cup T}$ is compactly central, presented by the compact central map $\alpha_S \otimes \alpha_T$.
\end{lem}

\begin{proof}
    By Lemma \ref{lem: composite of cc locs is cc}, $\alpha_S \otimes \alpha_T$ is compact and central, and $L_{\alpha_S \otimes \alpha_T} \simeq L_S \otimes L_T$. By Lemma \ref{lem: combining smashing locs}, $\ker(L_S \otimes L_T)$ is generated under colimits by $\ker L_S \cup \ker L_T$. Since $L_S$ and $L_T$ are finite (Lemma \ref{lem: compactly central implies finite}), it follows that $\ker(L_S \otimes L_T) = \langle \mathscr{I} \rangle$, where $\mathscr{I}$ is the thick tensor ideal generated by the union of the compact acyclics of $L_S$ and of $L_T$. By Lemma \ref{lem: finite locs and tensor ideals}, $\mathscr{I} = \langle \mathscr{I}\rangle \cap \CC^{\omega}$ is exactly the collection of compact acyclics of $L_S \otimes L_T$. By Lemma \ref{lem: support ideals}, the union of the supports of objects of $\mathscr{I}$ is $S \cup T$. Hence $L_{\alpha_S \otimes \alpha_T} \simeq L_{S \cup T}$.
\end{proof}

Intersections are more delicate. The operation on central maps which corresponds to intersection of supports is to take the fibres of the central maps presenting the localisations, tensor them together, and pass to the cofibre (see Corollary \ref{cor: box construction works}). This operation need not preserve centrality. However, it does preserve \emph{algebraic centrality} (Definition \ref{def: algebraically central general}) in any setting where we have a nilpotence-detecting family. In Subsection \ref{subsec: alg central to central}, we already studied how one may pass from algebraically central maps to central maps in $\Sp_{(p)}$. Those methods generalise to any stable rigid category $\CC$ with a nilpotence-detecting family, allowing us to find an operation on compact central maps that corresponds to intersection of supports. In Construction \ref{constr: box} and Corollary \ref{cor: box construction works} we build the operation in a way which outputs an algebraically central map, and then Theorem \ref{thm: alg central power is central general} and Lemma \ref{lem: intersection cc} allow us to pass to a large tensor power to make the map central. It is helpful to think about the properties which characterise the fibre of an algebraically central map.

\begin{defn}\label{def: algebraically cocentral}
    Let $\CC$ be stable and rigid with a nilpotence-detecting family $\{h_i\}_{i \in I}$. A map $\varphi: B \to \one$ in $\CC$ is \emph{algebraically co-central} if for every $i \in I$, either $h_i(\varphi)$ is surjective or $h_i(B) = 0$.
\end{defn}

\begin{lem}\label{lem: alg cocentral key prop}
    Let $\CC$ be stable and rigid with a nilpotence-detecting family, and let $B \xrightarrow{\varphi} \one \xrightarrow{f} \cof \varphi$ be a cofibre sequence in $\CC$. Then $\varphi$ is algebraically co-central if and only if $f$ is algebraically central.
\end{lem}

\begin{proof}
    Use the long exact sequence obtained by applying $h_i$ to $B \xrightarrow{\varphi} \one \xrightarrow{f} \cof \varphi$, for each $i \in I$. We see that $h_i(B) = 0$ if and only if $h_i(f)$ is an isomorphism, and $h_i(\varphi)$ is surjective if and only if $h_i(f) = 0$.
\end{proof}

\begin{lem}\label{lem: tensor-closed classes of maps}
    Let $\CC$ be stable and rigid with a nilpotence-detecting family. Each of the following classes of maps in $\CC$ is closed under finite tensor products: central maps $\one \to A$, algebraically central maps $\one \to A$, and algebraically co-central maps $B \to \one$.
\end{lem}

\begin{proof}
    By induction it suffices to consider two maps. For central maps this is Lemma \ref{lem: composite of cc locs is cc}. For the other two classes we work one $h_i$ at a time, using the K\"unneth isomorphism $h_i(f \otimes g) \cong h_i(f) \otimes_{k_i} h_i(g)$. The tensor product of a map which is zero or an isomorphism with another such map is itself zero or an isomorphism, so algebraically central maps are closed under tensor products. Let $\varphi: B \to \one$ and $\psi: C \to \one$ be algebraically co-central. If either $h_i(B) = 0$ or $h_i(C) = 0$ then $h_i(B \otimes C) = 0$. Otherwise $h_i(\varphi)$ and $h_i(\psi)$ are both surjective, and hence so is $h_i(\varphi \otimes \psi) \cong h_i(\varphi) \otimes_{k_i} h_i(\psi)$, since a tensor product of surjections of graded $k_i$-modules is surjective.
\end{proof}

\begin{constr}\label{constr: box}
    Let $\CC$ be stable and rigid, and let $\alpha: \one \to A$ and $\beta: \one \to B$ be maps in $\CC$ with fibre sequences $\fib \alpha \xrightarrow{\varphi_\alpha} \one \xrightarrow{\alpha} A$ and $\fib \beta \xrightarrow{\varphi_\beta} \one \xrightarrow{\beta} B$. Define $\alpha \boxtimes \beta$ to be the map $\one \to \cof(\varphi_\alpha \otimes \varphi_\beta)$ in the cofibre sequence
    \[\fib \alpha \otimes \fib \beta \xrightarrow{\varphi_\alpha \otimes \varphi_\beta} \one \xrightarrow{\alpha \boxtimes \beta} \cof(\varphi_\alpha \otimes \varphi_\beta),\]
    so that $\fib(\alpha \boxtimes \beta) \simeq \fib \alpha \otimes \fib \beta$.
\end{constr}

\begin{cor}\label{cor: box construction works}
    Let $\CC$ be stable and rigid with a nilpotence-detecting family, and let $\alpha: \one \to A$ and $\beta: \one \to B$ be compact algebraically central maps in $\CC$. Then $\alpha \boxtimes \beta$ is compact and algebraically central, and
    \[\Supp(\cof(\alpha \boxtimes \beta)) = \Supp(\cof \alpha) \cap \Supp(\cof \beta).\]
\end{cor}

\begin{proof}
    Since $\CC$ is rigid, $\fib(\alpha \boxtimes \beta) \simeq \fib \alpha \otimes \fib \beta$ is compact. By Lemma \ref{lem: alg cocentral key prop}, $\varphi_\alpha$ and $\varphi_\beta$ are algebraically co-central, so $\varphi_\alpha \otimes \varphi_\beta$ is algebraically co-central by Lemma \ref{lem: tensor-closed classes of maps}, and therefore $\alpha \boxtimes \beta$ is algebraically central by Lemma \ref{lem: alg cocentral key prop} again. Finally, Lemma \ref{lem: support properties} gives
    \[\Supp(\cof(\alpha \boxtimes \beta)) = \Supp(\fib \alpha \otimes \fib \beta) = \Supp(\cof \alpha) \cap \Supp(\cof \beta).\]
\end{proof}

\begin{thm}[cf Theorem \ref{thm: alpha is central}]\label{thm: alg central power is central general}
    Let $\CC$ be stable and rigid with a nilpotence-detecting family. If $\alpha: \one \to A$ is a compact algebraically central map in $\CC$, then $\alpha^{\otimes N}$ is central for $N \gg 0$.
\end{thm}

\begin{proof}
    Repeat the arguments of Subsection \ref{subsec: alg central to central} from Construction \ref{constr: averaging idempotents} through to Theorem \ref{thm: alpha is central} verbatim. Simply replace $K(n)_*$ by $h_i$, \cite{NilpotenceII} Theorem 3 by detection of nilpotence from Definition \ref{def: nilpotence-detecting family}, $\Sp_{(p)}$ by $\CC$, $\S_{(p)}$ by $\one$, and $n \in [0, \infty]$ by $i \in I$. Algebraically central in this upgraded proof refers to Definition \ref{def: algebraically central general} rather than Definition \ref{def: algebraically central} as it did in Subsection \ref{subsec: alg central to central}.
\end{proof}

\begin{lem}\label{lem: intersection cc}
    Let $\CC$ be stable and rigid with a nilpotence-detecting family. Let $S, T \subseteq \Spc(\CC^\omega)$ be Thomason subsets such that $L_S$ and $L_T$ are compactly central. Then $L_{S \cap T}$ is compactly central. Explicitly, $(\alpha_S \boxtimes \alpha_T)^{\otimes N}$ is a compact central map presenting $L_{S \cap T}$ for $N \gg 0$.
\end{lem}

\begin{proof}
    By Lemma \ref{lem: central implies alg central for schemes}, the compact central maps $\alpha_S$ and $\alpha_T$ are algebraically central. By Corollary \ref{cor: box construction works}, $\alpha_S \boxtimes \alpha_T$ is therefore compact and algebraically central with $\Supp(\cof(\alpha_S \boxtimes \alpha_T)) = S \cap T$. By Theorem \ref{thm: alg central power is central general}, $\varphi = (\alpha_S \boxtimes \alpha_T)^{\otimes N}$ is central for $N \gg 0$, and it is compact because $\CC$ is rigid. By Lemma \ref{lem: support of tensor power}, $L_\varphi$ has support $S \cap T$, so $L_\varphi \simeq L_{S \cap T}$ and $L_{S \cap T}$ is compactly central.
\end{proof}

\begin{thm}\label{thm: closed implies cc for R}
    Let $R$ be a commutative ring. Every finite localisation of $D(R)$ whose support is a closed Thomason subset of $\Spec R$ is compactly central.
\end{thm}

\begin{proof}
    From \ref{enum: DX is stable}, \ref{enum: DX is rigid}, and \ref{enum: DR every loc tensor compatible}, $D(R)$ is stable and rigid, and every localisation of $D(R)$ is tensor-compatible.

    First let $s \in R$, and write $s: \one \to \one$ for multiplication by $s$. Any endomorphism of the unit is central, because the left and right unit equivalences of $\one$ agree, so $s$ is central. Its cofibre is the perfect complex $R \xrightarrow{s} R$, so $s$ is compact. For $\mathfrak{p} \in \Spec R$, $k(\mathfrak{p})_* \cof s$ is the cofibre of multiplication by the image of $s$ on $k(\mathfrak{p})$, which is nonzero exactly when $s \in \mathfrak{p}$. So $\Supp(\cof s) = V(s)$ by Lemma \ref{lem: supp via residue fields}, and $L_s \simeq L_{V(s)}$ by Lemma \ref{lem: balmer central translator}. Hence $L_{V(s)}$ is compactly central.

    Now let $Z \subseteq \Spec R$ be a closed Thomason subset, with corresponding localisation $L_Z$ supported on $Z$. The distinguished opens $D(s)$ form a basis for the topology on $\Spec R$, and $\Spec R \setminus Z$ is quasicompact by Lemma \ref{lem: closed thomason subsets}. Hence $Z = V(s_1) \cap \cdots \cap V(s_n)$ for some $n \geq 1$ and $s_1, \ldots, s_n \in R$. The category $D(R)$ has the nilpotence-detecting family of Example \ref{ex: algebraically central for schemes}, so Lemma \ref{lem: intersection cc} applies. By induction on $n$, $L_Z$ is compactly central.
\end{proof}

\subsection{Schemes with an ample family}\label{subsec: cc locs building blocks scheme}

In the proof of Theorem \ref{thm: closed implies cc for R} we produced compactly central localisations of $D(R)$ from the maps $s: \one \to \one$ given by multiplication by elements $s \in R$. On a general qcqs scheme $X$ there are too few global sections of $\OO_X$ for this to work. For example, if $k$ is a field then $\Gamma(\P^1_k, \OO_{\P^1_k}) \cong k$, so multiplication by a global section is either zero or an equivalence. Instead we use global sections of line bundles, which give maps $\one \to \LL$ that are always central. When $X$ has an ample family of line bundles, the zero loci of such sections generate all closed Thomason subsets of $X$ under finite intersections.

Let $X$ be a qcqs scheme and $\LL$ a line bundle on $X$, i.e. an invertible $\OO_X$-module. For $s \in \Gamma(X, \LL)$ and $x \in X$, write $s(x)$ for the image of $s$ in the fibre $\LL(x) = \LL_x \otimes_{\OO_{X,x}} k(x)$. Let $X_s = \{x \in X : s(x) \neq 0\}$, which is open, and let $V(s) = X \setminus X_s$. Since $\LL$ is concentrated in degree zero, $[\one, \LL] \cong \Gamma(X, \LL)$, and we also write $s: \one \to \LL$ for the map in $D(X)$ corresponding to $s$.

\begin{lem}\label{lem: sections of line bundles are central}
    Let $X$ be a qcqs scheme, $\LL$ a line bundle on $X$, and $s \in \Gamma(X, \LL)$. Then $s: \one \to \LL$ is a compact central map in $D(X)$, and the localisation $L_s$ it presents has support $V(s)$. In particular $L_{V(s)}$ is compactly central.
\end{lem}

\begin{proof}
    The two maps $s \otimes \id_{\LL}$ and $\id_{\LL} \otimes s$ from $\LL$ to $\LL \otimes \LL$ differ by the symmetry equivalence of $\LL \otimes \LL$. Since $\LL$ is flat, $\LL \otimes \LL$ is the underived tensor product, which is again a line bundle in degree zero, and its symmetry equivalence is the identity. On an open $U$ where $\LL$ is trivial, $\LL|_U = \OO_U e$ for a generator $e$, and $fe \otimes ge = fg (e \otimes e) = ge \otimes fe$. Both maps are maps of sheaves concentrated in degree zero, so they agree in $D(X)$, and $s$ is central. Since $\one$ and $\LL$ are perfect, so is $\cof s$, and $s$ is compact.

    For $x \in X$, the map $k(x)_* s: k(x) \to k(x)_* \LL \cong \LL(x)$ is the map of 1-dimensional $k(x)$-vector spaces sending $1$ to $s(x)$. So $k(x)_* s$ is zero if $s(x) = 0$ and an isomorphism otherwise. Hence $s$ is algebraically central and $\Supp(\cof s) = V(s)$ by Lemma \ref{lem: supp of alg central}, so $L_s \simeq L_{V(s)}$ by Lemma \ref{lem: balmer central translator}.
\end{proof}

\begin{defn}[\cite{thomasontrobaugh} Definition 2.1.1, or \cite{sga6} II Proposition 2.2.3 and Definition 2.2.4]\label{def: ample family}
    A family $\{\LL_\lambda\}_{\lambda \in \Lambda}$ of line bundles on a qcqs scheme $X$ is \emph{ample} if either of the following two equivalent conditions holds:
    \begin{itemize}
        \item Every point of $X$ lies in an affine open subset of the form $X_s$, for some $\lambda \in \Lambda$, $n \geq 1$, and $s \in \Gamma(X, \LL_\lambda^{\otimes n})$.
        \item Those open subsets $X_s$ which are affine form a basis for the topology on $X$, where $\lambda \in \Lambda$, $n \geq 1$, and $s \in \Gamma(X, \LL_\lambda^{\otimes n})$.
    \end{itemize}
\end{defn}

Not every open subset $X_s \subseteq X$ coming from a section of a line bundle is necessarily affine. We are asking for the family of line bundles to produce enough affine open subsets of this form to cover $X$. 

\begin{example}[\cite{thomasontrobaugh} Example 2.1.2 (a)--(d), or \cite{sga6} II Corollary 2.2.7.1]\label{ex: ample families}
    If $X = \Spec R$ is affine then $\{\OO_X\}$ is an ample family, since $X_s = D(s)$ for $s \in R$. Every scheme which is quasiprojective over an affine scheme has an ample family consisting of a single ample line bundle, and every separated regular Noetherian scheme has an ample family of line bundles.
\end{example}

\begin{lem}[cf \cite{sga6} II Proposition 2.2.3]\label{lem: closed thomason subsets from sections}
    Let $X$ be a qcqs scheme with an ample family of line bundles $\{\LL_\lambda\}_{\lambda \in \Lambda}$. Every closed Thomason subset $Z \subseteq X$ has the form $Z = V(s_1) \cap \cdots \cap V(s_k)$ for some $k \geq 1$ and sections $s_j \in \Gamma(X, \LL_{\lambda_j}^{\otimes n_j})$. In particular this holds for $Z = X \setminus U$ with $U$ an affine open subset of $X$.
\end{lem}

\begin{proof}
    By Lemma \ref{lem: closed thomason subsets}, $X \setminus Z$ is quasicompact. If $X \setminus Z$ is empty, take $k = 1$ and $s_1 = 0$. Otherwise, from Definition \ref{def: ample family}, $X \setminus Z = X_{s_1} \cup \cdots \cup X_{s_k}$ for finitely many sections $s_j$ of the stated form. Hence $Z = V(s_1) \cap \cdots \cap V(s_k)$. An affine open $U$ is quasicompact, so $X \setminus U$ is closed with quasicompact complement, i.e. closed Thomason.
\end{proof}

\begin{thm}\label{thm: closed implies cc ample family}
    Let $X$ be a qcqs scheme with an ample family of line bundles. Every finite tensor-compatible localisation of $D(X)$ whose support is a closed Thomason subset of $X$ is compactly central.
\end{thm}

\begin{proof}
    From \ref{enum: DX is stable} and \ref{enum: DX is rigid}, $D(X)$ is stable and rigid. Let $Z \subseteq X$ be a closed Thomason subset, with corresponding localisation $L_Z$ supported on $Z$. By Lemma \ref{lem: closed thomason subsets from sections}, $Z = V(s_1) \cap \cdots \cap V(s_k)$ for sections $s_j$ of line bundles on $X$, and each $L_{V(s_j)}$ is compactly central by Lemma \ref{lem: sections of line bundles are central}. The category $D(X)$ has the nilpotence-detecting family of Example \ref{ex: algebraically central for schemes}, so Lemma \ref{lem: intersection cc} applies, i.e. the intersection of the supports of two compactly central localisations still corresponds to a compactly central localisation. By induction on $k$, $L_Z$ is compactly central.
\end{proof}

\printbibliography

\end{document}